\documentclass[reqno]{amsart}
\usepackage{amssymb}
\usepackage{amsmath}
\usepackage{amsthm}
\usepackage{mathtools}
\usepackage{graphicx}
\usepackage{comment}

\usepackage{amssymb, mathrsfs, amsfonts, amsmath}

\usepackage{hyperref}
\hypersetup{colorlinks}
\usepackage{tikz}
\usetikzlibrary{arrows}

\usepackage[english]{babel}

\usepackage[margin=1.4in, centering]{geometry}

\DeclareSymbolFont{bbold}{U}{bbold}{m}{n}
\DeclareSymbolFontAlphabet{\mathbbold}{bbold}
\newtheorem{thm}{Theorem}[section]
\newtheorem{prop}[thm]{Proposition}
\newtheorem{lem}[thm]{Lemma}
\newtheorem{cor}[thm]{Corollary}
\theoremstyle{definition}

\theoremstyle{remark}
\newtheorem{rem}[thm]{Remark}
\newcommand{\BR}{\mathbb{R}}
\newcommand{\BZ}{\mathbb{Z}}

\newcommand{\supp}{\text{supp}}

\newcommand{\R}{\mathbb{R}}
\newcommand{\C}{\mathbb{C}}

\newcommand{\N}{\mathbb{N}}
\newcommand{\Z}{\mathbb{Z}}
\newcommand{\A}{\mathbb{A}}
\newcommand{\eps}{\varepsilon}

\def\Xint#1{\mathchoice
{\XXint\displaystyle\textstyle{#1}}%
{\XXint\textstyle\scriptstyle{#1}}%
{\XXint\scriptstyle\scriptscriptstyle{#1}}%
{\XXint\scriptscriptstyle\scriptscriptstyle{#1}}%
\!\int}
\def\XXint#1#2#3{{\setbox0=\hbox{$#1{#2#3}{\int}$ }
\vcenter{\hbox{$#2#3$ }}\kern-.6\wd0}}

\def\dashint{\Xint-}

\title[Sobolev smoothing for bilinear maximal operators with fractal dilation sets]{Sobolev smoothing estimates for bilinear maximal operators with fractal dilation sets}

\author{Tainara Borges, Benjamin Foster, and Yumeng Ou}

\address[T. Borges]{Department of Mathematics, Brown University, Providence, RI 02912}
\address[B. Foster]{Department of Mathematics, Stanford University, Stanford, CA 94305}
\address[Y. Ou]{Department of Mathematics, University of Pennsylvania, Philadelphia, PA 19104}

\begin{document}

\begin{abstract}
    Given a hypersurface $S\subset \R^{2d}$, we study the bilinear averaging operator that averages a pair of functions over $S$, as well as more general bilinear multipliers of limited decay and various maximal analogs. Of particular interest are bilinear maximal operators associated to a fractal dilation set $E\subset [1,2]$; in this case, the boundedness region of the maximal operator is associated to the geometry of the hypersurface and various notions of the dimension of the dilation set. In particular, we determine Sobolev smoothing estimates at the exponent $L^{2}\times L^{2}\rightarrow L^2$ using Fourier-analytic methods, which allow us to deduce additional $L^{p}$ improving bounds for the operators and sparse bounds and their weighted corollaries for the associated multi-scale maximal functions. We also extend the method to study analogues of these questions for the triangle averaging operator and biparameter averaging operators. In addition, some necessary conditions for boundedness of these operators are obtained.
\end{abstract}

\maketitle

\tableofcontents 

\section{Introduction}

Let $d\geq 1$. Given a bounded measurable function $m:\R^{2d}\rightarrow \C$, we consider the associated bilinear multiplier operator defined as 
\begin{equation}\label{eqn: Tm}
T_m(f,g)(x)=\int_{\R^{2d}}\hat{f}(\xi)\hat{g}(\eta)m(\xi,\eta)e^{2\pi i x\cdot (\xi+\eta)} \,d\xi d\eta,\quad x\in \R^d,
\end{equation}for all $f,\,g\in C^{\infty}_{0}(\R^d)$.

     We say a bilinear multiplier $m:\R^{2d}\rightarrow \C$ has $a$-admissible decay up to order $k$, $k\geq 0$, if $m\in C^{k}(\R^{2d})$ and there is $a>\frac{d}{2}$ such that 
     \begin{equation}
         |\partial^{\alpha}m(\xi, \eta)|\lesssim_{\alpha}\frac{1}{(1+|\xi|+|\eta|)^{a}},
     \end{equation}
      for all multi-indices $\alpha=(\alpha_1,\dots,\alpha_{2d})$ with $|\alpha|\leq k.$ 
      
      In the case where $m=\hat{\mu}$ for $\mu$ being a compactly supported finite measure on $\R^{2d}$ satisfying
\begin{equation}\label{admissiblemeasure}
   |\hat{\mu}(\xi,\eta)|\lesssim \frac{1}{(1+|\xi|+|\eta|)^{a}} 
\end{equation}
     with $a>\frac{d}{2}$, we will simply say that $\hat{\mu}$ is $a$-admissible.

As observed in \cite{rubiodefrancia}, if $\mu$ is a compactly supported finite measure on $\R^{2d}$ satisfying the decay estimate (\ref{admissiblemeasure}), then for all multi-index $\alpha\in \Z_{\geq 0}^{2d}$, one also has $|\partial^{\alpha}(\hat{\mu})(\xi,\eta)|\lesssim_{\alpha} (1+|\xi|+|\eta|)^{-a}$. Hence, if $\hat{\mu}$ is $a$-admissible of order $0$, it is automatically $a$-admissible up to order $k$ for any $k\in \N$.

Multipliers like $m$ defined above are sometimes referred to as multipliers of limited decay in literature. For instance, a class of such bilinear multipliers have been studied by Grafakos--He--Honz\'ik in \cite{GHH} where $L^2\times L^2\to L^1$ bound of its multi-scale maximal function is obtained for $a$-admissible multipliers up to order $\lfloor \frac{d}{2} \rfloor+2$ with $a>\frac{d}{2}+1$. These operators are natural bilinear analogues of a class of linear operators studied by Rubio de Francia \cite{rubiodefrancia} that includes the spherical maximal functions (whose boundedness have been obtained by Stein \cite{stein} and Bourgain \cite{bourgaind2}) as a special case.

In this paper, we study boundedness properties of single-scale maximal functions of such bilinear multiplier operators over a fractal dilation set, as well as their corresponding multi-scale maximal functions. In addition to $L^p$ bounds, we also obtain Sobolev norm bounds and sparse bounds, which in turn imply vector valued estimates and weighted norm inequalities with respect to the Muckenhoupt $A_p$ weights for such operators. We give the definition of these operators below after discussing a few motivating examples of the multiplier $m$.

One important class of bilinear multiplier operators with admissible multipliers as defined above is given by bilinear averaging operators over smooth compact surfaces. Let $S\subset \R^{2d}$ be a $(2d-1)$-dimensional compact smooth hypersurface without boundary such that $k$ of the $(2d-1)$ principal curvatures do not vanish, where $k>d$. Let $\mu_S$ be the normalized natural surface measure on $S$ and define the (scale $1$) averaging operator
\begin{equation}\label{defkcurvatures}
     \mathcal{A}_{\hat\mu_S,1}(f,g)(x)=\int_{S}f(x-y)g(x-z)\,d\mu_S(y,z).
\end{equation}

Then it is easy to see that $\mathcal{A}_{\hat\mu_S,1}$ is a bilinear multiplier operator $T_m$ associated to $m(\xi,\eta)=\hat{\mu}_S(\xi,\eta)$ as defined above. It is well known \cite{Littman63} that in that case all the derivatives of $\hat{\mu}_S$ satisfy 
\begin{equation}\label{decaysurfacemulitplier}
   |\partial^{\alpha}(\hat{\mu}_S)(\xi, \eta)|\leq C_{\alpha} \dfrac{1}{(1+|(\xi,\eta)|)^{k/2}},\quad \textnormal{ for all } |\alpha|\geq 0,
\end{equation}hence the multiplier $\hat\mu_S$ is $\frac{k}{2}$-admissable. Similarly, fix a scale $t\in [1,2]$, one can define the (scale $t$) averaging operator 
\begin{equation}\label{eqn: St}
\mathcal{A}_{\hat\mu_S, t}(f,g)(x)=\int_S f(x-ty)g(x-tz)\,d\mu_S(y,z),
\end{equation}which is obviously an operator of the form $T_{m_t}$ with $\frac{k}{2}$-admissible multiplier $m_t(\xi,\eta)=\hat\mu_S(t\xi, t\eta)$.

As a particular case, when $d\geq 2$ and $\mu_S=\sigma_{2d-1}$, the normalized spherical measure on the unit sphere $S^{2d-1}\subset\R^{2d}$, then for every $t>0$, $m_t(\xi,\eta)=\hat\sigma_{2d-1}(t\xi, t\eta)$ is an $a$-admissible multiplier with $a=\frac{2d-1}{2}>\frac{d}{2}$ and we introduce the notation
$$\mathcal{A}_t(f,g)(x):=\mathcal{A}_{\hat\sigma_{2d-1},t}(f,g)(x)=\int_{S^{2d-1}}f(x-ty)g(x-tz)\,d\sigma_{2d-1}(y,z),\quad x\in \R^d.$$ Here, $\mathcal{A}_t$ is usually referred to as the bilinear spherical averaging operator at scale $t>0$.

Another important bilinear averaging operator we will be interested in is the triangle averaging operator of radius $t>0$, namely,
\begin{equation}
    \mathcal{T}_{t}(f,g)(x)=\int_{\mathcal{I}} f(x-ty)g(x-tz)\,d\mu(y,z),
\end{equation}
where $\mu$ is the natural normalized surface measure on the submanifold of $\R^{2d}$ given by 
\begin{equation}\label{trianglemanifold}
    \mathcal{I}=\{(y,z)\in \R^{2d}\colon |y|=|z|=|y-z|=1 \},\quad d\geq 2.
\end{equation}
This situation requires a more delicate analysis, as the submanifold $\mathcal{I}$ is codimension 3 rather than codimension 1, and the Fourier transform of the surface measure has worse decay properties. For instance, it was proved in \cite{triangleaveraging} that for any multi-indices $\alpha,\beta$
\begin{equation}\label{decaytriangle}
|\partial_{\xi}^{\alpha}\partial_{\eta}^{\beta}\hat{\mu}(\xi,\eta)|\leq C_{\alpha,\beta} \left(1+\min\{|\xi|,|\eta|\}|\sin{\theta}|\right)^{-\frac{d-2}{2}} (1+|(\xi,\eta)|)^{-\frac{d-2}{2}},
\end{equation}
where $\theta$ denotes the angle between $\xi$ and $\eta$. If one tries to bound this independently of $\theta$, the best one can say is that the right hand-side of (\ref{decaytriangle}) is bounded above by $(1+|(\xi,\eta)|)^{-\frac{d-2}{2}}$. Since $\frac{d-2}{2}<\frac{d}{2}$, $\hat{\mu}$ does not fall in our class of admissible multipliers and our general theorems later do not apply to it. To get decay bounds for the single-scale triangle averaging operators we will need a more careful analysis than the one for the decay bounds for single-scale maximal operators associated to admissible multipliers since we have to take into account the dependence in $\theta$ in the decay of the multiplier. We resolve this issue by applying an additional angular decomposition of the multiplier, and extend our main boundedness results to such triangle averaging operators. 

With these examples in mind, we are ready to introduce one of the main objects of the study in this paper: a class of general single-scale bilinear maximal operators defined as
\begin{equation}
    \mathcal{A}_{m,E}(f,g)(x)=\sup_{t\in E}|T_{m_t}(f,g)(x)|,
\end{equation}
where $m_t(\xi,\eta)=m(t\xi,t\eta)$, and $E\subset [1,2]$ is a nonempty fractal set of dilation scales. Our first goal is to understand its boundedness region on the Lebesgue spaces $L^p\times L^q\to L^r$ and similarly establish bounds for the associated multi-scale maximal operator 

\begin{equation}\label{definitionmultiscaleE}
    \mathcal{M}_{m,E}(f,g)(x):=\sup_{l\in \Z}\sup_{t\in E}|T_{m_{t2^l}}(f,g)(x)|.
\end{equation}
Note that the operator $\mathcal{A}_{m,E}$ is local in $t$, it is hence not scaling invariant and is expected to satisfy Lebesgue space bounds outside the H\"older range. Such bounds are often referred to as the $L^p$ improving bounds. 

In the case that $E=\{t\}$, we oftentimes write $\mathcal{A}_{m,\{t\}}=\mathcal{A}_{m,t}$ for short, similarly $\mathcal{M}_{m,\{t\}}=\mathcal{M}_{m,t}$. In the case that $m=\hat\sigma_{2d-1}$, these become the more familiar bilinear spherical maximal functions and we adopt the notation $\mathcal{A}_E:=\mathcal{A}_{\hat\sigma_{2d-1},E}$ and $\mathcal{M}_E:=\mathcal{M}_{\hat\sigma_{2d-1},E}$.

When $m=\hat{\mu}$ for $\mu$ being the natural normalized surface measure on the triangle manifold $\mathcal{I}$, for the sake of clarity and being in line with existing literature, we denote these operators as
\begin{equation}\label{TEdefinition}
    \mathcal{T}_{E}(f,g)(x):=\sup_{t\in E} |\mathcal{T}_t(f,g)(x)|
\end{equation}
and 
\begin{equation}\label{multiscaleTEdef}
\mathcal{T}_E^{*}(f,g)(x):=\sup_{l\in \Z}\sup_{t\in E} |\mathcal{T}_{t2^l}(f,g)(x)|,
\end{equation} 
respectively.

 Obtaining the sharp bounds of $\mathcal{A}_{m,E}$ and $\mathcal{T}_E$ for general $E$ is a very challenging question, and even in some special cases, our knowledge is far from complete. For the particular case that $m=\hat{\sigma}_{2d-1}$ and $E=[1,2]$, the study of $L^p$ improving bounds outside the H\"older range for $\tilde{\mathcal{M}}:=\mathcal{A}_{\hat{\sigma}_{2d-1},[1,2]}$ was initiated in \cite{JL}, but the sharp boundedness region was only very recently obtained in \cite{bhojak2023sharp} for $d\geq 2$. In $d=1$, nothing outside the H\"older range is known (see \cite{MCZZ, dosidisramos} for some H\"older bounds of $\mathcal{A}_{\hat\sigma_1, [1,2]}$ followed as a consequence of the same bounds for the full bilinear circular maximal function). The only other case of $E$ that partial results are known is $E=\{1\}$, see \cite{IPS, borgesfoster, choleeshuin} and the references therein for some boundedness results when $m=\hat\sigma_{2d-1}$, $d\geq 2$ and \cite{SS, DOberlin, BS} for the case $m=\hat\sigma_1$, $d=1$. The recent estimates in \cite{choleeshuin} also extends to the case $m=\hat\mu_S$ for $S\subset \mathbb{R}^{2d}$ being a compact smooth surface with $k\geq d$ nonvanishing principal curvatures. Even in these special cases, none of the above listed results for $E=\{1\}$ is known to be sharp. 
 
 In this paper, we initiate the study of this question for general dilation set $E\subset [1,2]$ and obtain a partial description of the boundedness region. Even in the case $m=\hat\sigma_{2d-1}$, the sharp $L^{p}$ improving region for $\mathcal{A}_{m,E}$ is still unknown, similarly for $\mathcal{T}_E$. We use Sobolev smoothing estimates for $\mathcal{A}_{m,E}$ at $L^{2}\times L^{2}\rightarrow L^{2}$ as a key tool in proving the sufficient conditions. For the case $m=\hat{\sigma}_{2d-1}$, we will also provide some necessary conditions for the parameters $p,q,r$ for which $\mathcal{A}_{\hat{\sigma}_{2d-1},E}:L^{p}\times L^{q}\rightarrow L^{r}$ is bounded for a general $E\subset [1,2]$. These necessary conditions match the sufficient conditions in \cite{bhojak2023sharp} for the case $E=[1,2]$ up to the boundary.

 In the linear setting, similar questions have been investigated by a large group of authors over the decades. The model operators there are the spherical maximal function over a fractal dilation set, defined as
 \[
 \mathcal{A}_{E,linear}f(x):=\sup_{t\in E}\left|\int_{S^{d-1}} f(x-ty)\,d\sigma_{d-1}(y)\right|,
 \]and its corresponding multi-scale maximal function similarly as defined in (\ref{definitionmultiscaleE}). In the particular case that $E=\{1\}$, the multi-scale maximal function becomes the classical (linear) lacunary spherical maximal function whose boundedness on $L^p$, $1<p\leq \infty$ is well known (see \cite{calderon, coifmanweiss}). While in the case that $E=[1,2]$, its multi-scale maximal function is the famous (linear) spherical maximal function, whose $L^p$ bound was obtained by Stein \cite{stein} in $d\geq 3$ and Bourgain \cite{bourgaind2} in $d=2$. The sharp $L^p$ improving region for $\mathcal{A}_{\{1\}, linear}$ follows from a classical result in \cite{Littman73}, and the case $\mathcal{A}_{[1,2], linear}$ are obtained in \cite{Schlag, SchlagSogge}. For the general dilation set case, the systematic study of this class of linear operators began in the 1990s, see \cite{SeegerWainger, DuoanVargas, SeegerWainger2, TaoSeegerWright} for the $L^p\to L^p$ boundedness of $\mathcal{A}_{E, linear}$ and some refinement. It was only a few years ago that an almost complete understanding of its sharp $L^p$ improving region was finally achieved, see \cite{AHRS, RS}. As it turned out, this region is closely related to various notions of dimension of the dilation set $E$. In this paper, we explore in the bilinear setting a similar relation between the boundedness region of the averaging operator and the dimension of the dilation set $E$, which seems to be the first result of its kind for bilinear operators. An interesting discovery is that, in many results that we prove below, there is an explicit relation between the allowed class of multipliers, more precisely the allowed admissibility of the multipliers (given by the decay parameter $a$), and the dimension of the dilation set $E$. Such a phenomenon also showed up in the linear setting, see \cite{DuoanVargas}.

Our second goal is to establish sparse bounds for $\mathcal{M}_{m,E}$ and to link the $L^{p}$ improving bounds $L^p\times L^q \to L^r$ of $\mathcal{A}_{m,E}$ with $r> 1$ to sparse bounds of $\mathcal{M}_{m,E}$. To this end, a machinery that resolves this problem is successfully developed in this paper. More precisely, we will prove Sobolev smoothing estimates for the single-scale operator $\mathcal{A}_{m,E}$, getting as corollary continuity estimates for $\mathcal{A}_{m,E}$ at the exponent $(2,2,2)$ (to be made precise below, see Theorem \ref{sobolevE} and Corollary \ref{continuityA_E}), which, via previously developed methods (see \cite{sparsetriangle, BFOPZ}), will imply sparse bound $(p,q,r')$ for $\mathcal{M}_{m,E}$ at the same exponent $(p,q,r)$ where the Lebesgue space bound of $\mathcal{A}_{m,E}$ holds true. 

Sparse domination has been a vividly growing area in harmonic analysis since its birth less than ten years ago. It was originally developed for the study of the $A_2$ conjecture on sharp weighted norm inequalities for Calder\'on-Zygmund operators \cite{Lerner, Lacey0}, but has shown to be a powerful tool in deriving many more properties of different types of operators. As it is impossible to exhaust the list of important prior works in the theory, here we only mention a few that are most closely related to the subject matter of our article, and refer the interested reader to the great amount of references listed there for further results on sparse domination. In \cite{Lacey}, sharp sparse bounds for the linear full and lacunary spherical maximal functions were derived, as a consequence of continuity estimates of their corresponding single-scale operator $\mathcal{A}_{\{1\}, linear}$ and $\mathcal{A}_{[1,2], linear}$ respectively. This then motivated many followup sparse domination results for other Radon type operators in the linear setting, see for instance \cite{CO, BRS, CDPPV}, and in particular for the multi-scale maximal operator associated to $\mathcal{A}_{E, linear}$ obtained in \cite{AHRS}. 

In the bilinear setting, only a few sparse domination results are known for operators in or related to the class of operators that we are interested in here. More precisely, sparse bounds for the full and lacunary bilinear spherical maximal functions were obtained in \cite{BFOPZ, sparsetriangle}, see also \cite{RSS} for a similar operator but with a product structure, and \cite{sparsetriangle} for the triangle maximal operators. A common framework all these works follow is a machinery (adapted from \cite{Lacey} to the bilinear setting) that turns continuity estimates for single-scale operators to sparse bounds for multi-scale maximal functions. This is also the framework we will follow in this paper, which, as mentioned above, reduces the matter to obtaining satisfactory single-scale estimates. We refer to Section \ref{sec:notation} for the precise definitions of sparse family and sparse domination and to Theorem \ref{sparseboundME} below for the precise statement of one of our sparse domination theorems for $\mathcal{M}_{m,E}$. It is also well known that sparse bounds have many applications. For example, it implies immediately quantitative weighted norm inequalities (with respect to multilinear Muckenhoupt $A_p$ weights in our case, and include the Lebesgue space bounds as a special case), endpoint estimates, and vector valued estimates. We do not explore the sharp consequences in these directions, see for instance \cite{CDPO, RSS, sparsetriangle} for some of these applications.

The sparse bounds for $\mathcal{M}_{m,E}$ that we obtain have close connections and applications to the aforementioned first goal of the paper, i.e. obtaining Lebesgue space bounds for $\mathcal{A}_{m,E}$ and $\mathcal{M}_{m,E}$. As already briefly mentioned in the above, a key feature of our machinery is that it reduces the understanding of the sparse bounds for $\mathcal{M}_{m,E}$ to that of the sharp $L^p\times L^q\times L^r$ boundedness region of $\mathcal{A}_{m,E}$. Indeed, as mentioned above, a key step in our proof of the sparse bounds for $\mathcal{M}_{m,E}$ is obtaining continuity estimates for $\mathcal{A}_{m,E}$. Once such a continuity estimate (at some exponent) is obtained, the framework becomes very flexible: any improvement in the known boundedness region of $\mathcal{A}_{m,E}$ would translate to improvement in known continuity estimates by using multilinear interpolation, hence leads to improvement in the range of known sparse bounds as well as the implied Lebesgue or weighted norm estimates. Our paper is the first effort to prove continuity estimates for multilinear operators associate with a general dilation set $E$. In the special case  $m=\hat{\sigma}_{2d-1}$ and $E=\{1\}$ or $E=[1,2]$, continuity estimates have been studied before simultaneously in \cite{sparsetriangle} and \cite{BFOPZ}. In \cite{sparsetriangle} they prove continuity estimates for $E=\{1\}$ and $E=[1,2]$ for all except some low dimensional cases, namely for $d\geq 2$ for $E=\{1\}$ and for $d\geq 4$ for $E=[1,2]$. In \cite{BFOPZ}, continuity estimates are obtained for $\mathcal{M}_{\hat\sigma_{2d-1},[1,2]}$ for any dimension $d\geq 2$  and for $\mathcal{M}_{\hat\sigma_{2d-1},\{1\}}$ in $d=1$. The proof of continuity estimates in \cite{sparsetriangle}  uses the $L^{2}\times L^{2}\rightarrow L^{2}$ boundedness criteria developed in \cite{GHS} and it relies heavily on the decay of the multiplier $\hat{\sigma}_{2d-1}$, but only leads to continuity estimates for $\tilde{\mathcal{M}}$ in sufficient large dimensions $d\geq 4$ (also see Remark \ref{comparisonwithpalsson} for more details on how that strategy compares to the one in this paper). The continuity estimates for $\mathcal{M}_{\hat\sigma_{2d-1},[1,2]}$ in \cite{BFOPZ} do not use the decay of $\hat{\sigma}_{2d-1}$ as directly, but in turn rely on the slicing technique developed in \cite{JL} for the sphere so it seems harder to generalize to the setting that we are aiming for in this work. 

Moreover, by applying our sparse bounds and exploiting some other techniques such as Sobolev embedding, we obtain a wide range of Lebesgue bounds for $\mathcal{M}_{m,E}$, which includes the particular case $L^2\times L^2 \to L^1$ (see Theorem \ref{Lebesgueboundsmultiscale}). Such bounds in such generality seem to be the first of their kind, except for the special cases $\mathcal{M}_{\hat\sigma_{2d-1},\{1\}}$ and $\mathcal{M}_{\hat\sigma_{2d-1}, [1,2]}$, for which better estimates were already known (\cite{JL, MCZZ, borgesfoster}). In the bilinear setting, obtaining the $L^2\times L^2\to L^1$ (and other Lebesgue bounds) is in general much more challenging compared to obtaining the natural $L^2\to L^2$ bound for linear operators, which oftentimes follows from Plancherel. They are also of intrinsic interest, since they imply among many other things almost everywhere convergence results. For example, suppose $m=\hat\mu_S$ and $E=[1,2]$, then our Lebesgue bound of $\mathcal{M}_{\hat\mu_S,[1,2]}$ (included as a special case of Theorem \ref{Lebesgueboundsmultiscale} below) implies immediately the convergence of $\mathcal{A}_{\hat\mu_S, t}(f,g)(x)$ to $f(x)g(x)$, for almost all $x$ as $t\to 0$, for $f,g\in L^2$. In this sense, our work naturally extends the line of investigation (e.g. \cite{GHH}) of Lebesgue bounds of bilinear analogues of classical linear operators of averaging type studied by Stein \cite{stein}, Bourgain \cite{bourgaind2}, Rubio de Francia \cite{rubiodefrancia}, and others (see also \cite{DuoanVargas} for an extension to the fractal dilation set setting of \cite{rubiodefrancia} in the linear case). And our boundedness results, which improve previously best known estimates in many directions (see Remark \ref{rmk: Lebesgue compare} below for a more detailed discussion), further confirms the strength and versatility of the sparse domination technique.

It turns out that, for the argument deducing sparse bounds from the continuity estimates to work, we will need the extra assumption that $m$ is given by the Fourier transform of a compactly supported finite measure $\mu$ in $\R^{2d}$. In that case the bilinear multiplier operator $T_{m_t}$ becomes a bilinear average over a compact set, namely
\begin{equation}
    T_{m_t}(f,g)(x)=\int f(x-ty)g(x-tz)\,d\mu(y,z).
\end{equation}The motivating example $\mathcal{A}_{\hat\mu_S,t}$ as defined in (\ref{eqn: St}) obviously falls into this category, but the theory is more general than that as $\mu$ is not necessarily given as a surface measure at all. 

We also extend the continuity estimates and sparse bounds to the triangle averaging operator $\mathcal{T}_E$ and its maximal function $\mathcal{T}^*_E$, partially resolving an issue showed up in \cite{sparsetriangle} and improving the previously best known results in this direction, see Remark \ref{rmk: PS} for more details.

In addition, we extend the aforementioned estimates regarding $\mathcal{A}_{m,E}$ and $\mathcal{M}_{m,E}$ to their slightly larger biparameter analogues, which were even less understood in literature before. We obtain interesting sufficient and necessary conditions for their boundedness regions. Given two dilation sets $E_1, E_2\subset [1,2]$, define 
\[
\mathcal{A}_{m,E_1,E_2}^{(2)}(f,g)(x):=\sup_{t_1\in E_1,t_2\in E_2}|T_{m_{t_1,t_2}}(f,g)(x)|,
 \]where 
 \[
  T_{m_{t_1,t_2}}(f,g)(x)=\int_{\R^{2d}}\hat{f}(\xi)\hat{g}(\eta)m(t_1\xi,t_2\eta)e^{2\pi i x\cdot(\xi+\eta)}\,d\xi d\eta.
 \]We are interested in the boundedness of $\mathcal{A}^{(2)}_{m,E_1,E_2}$ and its multi-scale maximal function defined as $\sup_{l\in \mathbb{Z}}\sup_{t_1\in E_1,t_2\in E_2}|T_{m_{2^l t_1, 2^l t_2}}|$. 
 
 In the case that $E_1=E_2=[1,2]$ and $m=\hat\sigma_{2d-1}$, the slicing technique developed in \cite{JL} applies equally well to these biparameter operators and implies satisfying $L^p$ improving bounds (as well as continuity estimates and sparse bounds, if one follows a similar strategy as in \cite{BFOPZ}). However, once one considers the general case of $E_1, E_2$, such slicing becomes significantly less effective even when $m=\hat\sigma_{2d-1}$, making the study of these biparameter bilinear operators very different from their one-parameter counterparts. We will show how our one-parameter arguments can be extended to the biparameter setting in Section \ref{sec: biparameter}. 
 
 Moreover, we obtain sparse domination for the multi-scale maximal function of $\mathcal{A}^{(2)}_{m, E_1, E_2}$. This may come as a surprise since it is well known that sparse domination techniques usually do not work well in the multiparameter setting (see for instance \cite{BCOR} for a counterexample to sparse bound of the multiparameter strong maximal function). However, our situation here is in fact slightly different. Since both $E_1, E_2$ are localized at roughly the same scale, the multi-scale maximal function in some sense only exhibits one-parameter behavior when it comes to sparse domination. Hence, the key step in the proof of the sparse bound is the continuity estimate for the single-scale operator $\mathcal{A}^{(2)}_{m,E_1, E_2}$. We defer the statements of these results to Section \ref{sec: biparameter}.

\subsection*{Main results of the article}We will be looking for Sobolev smoothing bounds with decay for the pieces of the single-scale operator $\mathcal{A}_{m,E}$, which in turn will lead to continuity estimates at the exponent $(2,2,2)$, that is, bounds of the form 
$$\|\mathcal{A}_{m,E}(f-\tau_{h_1}f,g-\tau_{h_2}g)\|_{L^{2}}\lesssim |h_1|^{\gamma_1}|h_2|^{\gamma_2}\|f\|_{L^2}\|g\|_{L^2},$$
where $|h_1|,\,|h_2|<1$, $\gamma_1, \gamma_2>0$ and $\tau_h(f)(x)=f(x-h)$, whose corollaries include sparse domination results for  $\mathcal{M}_{m,E}$. 

Sobolev smoothing bounds for bilinear single-scale Radon type operators are not only interesting in their own right but also are particularly useful in deriving continuity estimates. In our previous work \cite{BFOPZ}, a different Sobolev bound $H^{-\delta}\times L^2 \to L^1$ was obtained for $\mathcal{A}_{\hat\sigma_1, 1}$ in the $d=1$ case, which played a key role in obtaining continuity estimates of it and then sparse bound for its corresponding multi-scale maximal operator there. However, the proof of that Sobolev bound required strong machinery (trilinear smoothing estimate from \cite{MCZZ}) and the passage from it to continuity estimates was also quite involved. In contrast, the Sobolev estimates obtained in this paper (for example Theorem \ref{sobolevE} and \ref{sobolevtrianglethm}) are of the form $H^{-s_1}\times H^{-s_2}\to L^2$. An advantage of this type of results is that it enables us to have a much simplified derivation of the continuity estimates and allows for general dilation sets $E$ and general bilinear multipliers (with admissible decay) in the theory. For comparison, the slicing method applied in \cite{BFOPZ} replies heavily on the geometry of the sphere and the dilation set $[1,2]$, and does not seem to yield a natural extension to the more general situations.

Our first theorem concerns the Sobolev norm bound of the single-scale maximal operator $\mathcal{A}_{m,E}$ associated to a fractal subset $E\subset [1,2]$. Here and throughout this paper for $s\geq 0$, $H^{-s}$ consists of functions $f:\R^d\rightarrow \C$ for which $\|f\|_{H^{-s}}:=\|(1+|\xi|^2)^{-s/2}\hat{f}(\xi)\|_{L^2}<\infty$.

\begin{thm}[Sobolev smoothing bounds for $\mathcal{A}_{m,E}$]\label{sobolevE} Let $d\geq 1$ and $E\subset [1,2]$ with upper Minkowski dimension $\text{dim}_{M}(E)=\beta$. Let $m$ be an $a$-admissible bilinear multiplier up to order $1$, and assume $2a>d+\beta$. Given $s_1,s_2\geq 0$, with $s_1+s_2<\frac{2a-d-\beta}{2}$, then
\begin{equation}
      \|\mathcal{A}_{m,E}(f,g)\|_{L^2}\lesssim \|f\|_{H^{-s_1}}\|g\|_{H^{-s_2}}.
\end{equation}
In particular, $\mathcal{A}_{m,E}:L^{2}\times L^{2}\rightarrow L^{2}$.
\end{thm}
\begin{rem}
    For the case $E=\{1\}$ one can actually include the endpoint $s_1+s_2=\frac{2a-d-\beta}{2}$ and assume the weaker condition that $m$ is an admissible multiplier of order $0$. See Proposition \ref{sobolev}.
\end{rem}

In the case of the single-scale bilinear maximal operators associated to triangle averages $\mathcal{T}_E$, the Sobolev smoothing estimates for $\mathcal{T}_E$ state as follows.

\begin{thm}[Sobolev smoothing bounds for $\mathcal{T}_E$] \label{sobolevtrianglethm}
Let $E\subset [1,2]$ with upper Minkowski dimension  $\text{dim}_M(E)=\beta$, and assume that $d>4+\beta$. Given $s_1,s_2\geq 0$, with $s_1+s_2<\frac{d-4-\beta}{2}$, then
    \begin{equation}
        \|\mathcal{T}_E(f,g)\|_2\lesssim \|f\|_{H^{-s_1}}\|g\|_{H^{-s_2}}.
    \end{equation}
\end{thm}

\begin{rem}
    Observe that in the result above since $d$ is integer, the restriction $d>4+\beta$ becomes $d\geq 5$ for $\beta\in [0,1)$, and $d\geq 6$ for $\beta=1$. 
\end{rem}

The Sobolev smoothing theorem for $\mathcal{A}_{m,E}$ (Theorem \ref{sobolevE}) easily implies continuity estimates for $\mathcal{A}_{m,E}$ which we state and prove in Corollary \ref{continuityA_E}. Similarly, Theorem \ref{sobolevtrianglethm} implies continuity estimates for $\mathcal{T}_E$ (see Corollary \ref{continuityTE}). For the case $m=\hat{\mu}_S$ where $S$ is a $(2d-1)$-dimensional compact smooth hypersurface without boundary in $\R^{2d}$ with $k$ non-vanishing principal curvatures, Theorem \ref{sobolevE} combined with the decay bounds in inequality (\ref{decaysurfacemulitplier}) immediately imply the following corollary.

\begin{cor}[Consequences for the surface averaging operator $\mathcal{A}_{\hat\mu_S,E}$]\label{corollaryforsurfacesgenE} Let $d\geq 2$ and $E\subset [1,2]$ with upper Minkowski dimension $\text{dim}_{M}(E)=\beta$. 
\begin{enumerate}
    \item For  $m(\xi,\eta)=\hat{\sigma}_{2d-1}$, if $d>1+\beta$ (which is always the case for $d\geq 3$), then, one has that
    \begin{equation}
\|\mathcal{A}_{E}(f,g)\|_{L^2}\lesssim \|f\|_{H^{-s_1}}\|g\|_{H^{-s_2}},
    \end{equation}
     for any $s_1, s_2\geq 0$ with $s_1+s_2<\frac{d-1-\beta}{2}$. Moreover, continuity estimates like (\ref{continuityoneinputforE}) and (\ref{continuityinbothforE}) hold true for any $d\geq 2$ and $\beta\in [0,1]$, except possibly for $(d,\beta)=(2,1)$.
    \item  For a $(2d-1)$-dimensional compact smooth surface $S$ in $\R^{2d}$ without boundary such that $k$ of the $(2d-1)$ principal curvatures do not vanish, and $k>d+\beta$, one has that
\begin{equation}
\|\mathcal{A}_{\hat\mu_S,E}(f,g)\|_{L^2}\lesssim \|f\|_{H^{-s_1}}\|g\|_{H^{-s_2}},
\end{equation} for any $s_1,s_2\geq 0$ with $s_1+s_2<\frac{k-d-\beta}{2}$.
Moreover, continuity estimates like (\ref{continuityoneinputforE}) and (\ref{continuityinbothforE}) hold true for any $d\geq 2$ and $k>d+\beta$.
\end{enumerate}
   
\end{cor}
\begin{rem}
    Observe that in the second part of Corollary \ref{corollaryforsurfacesgenE}, since $k$ is an integer the assumption $k>d+\beta$ can be described as $k\geq d+1$ for $\beta\in[0,1)$, and $k\geq d+2$ if $\beta=1$.
\end{rem}

In the particular case where $m=\hat{\mu}$ where $\mu$ is a compactly supported finite measure in $\R^{2d}$, the continuity estimates given by Corollary \ref{continuityA_E} will have sparse bounds consequences for multi-scale bilinear maximal functions $\mathcal{M}_{\hat{\mu},E}$ which we state below. 

For such a given measure $\mu$, consider the $L^{p}$ improving boundedness region of $\mathcal{A}_{\hat{\mu},E}$, namely
\begin{equation}\label{regionAE}
\begin{split}
    \mathcal{R}(\mu,E):=\{(1/p,1/q,1/r)&\colon 1\leq p,q\leq \infty, r>0,\, 1/p+1/q\geq 1/r\\
    &\text{ and }\mathcal{A}_{\hat{\mu},E}:L^{p}\times L^q\rightarrow L^r\text{ is bounded}\}.  
\end{split}
\end{equation}

\begin{cor}[Sparse bounds for $\mathcal{M}_{\hat{\mu},E}$]\label{sparseboundME}
Let $E\subset[1,2]$ with upper Minkowski dimension $\beta$ and let $\mu$ be a compactly supported finite Borel measure in $\R^{2d}$. Suppose $\hat{\mu}$ is an $a$-admissible bilinear multiplier, with $2a>d+\beta$. Then, for any $(1/p,1/q,1/r)\in \text{int} (\mathcal{R}(\mu,E))$, with $r>1$ and $p,q\leq r$, one has sparse domination for $\mathcal{M}_{\hat{\mu},E}$ with parameters $(p,q,r')$. That is, for any $f,g,h\in C^{\infty}_{0}(\R^d)$, there exists a sparse family $\mathcal{S}$ such that 
\begin{equation}
    \begin{split}
|\langle\mathcal{M}_{\hat{\mu},E}(f,g),h\rangle|\lesssim \sum_{Q\in \mathcal{S}} \langle f \rangle_{Q,p} \langle g\rangle_{Q,q}  \langle h\rangle_{Q,r'} |Q|.
    \end{split}
\end{equation}
\end{cor}

These sparse bounds have interesting weighted bounds consequences. In particular, we can get the following Lebesgue bounds for  $\mathcal{M}_{\hat{\mu},E}$, by combining these consequences in the Lebesgue measure case and some further techniques.

\begin{thm}[Lebesgue bounds for $\mathcal{M}_{\hat{\mu},E}$]\label{Lebesgueboundsmultiscale} Let $E\subset [1,2]$ be a subset with upper Minkowski dimension $\beta$ and $\hat{\mu}$, the Fourier transform of a compactly supported finite Borel measure in $\mathbb{R}^{2d}$ that is $a$-admissible. Suppose  $2a>d+\beta$. Let $R(\beta,a)\subset [0,1]^{2}$ be the interior of the convex closure of the points 
$$(1,0),\,(0,0),\,(0,1),\,\left(\frac{1}{2},\frac{2a-\beta}{2d}\right),\text{ and }\left(\frac{2a-\beta}{2d},\frac{1}{2}\right).$$ 
Let $L_1$ be the open segment connecting $(1,0)$ to $(0,0)$, and $L_{2}$ be the open segment connecting $(0,0)$ to $(0,1)$. Then for any $(1/p,1/q)\in \mathcal{R}(\beta,a)\cup L_1\cup L_2\cup \{(0,0)\}$, we have
$$\mathcal{M}_{\hat{\mu},E}:L^{p}\times L^{q}\rightarrow L^{r}$$
for $r$ given by the H\"older relation $1/r=1/p+1/q$ (see Figure \ref{pictureforMEregion}).
\end{thm}

\begin{figure}[h]\label{pictureforMEregion}
\begin{center}
         \scalebox{1}{
\begin{tikzpicture}
\draw (-0.3,3.8) node {$\frac{1}{q}$};
\draw (3.8,0) node {$\frac{1}{p}$};
\draw (3,-0.4) node {$1$};
\draw (-0.4,3) node {$1$};
\draw (1.5,-0.5) node{$\frac{1}{2}$};
\draw (-0.5,1.5) node{$\frac{1}{2}$};
\draw (-0.5,2.5) node{$\frac{2a-\beta}{2d}$};
\draw (2.5,-0.5) node{$\frac{2a-\beta}{2d}$};
\draw[purple] (3.4,2.3) node {\small{$\frac{1}{p}+\frac{1}{q}=\frac{1}{2}+\frac{2a-\beta}{2d}$}};

\filldraw[purple!25!white] (1.5,2.5)--(2.5,1.5)--(3,0)--(0,0)--(0,3)--(1.5,2.5);
\draw[purple, dashed, line width=1pt] (0,3)--(1.5,2.5)--(2.5,1.5)--(3,0);

%drawing axis
\draw[->,line width=1pt] (-0.2,0)--(3.5,0);
\draw[->,line width=1pt] (0,-0.2)--(0,3.5); 
\draw[purple, line width=1pt] (0,3)--(0,0)--(3,0);

%drawing small segments for labels
\draw[-,line width=0.75pt] (3,-0.1)--(3,0.1);
\draw[-,line width=0.75pt] (-0.1,3)--(0.1,3);
\draw[-,line width=0.75pt] (1.5,-0.1)--(1.5,0.1);
\draw[-,line width=0.75pt] (-0.1,1.5)--(0.1,1.5);
\draw[-,line width=0.75pt] (-0.1,2.5)--(0.1,2.5);
\draw[-,line width=0.75pt] (2.5,-0.1)--(2.5,0.1);
%more dashed lines
\draw[dashed] (3,0)--(3,3)--(0,3);
\draw[dashed,purple] (2.5,0)--(2.5,1.5)--(0,1.5);
\draw[dashed,purple] (0,2.5)--(1.5,2.5)--(1.5,0);
%drawing points
\draw[purple,line width=0.75pt] (2.5,1.5) circle (2pt);
\draw[purple,line width=0.75pt] (1.5,2.5) circle (2pt);
\draw[purple,line width=0.75pt] (3,0) circle (2pt);
\draw[purple,line width=0.75pt] (0,3) circle (2pt);
\fill[white] (2.5,1.5) circle (2pt);
\fill[white] (1.5,2.5) circle (2pt);
\fill[white] (3,0) circle (2pt);
\fill[white] (0,3) circle (2pt);
\fill[purple] (0,0) circle (2pt);
\end{tikzpicture}}
\end{center}
\caption{Region of pairs $(1/p,1/q)$ for which Theorem \ref{Lebesgueboundsmultiscale} guarantees $\|\mathcal{M}_{\hat{\mu},E}\|_{L^{p}\times L^{q}\rightarrow L^{r}}<\infty$ for $1/r=1/p+1/q$.}
\end{figure}
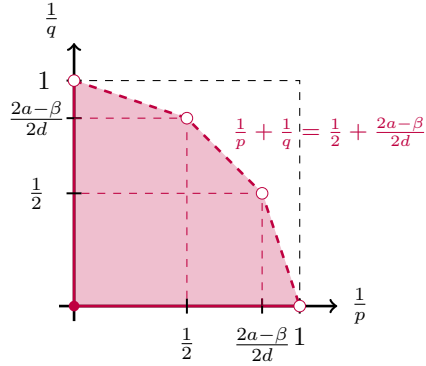

\begin{rem}\label{rmk: Lebesgue compare}
Even though the study of Lebesgue bounds usually precedes that of the sparse bounds for many operators, in our case, Theorem \ref{Lebesgueboundsmultiscale} in fact seems to be the first in literature about Lebesgue bounds of such bilinear averaging operators for general dilation set $E$ and improves many previously best known results for special cases of $\mathcal{M}_{\hat\mu,E}$. For example, in the case of $\mathcal{M}_{\hat\mu, [1,2]}$, Theorem \ref{Lebesgueboundsmultiscale} implies that it maps $L^2\times L^2\to L^1$ as long as $2a>d+1$. This improves \cite[Theorem 1.1]{GHH}, where the same bound is proved under the stronger assumption $2a>d+2$. (Note that the result in \cite{GHH} does have the advantage that it applies to general multipliers $m$ rather than only $m=\hat\mu$. Hence our bound is not fully comparable to theirs. Similar pair of results, where one applies to more general multipliers while the other has stronger bound, exist in the linear setting as well, see \cite{rubiodefrancia}.) 

For the case where $\mu=\mu_S$, the normalized surface measure in a smooth compact surface with $k$ non-vanishing principal curvatures, one can check that 
\begin{itemize}
    \item Theorem \ref{Lebesgueboundsmultiscale} implies $$\mathcal{M}_{\hat\mu_S,[1,2]}(f,g)(x)=\sup_{t>0}\left|\int_{S}f(x-ty)g(x-tz)\,d\mu_S(y,z)\right|$$
    is bounded from $L^{2}\times L^{2}\rightarrow L^{1}$, for $k\geq d+2$. This improves a result from \cite[Theorem 2]{CGHHS} where they proved $\mathcal{M}_{\hat\mu_S,[1,2]}:L^{2}\times L^{2}\rightarrow L^{1}$ for $k\geq d+3$.
    \item For $E\subset [1,2]$ with upper Minkowski dimension $\beta\in [0,1)$, as long as $k\geq d+1$, one has $\mathcal{M}_{\hat\mu_S,E}:L^{2}\times L^{2}\rightarrow L^{1}$. In the particular case $E=\{1\}$, this recovers the result originally proved in \cite{choleeshuin}.  
\end{itemize}
    
\end{rem}

We defer the statements of the necessary conditions for boundedness of $\mathcal{A}_{m,E}$, the Sobolev smoothing estimates, continuity estimates, and sparse domination for the biparameter analogues of $\mathcal{A}_{m,E},\mathcal{M}_{m,E}$ to later sections.

\subsection*{Outline of the article}
In Section \ref{sec:notation}, we collect various fundamental facts about Minkowski and Assouad dimension, as these are the right notion of size for the dilation set in this setting. We also recall the definition of sparse domination of a bilinear operator. We start Section \ref{section222} by studying Sobolev smoothing estimates for $\mathcal{A}_{m,E}$, single-scale bilinear Fourier multipliers with admissible decay that are associated with a general fractal dilation set $E\subset [1,2]$, giving a proof to Theorem \ref{sobolevE}. In our setting, the amount of smoothing depends on the geometry of the hypersurface, since having more non-vanishing principal curvatures gives better decay for the Fourier transform of the surface measure. We then extend this to the case of the single-scale triangle operator $\mathcal{T}_E$, proving Theorem \ref{sobolevtrianglethm}. As corollary, these yield continuity estimates for $\mathcal{A}_{m,E}$ and $\mathcal{T}_E$ which will later be a key ingredient in proving sparse bounds for their associated multi-scale operators. In Section \ref{sec: biparameter},  we make a digression to the even more general biparameter setting, where each argument of the operator has its own associated dilation set. In Section \ref{section:sparsebounds}, we interpolate our previous continuity estimates with known bounds to get a larger range of continuity estimates, allowing us to deduce sparse bounds (such as Corollary \ref{sparseboundME}) for $\mathcal{M}_{\hat{\mu},E}$ and its biparameter analogue, and the triangle maximal operator $\mathcal{T}^*_E$ by appealing to standard techniques. In Section \ref{sec:somelebesguebounds}, we explore various Lebesgue bounds including Theorem \ref{Lebesgueboundsmultiscale} for single-scale and multi-scale maximal operators; many arguments in this section are robust enough to also deduce bounds for the biparameter variants, as well. Finally, in Section \ref{necessaryEsinglescale}, we consider several examples to get necessary conditions for our single-scale maximal operators. 

\subsection*{Acknowledgements}
Y.O. is supported by NSF DMS-2142221 and NSF DMS-2055008. The authors would like to thank Professor Jill Pipher for motivating some of the questions explored in this paper and for many helpful math discussions.

\section{Some notation and preliminaries}\label{sec:notation}

\subsection{Notions of dimension for fractal sets}
Given $E\subset [1,2]$, we denote by $N(E,\delta)$ the minimal number of closed intervals of size $\delta$ that one needs to cover the set $E$. The upper Minkowski dimension of $E$ is given by 
\begin{equation}\label{def:upperminkowski}
    \text{dim}_{M}(E)=\limsup_{\delta\rightarrow 0} \frac{\log(N(E,\delta))}{\log(1/\delta)}.
\end{equation}
Equivalently, if the upper Minkowski dimension of $E$ is $\text{dim}_{M}(E)=\beta$, then $\beta$ is the smallest nonnegative number for which it holds that 
$$N(E,\delta)\leq C_{\eps} \delta^{-(\beta+\eps)},\,\text{ for all }\delta<1\text{ and }\eps>0.$$
Note that to show $\dim_M(E)\le \dim_M(E')$, it is enough to show that there exists a universal constant $C>0$ such that $N(E,\delta)\le CN(E',\delta)$ holds. In our special setting, we can establish some lemmas about how certain algebraic operations affect the upper Minkowski dimension.

\begin{lem}\label{lem:alg of mink dim}
    Let $E\subset [1/2,2]$. Then the upper Minkowski dimension of $E,E^2,\sqrt{E}$ and $1/E$ all agree, where $f(E)$ for a function $f$ denotes
    \[
    f(E)=\{f(t):t\in E\}.
    \]
\end{lem}
\begin{proof}
    These all follow immediately from the invariance of upper Minkowski dimension under bi-Lipschitz maps.
\end{proof}
Understanding how the upper Minkowski dimension can change is more complicated for binary operations on sets. We do have the following well-known bounds for a Minkowski sum of two sets.
\begin{lem}\label{lem:mink sumset}
    Let $E_1,E_2$ be bounded subsets of $\R$. Then
    \[
    \max\{\dim_M(E_1),\dim_M(E_2)\}\le \dim_M(E_1+E_2)\le \dim_M(E_1)+\dim_M(E_2).
    \]
\end{lem}
\begin{proof}
The lower bound is trivial since for any fixed $t_i\in E_i$, we have that $t_1+E_2\subset E_1+E_2$ and $\dim_M(t_1+E_2)=\dim_M(E_2)$ by translation invariance, so $\dim_M(E_2)\le\dim_M(E_1+E_2)$ and then argue symmetrically.
The upper bound can be seen by taking $\delta$-covers of $E_1,E_2$ of the form $\{I_{\delta}(a^{(k)}_i)\}_{i=1}^{N(E_k,\delta)}$ for $k=1,2$ where $I_{\delta}(a)$ stands for a closed interval of length $\delta$ centered at $a$ and getting a cover of $E_1+E_2$ via intervals of the form $\{I_{2\delta}(a^{(1)}_{i_1}+a^{(2)}_{i_2})\}_{i_k=1}^{N(E_k,\delta)}$. Since Euclidean space is doubling, we can cover this by a collection of $\delta$-intervals with cardinality at most $CN(E_1,\delta)N(E_2,\delta)$.
\end{proof}
The equality case for the upper bound is achieved when taking (for instance) $E_1=\{1-n^{-1}:n\in\mathbb{N}\}$ and $E_2=\{1-n^{-2}:n\in\mathbb{N}\}$.

There is another notion of dimension that is useful in the theory of the linear spherical maximal operators (see \cite{AHRS, RS}). The upper Assouad dimension of $E$ is given by
\begin{equation}\label{def:upperassouad}
\dim_A(E)=\inf_{\delta_0>0} \limsup_{\delta\rightarrow 0}\sup_{I:|I|\in(\delta,\delta_0)}\frac{\log(N(E\cap I,\delta))}{\log(|I|/\delta)},
\end{equation}
where $I$ is a subinterval of $[1,2]$. Equivalently, we say the upper Assouad dimension is $\gamma$ if $\gamma$ is the smallest nonnegative number such that there exists $\delta_0>0$ and $C_\epsilon>0$ for all $\eps>0$ such that for any $\delta<\delta_0$ and interval $I$ whose length satisfies $\delta<|I|<\delta_0$, we have
\[
N(E\cap I,\delta)\le C_{\epsilon}(\delta/|I|)^{-\gamma-\epsilon}.
\]
In fact, more generally, one can study the Assouad spectrum $\dim_{A,\theta}(E)$ for $0\le \theta<1$ which is given by
    \begin{equation}\label{def:upperassouadspec}
\dim_{A,\theta}(E)=\limsup_{\delta\rightarrow 0}\sup_{I:|I|=\delta^\theta}\frac{\log(N(E\cap I,\delta))}{\log(|I|/\delta)}.
\end{equation}
We direct the reader to the paper \cite{AHRS, RS}, which discusses all of these concepts and their role in understanding linear spherical maximal functions with fractal dilation sets. We will make use of the notion of an \emph{Assouad regular set}, which is a set $E$ where $\dim_{A,\theta} E=\dim_{A} E$ for all $1>\theta>1-\frac{\dim_M E}{\dim_A E}$.

In \cite[Theorem 2]{AHRS} they show that if $E\subset [1,2]$ is Assouad regular with $\text{dim}_{M}(E)=\beta$ and $\text{dim}_{A}(E)=\gamma$, then for 
$${A}_E(f)(x):=\sup_{t\in E}\left|\int f(x-ty)d\sigma (y)\right|,$$ 
a necessary condition (and also sufficient up to some boundary pieces as shown in \cite{AHRS}, \cite{RS}) for the boundedness of $A_E:L^{p}\rightarrow L^{r}$ is that $(1/p,1/r)\in \mathcal{Q}(\beta,\gamma)$, the closed convex hull of the points $Q_1=(0,0),\,Q_2(\beta)=(\frac{d-1}{d-1+\beta},\frac{d-1}{d-1+\beta}),\,Q_3(\beta)=(\frac{d-\beta}{d-\beta+1},\frac{1}{d-\beta+1})$ and $Q_{4}(\gamma)=(\frac{d(d-1)}{d^2+2\gamma-1},\frac{d-1}{d^2+2\gamma-1})$. Such region has the following precise description:
\begin{equation}
     \mathcal{Q}(\beta,\gamma)=\left\{\left(\frac{1}{p},\frac{1}{r}\right)\in[0,1]^2\text{ such that }\frac{1}{r}\leq \frac{1}{p}\leq m_{linear}(d,r,\beta,\gamma)\right\}
\end{equation}
for 
$$m_{linear}(d,r,\beta,\gamma)=\text{min}\left\{\frac{d-1}{d}+\frac{1-\beta}{dr},\frac{\beta(d-1)}{\beta(d-1)+2\gamma}+\frac{(d-\beta)2\gamma-(d-1)\beta}{\beta(d-1)+2\gamma}\frac{1}{r},\frac{d}{r}\right\}.$$
In particular, for Assouad regular sets there is an interesting connection between the necessary conditions we get in Proposition \ref{allnecessaryAE} with the sharp (up to boundary) boundedness region of $A_E$ in the linear case.

\subsection{Sparse families and sparse bounds}
Let $\mathcal{S}$ denote a (possibly infinite) collection of cubes in Euclidean space. We say that it is $\eta$-sparse (for some $0<\eta\leq 1$) if for each $Q\in\mathcal{S}$ there exists a subset $E_Q$ such that $|E_Q|\geq\eta |Q|$ and the $E_Q$'s are pairwise disjoint. We say that a bilinear or bisublinear operator $\mathcal{B}$ has a $(p,q,r)$ trilinear sparse domination if for all functions $f,g,h\in C^{\infty}_0(\R^d)$, there exists an $\eta$-sparse family $\mathcal{S}$ such that
\begin{equation}
    |\langle \mathcal{B}(f,g),h\rangle| \lesssim \sum_{Q\in \mathcal{S}} |Q| \langle f\rangle_{Q,p}\langle g\rangle_{Q,q}\langle h\rangle_{Q,r}.
\end{equation}
On the left side of the equation, the brackets denote the standard $L^2$ inner product, and on the right side, we are using the notation
\[
\langle f\rangle_{Q,p}=\left(\frac{1}{|Q|}\int_Q |f(x)|^p\,dx \right)^{1/p}.
\]

\section{Sobolev estimates and continuity estimates at \texorpdfstring{$L^2\times L^{2}\rightarrow L^2$}{Lg}}\label{section222}

This section focuses on proving Sobolev smoothing estimates of the form $H^{-s_1}\times H^{-s_2}\rightarrow L^2$, $s_1,\,s_2\geq 0$ for single-scale maximal bilinear operators of the form $\mathcal{A}_{m,E}$, associated to an admissible multiplier $m$ and a fractal subset $E\subset [1,2]$, and similar bounds to maximal bilinear averaging operator $\mathcal{T}_E$ associated to triangle averages. First we deal with the special case $E=\{1\}$ because the argument is simpler and the result is slightly stronger. We then move to the study for general $E$ with (upper) Minkowski dimension $\beta$, by passing through the proof of decay bounds for pieces of $\mathcal{A}_{m,E}$. We then adapt the methods to the study of triangle averaging operators $\mathcal{T}_{E}$ by a more careful use of the decay of the associated multiplier.
\subsection{Single-scale bilinear operators associated to admissible multipliers: case \texorpdfstring{$E=\{1\}$}{Lg}} \label{Tmsubsection} We start by proving Sobolev type bounds of the form $H^{-s_1}\times H^{-s_2}\rightarrow L^2$, with $s_1,s_2\geq  0$ for the bilinear multiplier operator $T_m$ defined in (\ref{eqn: Tm}) and we show how that implies continuity estimates for it. This corresponds to the special case that the dilation set $E=\{1\}$. Note that in this special case, we do not need to assume that $\nabla{m}$ has any decay, all we need is enough decay for $m$.

From the definition (\ref{eqn: Tm}), for any $\varphi\in \mathcal{S}$, one has 
\begin{equation*}
    \begin{split}
        \int_{\R^d} T_m(f,g)(x)\hat{\varphi}(x)\,dx=&\int_{\R^{2d}} \hat{f}(\xi)\hat{g}(\eta)m(\xi,\eta)\varphi(\xi+\eta)\,d\xi d\eta\\
    =&\int_{\R^d}\int_{\R^d} \hat{f}(\xi-\eta)\hat{g}(\eta)m(\xi-\eta,\eta)\varphi(\xi)\,d\xi d\eta.
    \end{split}
\end{equation*}

Hence, the Fourier transform of $T_m(f,g)(x)$ is given by
\begin{equation}
\mathcal{F}(T_m(f,g))(\xi)=\int_{\R^d} \hat{f}(\xi-\eta)m(\xi-\eta,\eta)\hat{g}(\eta)\,d\eta.
\end{equation}

We first have the following Sobolev estimate.

\begin{prop}[Sobolev bounds for $T_m$]\label{sobolev} Assume $d\geq 1$ and that $m$ is an $a$-admissible multiplier up to order $0$, for some $a>d/2$. Then, for any $s_1,s_2\geq 0$ satisfying $s_1+s_2\leq \frac{2a-d}{2}$, one has  \begin{equation}
    \|T_m(f,g)\|_{L^2}\lesssim \|f\|_{H^{-s_1}}\|g\|_{H^{-s_2}}.
\end{equation}
In particular, it follows that $T_m:L^{2}\times L^{2}\rightarrow L^2$. 
 \end{prop}

 \begin{proof}[Proof of Proposition \ref{sobolev}]
 From the inequality $\|f\|_{H^{-s}}\leq \|f\|_{H^{-s'}}$  if $s'\leq s$, once we know the proposition for $s_1+s_2=\frac{2a-d}{2}$ we actually get $T_m:\,H^{-s_1}\times H^{-s_2}\rightarrow L^2$ for all $s_1,s_2\geq 0$ with $s_1+s_2\leq \frac{2a-d}{2}$. So let us assume $s_1+s_2=\frac{2a-d}{2}$.
In this case, $s_1$ and $s_2$ can not both be zero, so we can also assume by symmetry that $s_1>0$.
Using Plancherel's Theorem, followed by Minkowski's and H\"older's inequality, we get
\begin{equation*}
    \begin{split}
\|T_m(f,g)\|_{L^2}=&\|\mathcal{F}(T_m(f,g))\|_{L^2}\\
        =&\left\|\int \hat{f}(\xi-\eta)\hat{g}(\eta)m(\xi-\eta,\eta)\,d\eta\right\|_{L^{2}_{\xi}}\\
        \leq&\int |\hat{g}(\eta)|\left\|\hat{f}(\xi-\eta)m(\xi-\eta,\eta)\right\|_{L^{2}_{\xi}}\,d\eta\\
        =&\int |\hat{g}(\eta)|\left(\int|\hat{f}(\xi)m(\xi,\eta)|^2\,d\xi\right)^{1/2}\,d\eta\\
        \leq & \int \frac{|\hat{g}(\eta)|}{(1+|\eta|)^{s_2}}\left(\int\frac{|\hat{f}(\xi)|^2}{(1+|\xi|+|\eta|)^{2s_1+d}}\,d\xi\right)^{1/2}\,d\eta\\
        \leq & \left(\int \frac{|\hat{g}(\eta)|^2}{(1+|\eta|)^{2s_2}}\,d\eta\right)^{1/2}\left( \int\int\frac{|\hat{f}(\xi)|^2}{(1+|\xi|+|\eta|)^{2s_1+d}}\,d\eta d\xi\right)^{1/2},
    \end{split}
\end{equation*}
where the second to last step in the above follows from the decay assumption on the symbol $m$. We observe that for any $\xi\in \R^d$,
    $$\int_{\R^d} \dfrac{1}{(1+|\xi|+|\eta|)^{2s_1+d}}\,d\eta\lesssim \dfrac{1}{(1+|\xi|)^{2s_1}}.$$
Indeed, using polar coordinates,
\begin{equation*}
    \begin{split}
      \int_{\R^d} \dfrac{1}{(1+|\xi|+|\eta|)^{2s_1+d}}\,d\eta=&c_d\int_{0}^{\infty} \dfrac{r^{d-1}}{(1+|\xi|+r)^{2s_1+d}}\,dr\\
      \lesssim & \dfrac{1}{(1+|\xi|)^{2s_1+d}}\int_{0}^{1+|\xi|} r^{d-1} \,dr+\int_{1+|\xi|}^{\infty}r^{-2s_1-1}\,dr\\
      \lesssim& (1+|\xi|)^{-2s_1}.
    \end{split}
\end{equation*}
 
Coming back to our estimate,
\begin{equation*}
    \begin{split}
        \|T_{m}(f,g)\|_{L^2}\lesssim \|g\|_{H^{-s_2}}\left(\int|\hat{f}(\xi)|^2 \dfrac{1}{(1+|\xi|)^{2s_1}}\,d\xi\right)^{1/2}\lesssim \|g\|_{H^{-s_2}} \|f\|_{H^{-s_1}}.
    \end{split}
\end{equation*}
\end{proof}

In the case that $T_m$ is a bilinear averaging operator over smooth compact surfaces, using the known decay estimate for $m=\hat{\mu}_S$, i.e. $\alpha=0$ in inequality (\ref{decaysurfacemulitplier}), we immediately derive the following results. 

\begin{cor}\label{corollaryforsurfaces}
    Let $d\geq 2$, and $m(\xi,\eta)=\hat{\sigma}_{2d-1}$. Then, for any $s_1, s_2\geq 0$ with $s_1+s_2\leq\frac{d-1}{2}$, one has
    \begin{equation}
\|\mathcal{A}_1(f,g)\|_{L^2}\lesssim \|f\|_{H^{-s_1}}\|g\|_{H^{-s_2}}.
    \end{equation}
    More generally, let $d\geq 2$, for a $(2d-1)$-dimensional compact smooth surface $S$ in $\R^{2d}$ without boundary such that $k$ of the $(2d-1)$ principal curvatures do not vanish, and $k>d$, one has that
\begin{equation}
\|\mathcal{A}_{\hat\mu_S,1}(f,g)\|_{L^2}\lesssim \|f\|_{H^{-s_1}}\|g\|_{H^{-s_2}},
\end{equation} for any $s_1,s_2\geq 0$ with $s_1+s_2\leq\frac{k-d}{2}$.
\end{cor}

\begin{rem}
    For the particular case $E=\{1\}$ and $m=\hat{\sigma}_{2d-1}$, the corollary above was obtained independently in \cite{GPP2023}, even though they stated it slightly differently in terms of what it implies for the pieces of the operator $\mathcal{A}_1$. Their result can recover $\|\mathcal{A}_{1}(f,g)\|_{L^2}\lesssim \|f\|_{H^{-s_1}}\|g\|_{H^{-s_2}}$ for $s_1,\,s_2\geq 0$ and $s_1+s_2<\frac{d-1}{2}$.  
\end{rem}

Back to the general bilinear operator $T_m$ and recall the notation $\tau_h(f)(x)=f(x-h)$ for any $h\in \R^d$. Proposition \ref{sobolev} can be used to deduce the following continuity estimate for $T_m$, which plays a key role in deriving the sparse bound for its corresponding multi-scale maximal operator.

\begin{cor}[Continuity estimates for $T_m$]\label{continuity222} Assume $d\geq 1$ and that $m$ is an $a$-admissible multiplier up to order $0$ for some $a>d/2$. Then there exists $\gamma>0$ such that
\begin{equation}\label{continuityinoneinput}
    \|T_m(f-\tau_h f, g)\|_{L^2}+\|T_m(f, g-\tau_h g)\|_{L^2}\lesssim |h|^{\gamma} \|f\|_{L^2} \|g\|_{L^2},\quad \forall |h|< 1.
\end{equation}
Moreover, under the same hypothesis there exist $\gamma_1,\gamma_2>0$ such that
\begin{equation}\label{continuityinboth}
    \|T_m(f-\tau_{h_1} f, g-\tau_{h_2}g)\|_{L^2}\lesssim |h_1|^{\gamma_1}|h_2|^{\gamma_2}\|f\|_{L^2} \|g\|_{L^2},\quad \forall |h_1|< 1, |h_2|<1.
\end{equation}
\end{cor}
\begin{proof}[Proof of Corollary \ref{continuity222}]
From Proposition \ref{sobolev}, with $s_2=0$ and $s_1=\frac{2a-d}{2}$, we have
\begin{equation*}
    \|T_m(f-\tau_h f,g)\|_{L^2}\lesssim \|f-\tau_h f\|_{H^{-s_1}}\|g\|_{L^2}= \left(\int |\hat{f}(\xi)|^2\dfrac{|1-e^{-2\pi i h\cdot \xi}|^2}{(1+|\xi|)^{2a-d}}\,d\xi\right)^{1/2}\|g\|_{L^2}. 
\end{equation*}
Let $\gamma=\frac{1}{2}\text{min}(2a-d,2)>0$. Then, 
\begin{equation*}
\begin{split}
    \dfrac{|1-e^{-2\pi i h\cdot \xi}|^2}{(1+|\xi|)^{2a-d}}\lesssim&\begin{cases}
        \dfrac{(|h||\xi|)^2}{(1+|\xi|)^{2a-d}}\leq \dfrac{(|h||\xi|)^{2\gamma}}{(1+|\xi|)^{2a-d}}\lesssim |h|^{2\gamma} &\text{ if }|\xi|\leq \dfrac{1}{|h|}\\
        \dfrac{1}{|\xi|^{2a-d}}\lesssim |h|^{2a-d} &\text{ if }|\xi|\geq \dfrac{1}{|h|}
    \end{cases} \\
    \lesssim& |h|^{2\gamma},\,\text{ for all }|h|\leq 1.
    \end{split}
\end{equation*}
Similarly, one can get continuity estimates with the translation in the second input function. The proof of (\ref{continuityinboth}) follows for a simple adaptation of this argument by applying Proposition \ref{sobolev} with $s_1=s_2=\frac{2a-d}{4}$ for example. Alternatively, by applying the continuity estimate in each separate input, we know there exists $\gamma_1,\,\gamma_2>0$ such that 
$$\|T_m(f-\tau_{h_1}, g-\tau_{h_2}g)\|_{L^2} \lesssim |h_1|^{\gamma_1}\|f\|_1\|g-\tau_{h_2}g\|_{L^2}\lesssim |h_{1}|^{\gamma_{1}}\|f\|_{L^2}\|g\|_{L^2}$$
and 
$$\|T_m(f-\tau_{h_1}, g-\tau_{h_2}g)\|_{L^2} \lesssim|h_{2}|^{\gamma_{2}}\|f\|_{L^2}\|g\|_{L^2}.$$
Multiplying these two inequalities we get 
$$\|T_m(f-\tau_{h_1}, g-\tau_{h_2}g)\|_{L^2}^2 \lesssim |h_1|^{\gamma_1}|h_{2}|^{\gamma_{2}}\|f\|_{L^2}^2\|g\|_{L^2}^2,$$
leading to the desired bound.

\end{proof}

\begin{rem}\label{multilinear222}
    The $L^2\times L^2\rightarrow L^2$ approach generalizes easily to the multilinear setting. Let $\ell\geq 2$ and assume that  
    $$|m(\xi)|\lesssim (1+|\xi|)^{-a}$$
    where $\xi=(\xi_1,\xi_2\cdots,\xi_\ell)\in \R^{\ell d}$. If $2a>(\ell-1)d$, then
$$T_{m}(f_1,\dots, f_\ell)(x):=\int_{S^{\ell d-1}}\hat{f_1}(\xi_1)\hat{f_2}(\xi_2)\cdots \hat{f_\ell}(\xi_\ell)m(\xi)e^{2\pi i x\cdot (\xi_1+\xi_2+\dots \xi_\ell)}\,d\xi$$
satisfies
$$\|T_m(f_1,f_2,\dots, f_\ell)\|_{L^2}\lesssim \|f_1\|_{H^{-\frac{2a-(\ell-1)d}{2}}}\|f_2\|_{L^2}\dots \|f_\ell\|_{L^2}.$$
More generally, $T_m:\, H^{-s_1}\times \cdots \times H^{-s_\ell}\rightarrow L^2$ for all $s_i\geq 0$ with $\sum_{i=1}^{\ell}s_i\leq \frac{2a-(\ell-1)d}{2}$.
\end{rem}

\subsection{Single-scale bilinear maximal operators associated to admissible multipliers and general dilation sets \texorpdfstring{$E\subset[1,2]$}{Lg}}\label{Eadmissiblemsubsection}

In this section, given $E\subset [1,2]$, with a certain (upper) Minkowski dimension $\text{dim}_{M}(E)$, we are interested in proving Sobolev smoothing bounds and continuity estimates for the more general single-scale bilinear maximal operators 
\begin{equation}\label{defAmE}
\mathcal{A}_{m,E}(f,g)(x)=\sup_{t\in E} |T_{m_t}(f,g)(x)|
\end{equation}
where $m_t(\xi,\eta)=m(t\xi,t\eta)$. When $m=\hat{\sigma}_{2d-1}$ we might simply write $\mathcal{A}_E=\mathcal{A}_{m,E}$.

Similarly to what is done in \cite{HHY,borgesfoster}, we start with choosing a radial function $\varphi\in \mathcal{S}(\R^d)$ such that 
\begin{equation}\label{phi def}
\hat{\varphi}(\xi)= \begin{cases}
   1, & \text{ if }|\xi|\leq 1,\\
   0,& \text{ if }|\xi|\geq 2.
    \end{cases}
\end{equation}

Let $\hat{\psi}(\xi)=\hat{\varphi}(\xi)-\hat{\varphi}(2\xi)$, which is supported in $\{1/2<|\xi|<2\}$. Then
\begin{equation}\label{psi def}
\hat{\varphi}(\xi)+\sum_{j=1}^{\infty} \hat{\psi}(2^{-j}\xi)\equiv 1.
\end{equation}

For all $i,j\geq 1$ and $t\in [1,2]$, define
\begin{equation}
\begin{split}
 T_{m_t}^{i,j}(f,g)(x):=&\int_{\R^{2d}}\hat{f}(\xi)\hat{g}(\eta)\hat{m}(t\xi, t\eta)\hat{\psi}(2^{-i}\xi)\hat{\psi}(2^{-j}\eta)e^{2\pi i x\cdot (\xi+ \eta)}\,d\xi d\eta \\
 =&T_{m_t}(f^i,g^j)(x)
\end{split}
\end{equation}
where $\hat{f^i}(\xi):=\hat{f}(\xi)\hat{\psi}(2^{-i}\xi)$.
If $i=0$, replace $\hat{\psi}(2^{-i}\xi)$ by $\hat{\varphi}(\xi)$ in the expression above, and similarly if $j=0$. Then one has 
$$T_{m_t}(f,g)(x)=\sum_{i,j\geq 0} T_{m_t}^{i,j}(f,g)(x).$$

Let us consider the pieces of the operator $\mathcal{A}_{m,E}$ given by
\begin{equation}
    \mathcal{A}_{m,E}^{i,j}(f,g)(x)=\sup_{t\in E}|T_{m_t}^{i,j}(f,g)(x)|=\mathcal{A}_{m,E}(f^i,g^j).
\end{equation}Obviously,
\[
\mathcal{A}_{m,E}(f,g)(x)\leq \sum_{i,j\geq 0}\mathcal{A}_{m,E}^{i,j}(f,g)(x).
\]

Recall from Section \ref{sec:notation} that $N(E,\delta)$ denotes the minimal number of closed intervals of size $\delta$ that one needs to cover the set $E$, and that if  $\text{dim}_{M}(E)=\beta$, then
$$N(E,\delta)\leq C_{\eps} \delta^{-(\beta+\eps)},\,\text{ for all }\delta<1\text{ and }\eps>0.$$

\begin{thm}[Decay estimates for the pieces of $\mathcal{A}_{m,E}$]\label{piecesofA_E}
    Let $d\geq 1$, and let $E\subset [1,2]$ with $\beta=\text{dim}_{M}E$. Let $m$ be an $a$-admissible bilinear multiplier up to order $1$. If $2a>d+\beta$, then for any $i,\,j\geq 0$ one has
    \begin{equation}
      \|\mathcal{A}_{m,E}^{i,j}(f,g)\|_{L^2}\lesssim N(E,2^{-\max\{i,j\}})^{\frac{1}{2}}2^{-\max\{i,j\}a}2^{\min\{i,j\}\frac{d}{2}}\|f\|_{L^2}\|g\|_{L^2}.
    \end{equation}
    In particular, 
    \begin{equation}\label{decaymaximumij}
        \begin{split}
           \|\mathcal{A}_{m,E}^{i,j}(f,g)\|_{L^2} \lesssim_{\eps}& 2^{-\max\{i,j\}(\frac{2a-d-\beta-\eps}{2})}\|f\|_{L^2}\|g\|_{L^2}.
        \end{split}
    \end{equation}
\end{thm}

Before proving Theorem \ref{piecesofA_E}, let us discuss how it implies Sobolev smoothing estimates and continuity estimates for $\mathcal{A}_{m,E}$ (stated in Theorem \ref{sobolevE} and Corollary \ref{continuityA_E} respectively).

\begin{proof}[Proof of Theorem  \ref{sobolevE} assuming Theorem \ref{piecesofA_E}]
From Theorem \ref{piecesofA_E} we know that for all $i,j\geq 0$,
$$ \|\mathcal{A}_{m,E}^{i,j}(f,g)\|_{L^2}\lesssim_{\eps} 2^{-\max\{i,j\}(\frac{2a-d-\beta-\eps}{2})}\|f^i\|_{L^2}\|g^j\|_{L^2},$$
where $\hat{f^i}(\xi)=\hat{f}(\xi)\hat{\psi}(2^{-i}\xi)$, $i\geq 1$, and $\hat{f}^0(\xi)=\hat{f}(\xi)\hat{\varphi}(\xi)$.
Then,
\begin{equation}
    \begin{split}
        \|\mathcal{A}_{m,E}(f,g)\|_{L^2}\lesssim & \sum_{i,j\geq 0} \|\mathcal{A}_{m,E}(f^i,g^j)\|_{L^2}\\
        \lesssim_{\eps} &\sum_{i\geq 0}\|f^i\|_{L^2}\sum_{j\geq i} 2^{-j(\frac{2a-d-\beta-\eps}{2})}\|g^j\|_{L^2}\\
        &+\sum_{i\geq 0}\sum_{j< i} \|f^i\|_{L^2}2^{-i(\frac{2a-d-\beta-\eps}{2})}\|g^j\|_{L^2}=I+II.
    \end{split}
\end{equation}
Observe that by H\"older's inequality (first in the sum in $j$ and later in the sum in $i$),
\begin{equation*}
    \begin{split}
        I\lesssim &\sum_{i\geq 0}\|f^i\|_{L^2}\sum_{j\geq i}2^{-j(\frac{2a-d-2s_2-\beta-\eps}{2})}2^{-js_2}\|g^j\|_{L^2}\\
        \lesssim& \sum_{i\geq 0}\|f^i\|_{L^2}\|g\|_{H^{-s_2}}\left(\sum_{j\geq i} 2^{-j(2a-d-2s_2-\beta-\eps)}\right)^{1/2}\\
        \lesssim&\|g\|_{H^{-s_2}}\sum_{i\geq 0}2^{-is_1}\|f^i\|_{L^2} 2^{-i(\frac{2a-d-2s_1-2s_2-\beta-\eps}{2})}\\
        \lesssim_{\eps}& \|f\|_{H^{-s_1}}\|g\|_{H^{-s_2}}
    \end{split}
\end{equation*}
because $2a>d+2s_1+2s_2+\beta$.
Similarly, by reversing the roles of $i$ and $j$,
\begin{equation*}
    II\lesssim \sum_{j\geq 0}\|g^j\|_{L^2}\sum_{i>j}2^{-i(\frac{2a-d-2s_1-\beta-\eps}{2})}2^{-is_1}\|f^i\|_{L^2}
    \lesssim \|f\|_{H^{-s_1}}\|g\|_{H^{-s_2}}
\end{equation*}
because $2a>d+2s_1+2s_2+\beta$.
\end{proof}

\begin{cor}[Continuity estimates for $\mathcal{A}_{m,E}$]\label{continuityA_E}
Let $d\geq 1$ and $E\subset [1,2]$ with upper Minkowski dimension $\text{dim}_{M}(E)=\beta$. Let $m$ be an $a$-admissible bilinear multiplier up to order $1$, and assume that $2a>d+\beta$. Then, there exists $\gamma>0$ such that
\begin{equation}\label{continuityoneinputforE}
    \|\mathcal{A}_{m,E}(f-\tau_h f,g)\|_{L^2}+\|\mathcal{A}_{m,E}(f,g-\tau_h g)\|_{L^2}\lesssim |h|^{\gamma}\|f\|_{L^2}\|g\|_{L^2},\quad\forall |h|<1.
\end{equation}
Moreover, under the same hypothesis there exist $\gamma_1,\gamma_2>0$ such that
\begin{equation}\label{continuityinbothforE}
    \|\mathcal{A}_{m,E}(f-\tau_{h_1} f, g-\tau_{h_2}g)\|_{L^2}\lesssim |h_1|^{\gamma_1}|h_2|^{\gamma_2}\|f\|_{L^2} \|g\|_{L^2},\quad \forall |h_1|< 1, |h_2|<1.
\end{equation}
\end{cor}

\begin{proof}[Proof of Corollary \ref{continuityA_E} assuming Theorem \ref{sobolevE}]
Since we are assuming $2a>d+\beta$, one can take $s_1\in(0, \frac{2a-d-\beta}{2})$ and $s_2=0$ in Theorem \ref{sobolevE} to say that 
\begin{equation*}
    \begin{split}
        \|\mathcal{A}_{m,E}(f-\tau_hf,g)\|_{L^2}\lesssim& \|f-\tau_h f\|_{H^{-s_1}} \|g\|_{L^2}\\
        = & \left\|\frac{(1-e^{2\pi ih\cdot \xi})}{(1+|\xi|)^{s_1}}\hat{f}(\xi)\right\|_{L^2_{\xi}} \|g\|_{L^2} \lesssim |h|^{\gamma} \|f\|_{L^2}\|g\|_{L^2},
    \end{split}
\end{equation*}
for $\gamma=\min\{s_1, 1\}>0$. To bound $\|\mathcal{A}_{m,E}(f,g-\tau_h g)\|_{2}$ one can apply Theorem \ref{sobolevE} with $s_1=0$ and $s_2\in (0, \frac{2a-d-\beta}{2})$ and for the bound for $\|\mathcal{A}_{m,E}(f-\tau_{h_1}f,g-\tau_{h_2}g )\|_2$ it is enough to take $s_1, s_2>0$ with $s_1+s_2<\frac{2a-d-\beta}{2}$.
\end{proof}
\begin{rem}\label{rem: continuityforsigmahat}
    The continuity estimates (\ref{continuityoneinputforE}) and (\ref{continuityinbothforE}) for the particular case of $m=\hat{\sigma}_{2d-1}$ were already known for any $d\geq 2$ and $\beta\in [0,1]$ as a consequence of the continuity estimates for $\tilde{\mathcal{M}}=\mathcal{M}_{\hat\sigma_{2d-1},[1,2]}$ in $d\geq 2$ at the exponent $L^{2}\times L^{2}\rightarrow L^1$ proved in \cite{BFOPZ}. As already observed in there, interpolation leads to continuity estimates at any point in the interior of the boundedness region of $\tilde{\mathcal{M}}$, and from the partial description of the boundedness region for $\tilde{\mathcal{M}}$ given by \cite{JL}, one could check that the point $(1/2,1/2,1/2)$ lives in the interior of the boundedness region for $\tilde{\mathcal{M}}$. Therefore, Corollary \ref{continuityA_E} (which, when $m=\hat\sigma_{2d-1}$, holds in all cases of ($d, \beta$) except ($2,1$)) recovers, as a special case, all previously known continuity estimates for $\tilde{\mathcal{M}}$ in $d\geq 3$. 
\end{rem}

\begin{rem}\label{comparisonwithpalsson}
    When $\beta>0$, the use of the $L^2\times L^2\rightarrow L^2$ decay bounds for the pieces seems more effective for obtaining continuity estimates than using the $L^2\times L^2\rightarrow L^1$ boundedness criteria from \cite{GHS}, since using that one the decay bounds look more like 
    \begin{equation}
    \begin{split}
        \|\mathcal{A}_{m,E}^{i,j}(f,g)\|_{1}\lesssim& N(E, 2^{-\max\{i,j\}})2^{-\max\{i,j\}a}2^{(i+j)\frac{d}{4}}\|f\|_{L^2}\|g\|_{L^2}\\
        \lesssim_{\eps}& 2^{-\max\{i,j\}\frac{(2a-d-2\beta-2\eps)}{2}}\|f\|_{L^2}\|g\|_{L^2}.
    \end{split}
    \end{equation}
    Since $2a-d-\beta>2a-d-2\beta$, it is easier to have a negative power of $2^{\max\{i,j\}}$ with our $L^2\times L^{2}\rightarrow L^2$ strategy.
    That explains why in the case  $m=\hat{\sigma}_{2d-1}$ and $E=[1,2]$, in \cite{sparsetriangle} they needed the assumption $d\geq 4$ to get continuity estimates for $\tilde{\mathcal{M}}$.
\end{rem}

\begin{proof}[Proof of Theorem \ref{piecesofA_E}]
Let us look at the case $i\geq j$. The case $j\geq i$ is analogous.

We start with the case $E=\{t\}$, with $t\sim 1$. Then $ |m(t\xi,t\eta)|\lesssim \frac{1}{(1+|\xi|+|\eta|)^{a}}$.

From the proof of Proposition \ref{sobolev} it follows that
\begin{equation}\label{boundpieces}
    \begin{split}
        \|T_{m_t}^{i,j}(f,g)\|_{L^2}\lesssim & \int_{|\eta|\sim 2^j} |\hat{g}(\eta)|\left(\int_{|\xi|\sim 2^i}\dfrac{|\hat{f}(\xi)|^2}{(1+|\xi|+|\eta|)^{2a}}\,d\xi\right)^{1/2}\,d\eta \\
        \lesssim &2^{-ia}\|f\|_{L^2} \int_{|\eta|\sim 2^j}|\hat{g}(\eta)|\,d\eta\lesssim 2^{-ia}2^{jd/2}\|f\|_{L^2} \|g\|_{L^2}
    \end{split}
\end{equation}
and for this bound one does not even need to assume $2a>d$.

For each $\nu\in \mathbb{Z}_{\leq 0}$, let $\mathcal{I}_\nu=\mathcal{I}_{\nu}(E)$ be the family of all binary intervals $[n2^{\nu},(n+1)2^{\nu})$, $n\in \N$, that contains a point in $E$. Then one has $\#\mathcal{I}_{-i}\lesssim N(E,2^{-i})$, $\forall i\geq 0$.

Take a non-negative function $\alpha\in C_c^{\infty}(\R)$, such that $\alpha\equiv 1$ in $|t|\leq 1/2$ and $\text{supp}(\alpha)\subset (-1,1)$. For $I\in \mathcal{I}_{\nu}$, say 
$I=[c_I-2^{\nu-1},c_I+2^{\nu-1})$, define
$$\alpha_{I}=\alpha(2^{-\nu}(t-c_I)).$$
Notice that $\alpha_I\equiv 1$ on $I$ and $\text{supp}(\alpha_I)\subset 2I=:\tilde{I}$, the concentric dilation of $I$ with length $2|I|$.

One has that 
\begin{equation*}
    \begin{split}
        |\mathcal{A}_{m,E}^{i,j}(f,g)(x)|^2= &\sup_{t\in E} |T_{m_t}^{i,j} (f,g)(x)|^2 \\
        \leq& \sum_{I\in \mathcal{I}_{-i}} \sup_{t\in I} |T_{m_t}^{i,j}(f,g)(x)|^2\leq \sum_{I\in \mathcal{I}_{-i}} \sup_{t\in I} |\alpha_{I}(t)T_{m_t}^{i,j}(f,g)(x)|^2.
    \end{split}
\end{equation*}
By the Fundamental theorem of calculus, 
\begin{equation*}
    \begin{split}
        \sup_{t\in I}|\alpha_I(t)&T_{m_t}^{i,j}(f,g)(x)|^2=\sup_{t\in I}
        \left|\int_{c_I-2^{-i}}^{t}\dfrac{d}{ds}|\alpha_I(s)T^{i,j}_{m_s}(f,g)(x)|^2\,ds\right|\\
        &\lesssim |I|^{-1}\int_{\tilde{I}} |T_{m_t}^{i,j}(f,g)(x)|^2\,dt+\int_{\tilde{I}} \left|T_{m_t}^{i,j}(f,g)(x)\right|\left|\frac{d}{dt}T_{m_t}^{i,j}(f,g)(x)\right|\,dt\\
        &\lesssim |I|^{-1}\int_{\tilde{I}} |T_{m_t}^{i,j}(f,g)(x)|^2\,dt+\frac{1}{\big|\tilde{I}\big|}\int_{\tilde{I}} |T_{m_t}^{i,j}(f,g)(x)|\left|2^{-i}\frac{d}{dt}T_{m_t}^{i,j}(f,g)(x)\right|\,dt.
        \end{split}
\end{equation*}

Putting all together and using inequality (\ref{boundpieces}),
\begin{equation*}
    \begin{split}
        \int_{\R^d} &|\mathcal{A}_{m,E}^{i,j} (f,g)(x)|^2 \,dx\\
        \lesssim& \sum_{I\in \mathcal{I}_{-i}} \left\{ \dashint_{\tilde{I}} \|T_{m_t}^{i,j} (f,g)(x)\|_{L^2_x}^2  \,dt+ \dashint_{\tilde{I}} \left\|T_{m_t}^{i,j}(f,g)(x)\cdot 2^{-i}\frac{d}{dt}T_{m_t}^{i,j}(f,g)(x)\right\|_{L^1_x}\,dt\right\}\\
        \lesssim &\sum_{I\in \mathcal{I}_{-i}} \left\{  2^{-2ai}2^{jd}\|f\|_{L^2}^2\|g\|_{L^2}^{2} + \dashint_{\tilde{I}} \|T_{m_t}^{i,j}(f,g)(x)\|_{L^{2}_x}\left\| 2^{-i}\frac{d}{dt}T_{m_t}^{i,j}(f,g)(x)\right\|_{L^2_x}\,dt\right\}.
        \end{split}
\end{equation*} 

   We claim that $  \left\|2^{-i}\frac{d}{dt}T_{m_t}^{i,j}(f,g)(x)\right\|_{L^2}\lesssim 2^{-i(\frac{2a}{2})}2^{j\frac{d}{2}}\|f\|_{L^2}\|g\|_{L^2}$, for all $t\in [1,2]$. 
Assuming this claim is true, then the theorem follows because
 \begin{equation*}
        \int_{\R^d} |\mathcal{A}_{m,E}^{i,j}(f,g)(x)|^2\, dx \lesssim   \sum_{I\in
        \mathcal{I}_{-i}}  2^{-i2a}2^{jd}\|f\|_{L^2}^2\|g\|_{L^2}^2 \lesssim (\#\mathcal{I}_{-i})2^{-i2a}2^{jd}\|f\|_{L^2}^2\|g\|_{L^2}^2,
        \end{equation*}
 which implies
\begin{equation*}
\begin{split}
\|\mathcal{A}_{m,E}^{i,j} (f,g)(x)\|_{L^2}^2\lesssim &N(E,2^{-i})2^{-i2a}2^{jd}\|f\|_{L^2}^2\|g\|_{L^2}^2\\
 =&N(E, 2^{-\max\{i,j\}})2^{-\max\{i,j\}2a}2^{\min\{i,j\}d}\|f\|_{L^2}^2\|g\|_{L^2}^2
 \end{split}
\end{equation*}
since we are assuming $i\geq j$.

Now let us check the claim. Since
$$T_{m_t}^{i,j}(f,g)(x)=\int_{\R^d}m(t\xi,t\eta)\hat{\psi}(2^{-i}\xi)\hat{\psi}(2^{-j}\eta)\hat{f}(\xi)\hat{g}(\eta)e^{2\pi ix\cdot (\xi+\eta)}\,d\xi d\eta,$$
one has 
$$\dfrac{d}{dt}T_{m_t}^{i,j}(f,g)(x)=\int_{\R^d} \dfrac{d}{dt}(m(t\xi,t\eta))\hat{\psi}(2^{-i}\xi)\hat{\psi}(2^{-j}\eta)\hat{f}(\xi)\hat{g}(\eta)e^{2\pi i x\cdot (\xi+\eta)}\,d\xi d\eta.$$
Observe that since $t\sim 1$, 
$\tilde{m}_t(\xi,\eta):=\dfrac{d}{dt}\left(m(t\xi, t\eta)\right)$ satisfies 
\begin{equation*}
    \begin{split}
       | \tilde{m}_t(\xi,\eta)|=&|\nabla  m(t\xi, t\eta)\cdot (\xi,\eta)|\\
       \lesssim & \dfrac{1}{(1+|\xi|+|\eta|)^{a-1}}.
    \end{split}
\end{equation*} 
Therefore, from the same argument as in estimate (\ref{boundpieces}) one has 
\begin{equation*}
    \begin{split}
        \left\|\dfrac{d}{dt}T_{m_t}^{i,j}(f,g)\right\|_{L^2}=&\|T_{\tilde{m}_t}(f^i,g^j)\|_{L^2}\\
        \lesssim& 2^{-i(\frac{2a-2}{2})}2^{j\frac{d}{2}}\|f\|_{L^2}\|g\|_{L^2}\\
        =& 2^i 2^{-i(\frac{2a}{2})}2^{j\frac{d}{2}}\|f\|_{L^2}\|g\|_{L^2}.
    \end{split}
\end{equation*}

\end{proof}

\subsection{Single-scale maximal operators associated to triangle averages and general dilation sets \texorpdfstring{$E\subset[1,2]$}{Lg}}

The goal of this subsection is to prove decay estimates for single-scale maximal bilinear operator associated with the triangle averages $\mathcal{T}_t$. Recall from the introduction that the multiplier associated with the triangle average is not admissible, hence results from the previous section does not apply. In addition to the techniques already mentioned, we need to consider a more refined angular decomposition of the multiplier here.

Here are the details. Assume for now that $t=1$. Along the lines of \cite{borgesfoster}, the way we choose to break $\mathcal{T}_1$ into pieces is 
$$\mathcal{T}_1(f,g)(x)=\sum_{i\geq 0}\sum_{j\geq 0}\left(\sum_{k\geq 0} \mathcal{T}_1^{i,j,k}(f,g)(x)\right)=\sum_{i\geq 0}\sum_{j\geq 0} \mathcal{T}_1^{i,j}(f,g)(x),$$
where the Fourier transform of a piece of the triangle operator is given by
\[
\widehat{\mathcal{T}^{i,j,k}_1}(f,g)(\xi)= \int_{\BR^d}\hat{f}(\xi-\eta)\hat{g}(\eta)m_{i,j,k}(\xi-\eta,\eta)\,d\eta.
\]
Here, the $i$ index denotes localization in $\xi$, the $j$ index denotes localization in $\eta$ and the $k$ index denotes localization in $\sin\theta$ where $\theta$ is the angle between $\xi,\eta$. The multiplier $m_{i,j,k}$ is the piece of the multiplier with the above localizations. More precisely, 
$$m_{i,j,k}(\xi,\eta)=\hat{\mu}(\xi,\eta)\hat{\psi}(2^{-i}\xi)\hat{\psi}(2^{-j}\eta)\rho_k(\xi,\eta),$$
where $\hat{\psi}$ is the same as in the previous subsection, and $\rho_k$ are smooth functions satisfying $\sum_{k\geq 0}\rho_k(\xi,\eta)\equiv1$ except at the origin and
\begin{equation}
\begin{split}
\text{supp}(\rho_k)&\subset\{(\xi, \eta)\colon 2^{-k-1}\leq |\sin(\theta)|\leq 2^{-k+1}\},\,\text{if }k\geq1;\\
    \text{supp}(\rho_0)&\subset\{(\xi, \eta)\colon |\sin(\theta)|\geq 1/2\}.
\end{split}
\end{equation}

We estimate the $L^2$ norm using Plancherel identity and Minkowski inequality:
\[
\Vert \mathcal{T}^{i,j,k}_1(f,g)\Vert_{L^2}\le \int|\hat{g}(\eta)|\Vert \hat{f}(\cdot-\eta)m_{i,j,k}(\cdot-\eta,\eta)\Vert_{L^2}\,d\eta.
\]
We change variables $\xi-\eta\mapsto \xi$ and apply Cauchy-Schwarz to further bound this by
\[
\Vert \mathcal{T}^{i,j,k}_1(f,g)\Vert_{L^2}\le \Vert g\Vert_{L^2} \left(\iint |\hat{f}(\xi)|^2|m_{i,j,k}(\xi,\eta)|^2\,d\eta d\xi\right)^{1/2}.
\]
Now, we appeal to the results on the size and support of the truncated multiplier; from the decay bounds recalled in inequality (\ref{decaytriangle}) we have that
\[
|m_{i,j,k}(\xi,\eta)|\lesssim (1+2^{\min\{i,j\}}2^{-k})^{-(d-2)/2}(1+2^{\max(i,j)})^{-(d-2)/2}
\]
and that for fixed $\xi$, the support in the $\eta$ variable has measure at most $2^{jd}2^{-k(d-1)}$. Let's assume $i\ge j$ for simplicity.
In order to get the most decay possible from the multiplier there are two cases for $k$ to consider.
$$|m_{i,j,k}(\xi,\eta)|\lesssim \begin{cases}
   2^{-(j-k)\frac{d-2}{2}}2^{-i(\frac{d-2}{2})},\,\text{ if }0\leq k\leq j;\\
   2^{-i(\frac{d-2}{2})},\,\text{ if } k> j.
\end{cases}$$

If $0\leq k\leq j$
we get the estimate on the last factor
\begin{equation*}
    \begin{split}
       \left(\int |\hat{f}(\xi)|^2\left(\int|m_{i,j,k}(\xi,\eta)|^2\,d\eta\right) \,d\xi\right)^{1/2}\lesssim &\left(\int |\hat{f}(\xi)|^22^{-(d-2)(i+j-k)}|\text{supp}(m_{i,j,k})(\xi,\cdot)|\,d\xi\right)^{1/2}\\
       \lesssim &\left(\int |\hat{f}(\xi)|^22^{-(d-2)(i+j-k)}2^{jd}2^{-k(d-1)}\,d\xi\right)^{1/2}\\
       \lesssim& \Vert f\Vert_{L^2} 2^{-i(d-2)/2+j}2^{-k/2}. 
    \end{split}
\end{equation*}

Adding over $0\leq k\leq j$ we get  
$$\sum_{k=0}^{j}\|\mathcal{T}_1^{i,j,k}(f,g)\|_{L^2}\lesssim 2^{-i(d-2)/2}2^{j}\sum_{k=0}^{j}2^{-k/2}\|f\|_{L^2}\|g\|_{L^2}\lesssim 2^{-i(d-4)/2}\|f\|_{L^2}\|g\|_{L^2}.$$
Alternatively, for each $k$ such that $k>j$, the decay of the multiplier is instead 
$$|m_{i,j,k}(\xi,\eta)|\lesssim (1+2^{\max\{i,j\}})^{-(d-2)/2}$$
so we get instead
\[
\|\mathcal{T}_1^{i,j,k}(f,g)\|_{L^2}\lesssim \|g\|_{L^2} \Vert f\Vert_{L^2} 2^{-(i(d-2)-jd+k(d-1))/2}
\]
and we have 
$$\sum_{k=j}^{\infty} \|\mathcal{T}_1^{i,j,k}(f,g)\|_{L^2}\lesssim \|f\|_{L^2}\|g\|_{L^2} 2^{-i(d-2)/2+jd/2}2^{-j(d-1)/2}\lesssim 2^{-i(d-3)/2}\|f\|_{L^2}\|g\|_{L^2}.$$

Putting all together, for any $i\geq j$
\begin{equation}\label{decayforsinglescaletriangle}
    \begin{split}
\|\mathcal{T}_1^{i,j}(f,g)\|_{L^{2}}\lesssim\sum_{k=0}^{\infty}\|\mathcal{T}_{1}^{i,j,k}\|_{L^2}
    \lesssim & (2^{-i(d-3)/2}+2^{-i(d-4)/2})\|f\|_{L^{2}}\|g\|_{L^{2}}\\
    \lesssim &2^{-i(d-4)/2}\|f\|_{L^{2}}\|g\|_{L^{2}},
    \end{split}
\end{equation}
and more generally for any $i,j\geq 0$,
\begin{equation}
    \begin{split}
\|\mathcal{T}_1^{i,j}(f,g)\|_{L^{2}}\lesssim 2^{-\max\{i,j\}(d-4)/2}\|f\|_{L^{2}}\|g\|_{L^{2}}.
    \end{split}
\end{equation}

This argument works for any $t\in [1,2]$. Indeed, let 
$$m_{i,j,k;t}(\xi,\eta)=\hat{\mu}(t\xi,t\eta)\hat{\psi}(2^{-i}\xi)\hat{\psi}(2^{-j}\eta)\rho_k(\xi,\eta),$$
so 
$$\mathcal{T}_t(f,g)(x)=\sum_{i\geq 0}\sum_{j\geq 0}\sum_{k\geq 0} \mathcal{T}_{t}^{i,j,k}(f,g)(x)$$
where 
$$\mathcal{T}_{t}^{i,j,k}(f,g)(x)=\int \hat{f}(\xi)\hat{g}(\eta)m_{i,j,k;t}(\xi,\eta)e^{2\pi i x\cdot(\xi+\eta)}\,d\xi d\eta.$$
Observe that 
\[
\widehat{\mathcal{T}^{i,j,k}_t}(f,g)(\xi)= \int_{\BR^d}\hat{f}(\xi-\eta)m_{i,j,k;t}(\xi-\eta,\eta)\hat{g}(\eta)\,d\eta,
\]
and 
$$\mathcal{T}_{t}^{i,j}(f,g)(x):=\mathcal{T}_{t}(f^{i},g^{j})(x)=\sum_{k\geq 0}\mathcal{T}_{t}^{i,j,k}(f,g)(x).$$
It is immediate from the proof of inequality (\ref{decayforsinglescaletriangle}) that for any $t\in [1,2]$, 
\begin{equation}\label{decayforsinglescaletriangleatt}
    \|\mathcal{T}_t^{i,j}(f,g)\|_{L^{2}}\lesssim 2^{-\max\{i,j\}(d-4)/2}\|f\|_{L^{2}}\|g\|_{L^{2}}.
\end{equation}

Next we will use these decay bounds for pieces of $T_t$ to take of the more general dilation set $E\subset [1,2]$.

\begin{thm}[Decay estimates for the pieces of $\mathcal{T}_{E}$]\label{decayforpiecesofTE}
    Let $E\subset [1,2]$ with $\beta=\text{dim}_{M}E$. If $d>4+\beta$, then for any $i,\,j\geq 0$ one has
    \begin{equation}
        \begin{split}
         \|\mathcal{T}_{E}(f^i,g^j)\|_{L^2}&\lesssim N(E,2^{-\max\{i,j\}})^{\frac{1}{2}}2^{-\max\{i,j\}(d-4)/2}\|f^i\|_{L^2}\|g^j\|_{L^2} \\
          &\lesssim_{\eps} 2^{-\max\{i,j\}(d-4-\beta-\eps)/2}\|f\|_{L^2}\|g\|_{L^2}.
        \end{split}
    \end{equation}
\end{thm}

With the same argument as the one used in Subsection \ref{Eadmissiblemsubsection} for $\mathcal{A}_{m,E}$, the proof of Theorem \ref{sobolevtrianglethm} and Corollary \ref{continuityTE} below (Sobolev smoothing bounds and continuity estimates for $\mathcal{T}_{E}$ respectively) will follow once we prove Theorem \ref{decayforpiecesofTE}.

\begin{cor}[Continuity estimates for $\mathcal{T}_E$]\label{continuityTE}
Let $E\subset [1,2]$ with upper Minkowski dimension $\text{dim}_{M}(E)=\beta$, and assume that $d>4+\beta$. Then, there exists $\gamma>0$ such that
\begin{equation}
    \|\mathcal{T}_{E}(f-\tau_h f,g)\|_{L^2}+\|\mathcal{T}_{E}(f,g-\tau_h g)\|_{L^2}\lesssim |h|^{\gamma}\|f\|_{L^2}\|g\|_{L^2},\quad\forall |h|<1.
\end{equation}
Moreover, under the same hypothesis there exist $\gamma_1,\gamma_2>0$ such that
\begin{equation}
    \|\mathcal{T}_{E}(f-\tau_{h_1} f, g-\tau_{h_2}g)\|_{L^2}\lesssim |h_1|^{\gamma_1}|h_2|^{\gamma_2}\|f\|_{L^2} \|g\|_{L^2},\quad \forall |h_1|< 1, |h_2|<1.
\end{equation}
\end{cor}

\begin{proof}[Proof of Theorem \ref{decayforpiecesofTE}]
This proof is an adaptation of Theorem \ref{piecesofA_E}. We sketch the argument. Assume we are in the case $i\geq j$. Following the same notation of Theorem \ref{piecesofA_E},

   \begin{equation*}
    \begin{split}
        |\mathcal{T}_{E}(f^i,g^j)(x)|^2= &\sup_{t\in E} |\mathcal{T}_{t} (f^i,g^j)(x)|^2 \\
        \leq& \sum_{I\in \mathcal{I}_{-i}} \sup_{t\in I} |\mathcal{T}_{t}(f^i,g^j)(x)|^2\leq \sum_{I\in \mathcal{I}_{-i}} \sup_{t\in I} |\alpha_{I}(t)\mathcal{T}_{t}(f^i,g^j)(x)|^2.
    \end{split}
\end{equation*}
For any $I\in \mathcal{I}_{-i},$
\begin{equation*}
    \begin{split}
        \sup_{t\in I}|\alpha_I(t)&\mathcal{T}_{t}(f^i,g^j)(x)|^2=\sup_{t\in I}
        \left|\int_{c_I-2^{\nu}}^{t}\dfrac{d}{ds}|\alpha_I(s)\mathcal{T}_{s}(f^i,g^j)(x)|^2\,ds\right|\\
        &\lesssim |I|^{-1}\int_{\tilde{I}} |\mathcal{T}_{t}(f^i,g^j)(x)|^2\,dt+\frac{1}{|\tilde{I}|}\int_{\tilde{I}} |\mathcal{T}_{t}(f^i,g^j)(x)|\left|2^{-i}\frac{d}{dt}\mathcal{T}_{t}(f^i,g^j)(x)\right|\,dt.
        \end{split}
\end{equation*}

Observe that 
$$\frac{d}{dt}\mathcal{T}_{t}(f^{i},g^j)(x)=\int \frac{d}{dt}(\hat{\mu}(t\xi,t\eta)) \hat{\psi}_i(\xi)\hat{\psi}_j(\eta)\hat{f}(\xi)\hat{g}(\eta)e^{2\pi x\cdot (\xi+\eta)}\,d\xi d\eta$$
which is a bilinear multiplier operator whose multiplier $\tilde{m}(\xi,\eta)=\frac{d}{dt}(\hat{\mu}(t\xi,t\eta))$ satisfies for $|\xi|\sim 2^{i}$, $|\eta|\sim 2^{j}$,
$$|\tilde{m}(\xi,\eta)|=|\nabla\hat{\mu}(t\xi,t\eta)\cdot (\xi,\eta)|\lesssim 2^{i}(1+2^{\min\{i,j\}}|\sin(\theta)|)^{-(d-2)/2}(1+2^{\max\{i,j\}})^{-(d-2)/2}.$$
The same strategy that lead to estimate (\ref{decayforsinglescaletriangleatt}) gives us
$$\left\|2^{-i}\frac{d}{dt}\mathcal{T}_t(f^i,g^j)\right\|_{L^2}\lesssim 2^{-i\frac{(d-4)}{2}}\|f\|_{L^{2}}\|g\|_{L^2}.$$
This is enough since it will imply for $i\geq j$
\begin{equation*}
    \begin{split}
        &\int |\mathcal{T}_E(f^i,g^j)(x)|^2 \,dx\\
        \leq& \sum_{I\in \mathcal{I}_{-i}} \left(\dashint_{\tilde{I}} \|\mathcal{T}_t(f^{i},g^j)(x)\|_{L^2_x}^2 \,dt+\dashint_{\tilde{I}} \|\mathcal{T}_t(f^{i},g^j)(x)\|_{L^2_x}\Big\|2^{-i}\frac{d}{dt}\mathcal{T}_t(f^i,g^j)\Big\|_{L^2_x}\,dt\right)\\
        \lesssim & N(E,2^{-i}) 2^{-i(d-4)}\|f\|_{L^{2}}^{2}\|g\|_{L^{2}}^{2}.
    \end{split}
\end{equation*}

\end{proof}

\section{Bilinear biparameter-like maximal operators with dilation sets \texorpdfstring{$E_1,\,E_2\subset [1,2] $}{Lg}}\label{sec: biparameter}

Since we are working in the bilinear setting, it is also natural to investigate boundedness properties of biparameter analogues of the operator $\mathcal{A}_{m,E}$ or its associated multi-scale maximal operator $\mathcal{M}_{m,E}$. Namely, let $m$ be a bilinear multiplier in $\R^d\times \R^d$ and let $E_1,\,E_2\subset [1,2]$ be two dilation sets. We define
\begin{equation}
\mathcal{A}_{m,E_1,E_2}^{(2)}(f,g)(x):=\sup_{t_1\in E_1,t_2\in E_2}|T_{m_{t_1,t_2}}(f,g)(x)|,
 \end{equation}
 where 
 \begin{equation}
  T_{m_{t_1,t_2}}(f,g)(x)=\int_{\R^{2d}}\hat{f}(\xi)\hat{g}(\eta)m(t_1\xi,t_2\eta)e^{2\pi i x\cdot(\xi+\eta)}\,d\xi d\eta,
 \end{equation}and its associated multi-scale maximal operator
 \begin{equation}
\mathcal{M}^{(2)}_{m,E_1,E_2}(f,g)(x):=\sup_{l\in \Z}\sup_{t_1\in E_1,\,t_2\in E_2}|T_{m_{2^lt_1,2^lt_2}}(f,g)(x)|.
\end{equation}

    We observe that in the particular case where $m(\xi,\eta)=\hat{\mu}(\xi,\eta)$, for a compactly supported finite Borel measure $\mu$ in $\R^{2d}$, then 
    $$T_{m_{t_1,t_2}}(f,g)(x)=\int f(x-t_1y)g(x-t_2z)\,d\mu(y,z).$$
    In this case the biparameter multi-scale maximal operator we are interested in takes the form
    \begin{equation}
\mathcal{M}_{\hat{\mu},E_1,E_2}^{(2)}(f,g)(x):=\sup_{l\in \Z}\sup_{t_1\in E_1,\,t_2\in E_2}\left|\int f(x-2^lt_1y) g(x-2^{l}t_2z)\,d\mu(y,z)\right|.
    \end{equation}
    Note that for $\mathcal{M}_{\hat{\mu},E_1,E_2}^{(2)}(f,g)$ the two averaging parameters for $f$ and $g$ are not the same but they vary at the same dyadic scale $2^l$. It would also be interesting to study the following larger bilinear biparameter operator given by 
$$\mathcal{M}^{bip}_{\hat{\mu},E_1,E_2}(f,g)(x):=\sup_{k,l\in \Z}\sup_{t_1\in E_1,\,t_2\in E_2}\left|\int f(x-2^kt_1y) g(x-2^{l}t_2z)\,d\mu(y,z)\right|,$$
    in which the averaging parameter for $f$ and $g$ can be any pair in $2^{\Z}E_1\times 2^{\Z}E_2$. For the particular case when $E_1=E_2=[1,2]$, these two different biparameter operators become
    $$\mathcal{M}_{\hat{\mu},[1,2],[1,2]}^{(2)}(f,g)(x)=\sup_{l\in \Z}\sup_{t,s\in [2^{l},2^{l+1}]}\left|\int f(x-ty) g(x-sz)\,d\mu(y,z)\right|$$
    and
    $$\mathcal{M}^{bip}_{\hat{\mu},[1,2],[1,2]}(f,g)(x)=\sup_{t_1,t_2>0}\left|\int f(x-t_1y) g(x-t_2z)\,d\mu(y,z)\right|.$$

 Inspired from our earlier approach, we would be interested in understanding the decay properties of the pieces $(\mathcal{A}^{(2)}_{m,E_1,E_2})^{i,j}$ of this operator, defined as
 $$(\mathcal{A}_{m,E_1,E_2}^{(2)})^{i,j}(f,g)(x):=\mathcal{A}_{m,E_1,E_2}^{(2)}(f^i,g^j)(x)$$
for $i,j\geq 0$, where $f^i,\, g^j$ are defined as in Subsection \ref{Eadmissiblemsubsection}. Our goal in this section is to derive Sobolev smoothing estimates and continuity estimates for $\mathcal{A}_{m,E_1, E_2}^{(2)}$, analogously as for the one-parameter operators. In the next section, such results will be applied to derive sparse bound for the multi-scale operator $\mathcal{M}_{\hat{\mu},E_1, E_2}^{(2)}$. A key difference between the two biparameter operators $\mathcal{M}^{(2)}_{m,E_1, E_2}$ and $\mathcal{M}^{bip}_{\hat{\mu},E_1,E_2}$ is that the sparse theory for the former is essentially one parameter. Once one has a good understanding of the biparameter behavior of its corresponding single-scale operator $\mathcal{A}^{(2)}_{\hat{\mu}, E_1, E_2}$, the sparse bound would follow via a similar argument as in the one-parameter case.

The theorem below can be seen as a biparameter version of Theorem \ref{piecesofA_E}.

 \begin{thm}[Decay bounds for pieces of $\mathcal{A}_{m,E_1,E_2}^{(2)}$]\label{thm:piecesbiparameter}
 Let $E_1$ and $E_2$ be two subsets of $[1,2]$ with upper Minkowski dimension $\text{dim}_{M}(E_i)=\beta_i,\,i=1,2$. Let $m$ be an $a$-admissible bilinear multiplier up to order $2$, and assume that $2a>d+\beta_1+\beta_2$. Then,
    \begin{equation*}
        \begin{split}
\left\|\mathcal{A}_{m,E_1,E_2}^{(2)}(f^i,g^j)\right\|_{L^2}&\lesssim N(E_1,2^{-i})^{\frac{1}{2}}N(E_2,2^{-j})^{\frac{1}{2}}2^{-\max\{i,j\}a}2^{\min\{i,j\}\frac{d}{2}}\|f\|_{L^2}\|g\|_{L^2}\\
            &\lesssim_{\eps}  2^{-\frac{1}{2}\max\{i,j\}(2a-d-\beta_1-\beta_2-\eps)}\|f\|_{L^2}\|g\|_{L^2}.
        \end{split}
    \end{equation*}
    \end{thm}

\begin{proof}[Proof of Theorem \ref{thm:piecesbiparameter}]
For $\nu\leq 0$, let $\mathcal{I}_{\nu}^{(k)}(E)$ be the family of all binary intervals $[n2^{\nu},(n+1)2^{\nu+1})$ that intersect $E_k$, $k=1,2$. Then we know that 
$$\#\mathcal{I}_{\nu}^{(k)}\lesssim N(E_k,2^{\nu})\lesssim_{\eps} 2^{-\nu(\beta_k+\eps)},\quad k=1,2.$$
One has 
\begin{equation}\label{inequalityforsupE1E2}
\begin{split}
|\mathcal{A}^{(2)}_{m,E_1,E_2}(f^i,g^j)(x)|^2\leq &\sum_{I\in \mathcal{I}^{(1)}_{-i}}\sum_{J\in \mathcal{I}^{(2)}_{-j}}\sup_{t_1\in I, t_2\in J}|T_{m_{t_1,t_2}}(f^i,g^j)|^2\\
=&\sum_{I\in \mathcal{I}^{(1)}_{-i}}\sum_{J\in \mathcal{I}^{(2)}_{-j}}\sup_{t_1\in I, t_2\in J}\alpha_I^{2}(t_1)\alpha_J^{2}(t_2)|T_{m_{t_1,t_2}}(f^i,g^j)|^2
\end{split}
\end{equation}
with $\alpha_I$ as defined in the proof of Theorem \ref{piecesofA_E}. 

By the Fundamental Theorem of Calculus and the fact that $\supp(\alpha_I)\subset 2I=[c_I-2^{-i},c_I+2^{-i})$, we have for any $t_1\in I$ and $t_2\in J$,
\begin{equation*}
    \begin{split}
&\alpha_I^2(t_1)\alpha_J^2(t_2)|T_{m_{t_1,t_2}}(f^i,g^j)(x)|^2\\
=&\int_{c_I-2^{-i}}^{t_1}\int_{c_J-2^{-j}}^{t_2}\Big\{2\alpha_I(s_1)\alpha_{I}'(s_1)2\alpha_J(s_2)\alpha_J'(s_2)|T_{m_{s_1,s_2}}(f,g)(x)|^2\\
&\quad\quad\quad\quad\quad+2\alpha_I(s_1)\alpha_{I}'(s_1)\alpha_J^2(s_2)\partial_{s_2}(|T_{m_{s_1,s_2}}(f,g)(x)|^2)\\
&\quad\quad\quad\quad\quad + \alpha_{I}^2(s_1)2\alpha_J(s_2)\alpha_J'(s_2)\partial_{s_1}(|T_{m_{s_1,s_2}}(f,g)(x)|^2)\\
&\quad\quad\quad\quad\quad + \alpha_{I}^2(s_1)\alpha_J^2(s_2)\partial_{s_1,s_2}(|T_{m_{s_1,s_2}}(f,g)(x)|^2)\Big\}\,ds_1 ds_2.
    \end{split}
\end{equation*}
Hence,
\begin{equation*}
\begin{split}
     &\sup_{t_1\in I, t_2\in J}|\alpha_I(t_1)\alpha_J(t_2)T_{m_{t_1,t_2}}(f^i,g^j)(x)|^2\\
     \lesssim &\int_{\tilde{I}}\int_{\tilde{J}}\Big\{|I|^{-1}|J|^{-1}|T_{m_{s_1,s_2}}(f,g)(x)|^2\\
&\quad\quad\quad\quad+|I|^{-1}|T_{m_{s_1,s_2}}(f,g)(x)||\partial_{s_2}T_{m_{s_1,s_2}}(f,g)(x)|\\
&\quad\quad\quad\quad+|J|^{-1}|T_{m_{s_1,s_2}}(f,g)(x)||\partial_{s_1}T_{m_{s_1,s_2}}(f,g)(x)|\\
&\quad\quad\quad\quad+|T_{m_{s_1,s_2}}(f,g)(x)||\partial_{s_1,s_2}(T_{m_{s_1,s_2}}(f,g)(x))|\Big\}\,ds_1 ds_2\\
\lesssim &\dashint_{\tilde{I}}\dashint_{\tilde{J}} \Big\{|T_{m_{s_1,s_2}}(f,g)(x)|^2 +|T_{m_{s_1,s_2}}(f,g)(x)||2^{-j}\partial_{s_2}T_{m_{s_1,s_2}}(f,g)(x)|\\
& \quad\quad\quad\quad+|T_{m_{s_1,s_2}}(f,g)(x)||2^{-i}\partial_{s_1}T_{m_{s_1,s_2}}(f,g)(x)|\\
&\quad\quad\quad\quad +|T_{m_{s_1,s_2}}(f,g)(x)||2^{-i}2^{-j}\partial_{s_1,s_2}(T_{m_{s_1,s_2}}(f,g)(x))|\Big\}\,ds_1ds_2.
\end{split}
\end{equation*}
Taking integrals in inequality (\ref{inequalityforsupE1E2}), and by Fubini and H\"older's inequality, 
\begin{equation*}
    \begin{split}
       \int |\mathcal{A}^{(2)}_{m,E_1,E_2}(f^i,g^j)&(x)|^2\,dx\\
       \lesssim \sum_{I\in \mathcal{I}^{(1)}_{-i}}\sum_{J\in \mathcal{I}^{(2)}_{-j}}&\int \sup_{t_1\in I, t_2\in J}|\alpha_I(t_1)\alpha_J(t_2)T_{m_{t_1,t_2}}(f^i,g^j)(x)|^2\,dx\\
       \lesssim\sum_{I\in \mathcal{I}^{(1)}_{-i}}\sum_{J\in \mathcal{I}^{(2)}_{-j}}&\dashint_{\tilde{I}}\dashint_{\tilde{J}}\Big\{ \|T_{m_{s_1,s_2}}(f^i,g^j)(x)\|_2^{2}\\
+&\|T_{m_{s_1,s_2}}(f^i,g^j)(x)\cdot 2^{-j}\partial_{s_2}T_{m_{s_1,s_2}}(f^i,g^j)(x)\|_1\\
+&\|T_{m_{s_1,s_2}}(f^i,g^j)(x)\cdot 2^{-i}\partial_{s_1}T_{m_{s_1,s_2}}(f^i,g^j)(x)\|_1\\
+&\|T_{m_{s_1,s_2}}(f^i,g^j)(x)\cdot 2^{-i}2^{-j}\partial_{s_1,s_2}(T_{m_{s_1,s_2}}(f^i,g^j)(x))\|_1\Big\}\,ds_1ds_2\\
\lesssim\sum_{I\in \mathcal{I}^{(1)}_{-i}}\sum_{J\in \mathcal{I}^{(2)}_{-j}}&\dashint_{\tilde{I}}\dashint_{\tilde{J}} \Big\{\|T_{m_{s_1,s_2}}(f^i,g^j)(x)\|_2^{2}\\
+&\|T_{m_{s_1,s_2}}(f^i,g^j)(x)\|_2\| 2^{-j}\partial_{s_2}T_{m_{s_1,s_2}}(f^i,g^j)(x)\|_2\\
+&\|T_{m_{s_1,s_2}}(f^i,g^j)(x)\|_2\| 2^{-i}\partial_{s_1}T_{m_{s_1,s_2}}(f^i,g^j)(x)\|_2\\
+&\|T_{m_{s_1,s_2}}(f,g)(x)\|_2 \|2^{-i}2^{-j}\partial_{s_1,s_2}(T_{m_{s_1,s_2}}(f^i,g^j)(x))\|_2\Big\}\,ds_1ds_2.\\
    \end{split}
\end{equation*}

Observe that if $m$ is an $a$-admissible multiplier up to order 2, then 
\begin{equation*}
    \begin{split}
        \|T_{m_{s_1,s_2}}(f^i,g^j)(x)\|_2&\lesssim 2^{-a\max\{i,j\}}2^{\min\{i,j\}\frac{d}{2}}\|f\|_2\|g\|_2,\\
    \|\partial_{s_2}T_{m_{s_1,s_2}}(f^i,g^j)(x)\|_2&\lesssim  2^{j}2^{-a\max\{i,j\}}2^{\min\{i,j\}\frac{d}{2}}\|f\|_2\|g\|_2,\\
        \| \partial_{s_1}T_{m_{s_1,s_2}}(f^i,g^j)(x)\|_2&\lesssim  2^{i}2^{-a\max\{i,j\}}2^{\min\{i,j\}\frac{d}{2}}\|f\|_2\|g\|_2,\\
        \| \partial_{s_1,s_2}T_{m_{s_1,s_2}}(f^i,g^j)(x)\|_2&\lesssim 2^{i+j} 2^{-a\max\{i,j\}}2^{\min\{i,j\}\frac{d}{2}}\|f\|_2\|g\|_2,
       \\  
    \end{split}
\end{equation*}
and that is all one needs to finish up the proof.

\end{proof}

With the same argument as the one used in Subsection \ref{Eadmissiblemsubsection} for $\mathcal{A}_{m,E}$, the following corollaries follow from the decay bounds for the pieces of $\mathcal{A}_{m,E_1,E_2}^{(2)}$ given by Theorem \ref{thm:piecesbiparameter}.

\begin{cor}[Sobolev bounds for $\mathcal{A}^{(2)}_{m,E_1,E_2}$ ]\label{biparametersobolev} Let $d\geq 1$ and $E_i\subset[1,2]$ with upper Minkowski dimension $\text{dim}_{M}(E_i)=\beta_i$, $i=1,2$. Let $m$ be an $a$-admissible bilinear multiplier up to order $2$, and assume $2a>d+\beta_1+\beta_2$. Given $s_1,s_2\geq 0$, with $s_1+s_2<\frac{2a-d-\beta_1-\beta_2}{2}$, then
\begin{equation}
      \|\mathcal{A}^{(2)}_{m,E_1,E_2}(f,g)\|_{L^2}\lesssim \|f\|_{H^{-s_1}}\|g\|_{H^{-s_2}}.
\end{equation}
\end{cor}

    \begin{cor}
        [Continuity estimates for $\mathcal{A}^{(2)}_{m,E_1,E_2}$]\label{continuitybiparameter} Let $d\geq 1$ and $E_i\subset [1,2]$ with upper Minkowski dimension $\text{dim}_{M}(E)=\beta_i$, $i=1,2$. Let $m$ be an $a$-admissible bilinear multiplier up to order $2$, and assume that $2a>d+\beta_1+\beta_2$. Then, there exists $\gamma>0$ such that
\begin{equation}
    \|\mathcal{A}_{m,E_1,E_2}^{(2)}(f-\tau_h f,g)\|_{L^2}+\|\mathcal{A}_{m,E_1,E_2}^{(2)}(f,g-\tau_h g)\|_{L^2}\lesssim |h|^{\gamma}\|f\|_{L^2}\|g\|_{L^2},\quad\forall |h|<1.
\end{equation}
Moreover, under the same hypothesis there exist $\gamma_1,\gamma_2>0$ such that
\begin{equation}
    \|\mathcal{A}_{m,E_1,E_2}^{(2)}(f-\tau_{h_1} f, g-\tau_{h_2}g)\|_{L^2}\lesssim |h_1|^{\gamma_1}|h_2|^{\gamma_2}\|f\|_{L^2} \|g\|_{L^2},\quad \forall |h_1|< 1, |h_2|<1.
\end{equation}
    \end{cor}

\section{Sparse bounds consequences for multi-scale bilinear maximal functions}\label{section:sparsebounds}

In this section, we consider a special class of multipliers $m(\xi, \eta)=\hat{\mu}(\xi,\eta)$, where $\mu$ is a compactly supported finite Borel measure on $\R^{2d}$ and assume that $(0,0)\notin \text{supp}(\mu)$, so that 
\begin{equation}\label{conditiononsuppofmu}
    \min\{|(y,z)|\colon (y,z)\in \text{supp}({\mu})\}>0.
\end{equation}

For a dilation set $E\subset [1,2]$, recall the definition of the multi-scale bilinear maximal operator associated to $\mathcal{A}_{\hat{\mu},E}$ given by
\begin{equation}
\begin{split}
     \mathcal{M}_{\hat{\mu},E}(f,g)(x)=&\sup_{l\in \Z}\sup_{t\in E}|T_{\hat{\mu}_{t2^l}}(f,g)(x)|\\
     =& \sup_{l\in \Z} \sup_{t\in E}\left|\int f(x-t2^{l} y)g(x-t2^{l}z)\,d\mu(y,z)\right|.
\end{split}
\end{equation}

In the particular case $E=\{1\}$ one gets the lacunary bilinear maximal operator associated to the bilinear multiplier $\hat{\mu}$ given by
\begin{equation}
\begin{split}
     \mathcal{M}_{\hat{\mu},lac}(f,g)(x)=&\sup_{l\in \Z} |T_{\hat{\mu}_{2^l}}(f,g)(x)|\\
     =& \sup_{l\in \Z} \left|\int f(x-2^{l} y)g(x-2^{l}z)\,d\mu(y,z)\right|.
\end{split}
\end{equation}

From the techniques in \cite{BFOPZ,sparsetriangle}, one can check the continuity estimates below are enough to get the sparse domination result claimed in Theorem \ref{sparseboundME}. The assumption that $m=\hat{\mu}$ for some compactly supported finite measure $\mu$ is not important for the continuity estimates below but it is crucial for the reduction to sparse bounds of dyadic maximal operators (see \cite[Section 4]{BFOPZ}). The assumption that $(0,0)\notin \text{supp}(\mu)$ implies condition (\ref{conditiononsuppofmu}), which we can use to get the analogue of \cite[Lemma 19]{BFOPZ} in our more general setting.

\begin{prop}\label{prop:generalconinuityAE}
    Let $E\subset[1,2]$ with upper Minkowski dimension $\beta$ and let $\mu$ be a compactly supported finite Borel measure in $\R^{2d}$. Assume that $\hat{\mu}$ is $a$-admissible with $2a>d+\beta$. Then the following continuity estimates hold for any point $(1/p,1/q,1/r)\in \text{int}(\mathcal{R}(\mu,E))$.\\
    There exists $\gamma>0$ such that
\begin{equation}
    \|\mathcal{A}_{\hat{\mu},E}(f-\tau_h f,g)\|_{L^r}+\|\mathcal{A}_{\hat{\mu},E}(f,g-\tau_h g)\|_{L^r}\lesssim |h|^{\gamma}\|f\|_{L^p}\|g\|_{L^q},\quad\forall\,|h|<1.
\end{equation}
Moreover, under the same hypothesis there exist $\gamma_1,\gamma_2>0$ such that
\begin{equation}
    \|\mathcal{A}_{\hat{\mu},E}(f-\tau_{h_1} f, g-\tau_{h_2}g)\|_{L^r}\lesssim |h_1|^{\gamma_1}|h_2|^{\gamma_2}\|f\|_{L^p} \|g\|_{L^q},\quad \forall\,|h_1|< 1,\,\forall\, |h_2|<1.
\end{equation}
\end{prop}

\begin{proof}
    We already know by Corollary \ref{continuityA_E} that $\mathcal{A}_{\hat{\mu},E}$ satisfy continuity estimates at $L^{2}\times L^{2}\rightarrow L^2$ if $2a>d+\beta$. For any $(1/p_0,1/q_0,1/r_0)\in \mathcal{R}(\mu,E)$, we have
    $$\|\mathcal{A}_{\hat{\mu},E}(f-\tau_h f,g)\|_{L^{r_0}}+\|\mathcal{A}_{\hat{\mu},E}(f,g-\tau_h g)\|_{L^r}\lesssim\|f\|_{L^{p_0}}\|g\|_{L^{q_0}},\quad \forall\,|h|<1$$
    and 
    \begin{equation*}
    \|\mathcal{A}_{\hat{\mu},E}(f-\tau_{h_1} f, g-\tau_{h_2}g)\|_{L^{r_0}}\lesssim \|f\|_{L^{p_0}} \|g\|_{L^{q_0}},\quad \forall\,|h_1|< 1,\,\forall\, |h_2|<1.
\end{equation*}
    By multilinear interpolation, one immediately gets continuity estimates for $\mathcal{A}_{\hat{\mu},E}$ at $L^{p}\times L^{q}\rightarrow L^{r}$ for any triple $(\frac{1}{p},\frac{1}{q},\frac{1}{r})$ in the interior of the region $\mathcal{R}(\mu,E)$.
\end{proof}

\begin{rem}\label{nicepropertiesforregion} Let $E\subset[1,2]$ with upper Minkowski dimension $\beta$ and let $\mu$ be a compactly supported finite Borel measure in $\R^{2d}$. Assume that for all $1<p<\infty$,
\begin{equation}\label{firstpropertyregion}
\mathcal{A}_{\hat{\mu},E}:L^{p}\times L^{q}\rightarrow L^{p} 
\end{equation}
for all sufficiently large $q$ and symmetrically, for all $1<q<\infty$
\begin{equation}\label{secondpropertyregion}
\mathcal{A}_{\hat{\mu},E}:L^{p}\times L^{q}\rightarrow L^{q}
\end{equation}
for all sufficiently large $p$.
Then as shown in \cite{BFOPZ} the restriction $p,q\leq r$ in the Theorem \ref{sparseboundME} can be removed. This is the case for $\mu=\sigma_{2d-1}$ for example, as one can check from the description of the boundedness region of $\tilde{\mathcal{M}}=\mathcal{A}_{\hat{\sigma},[1,2]}$, initiated in \cite{JL} and sharpened in \cite{bhojak2023sharp}. 
\end{rem}

Let $\mathcal{T}_E$ and $\mathcal{T}_{E}^{*}$ be the single-scale and multi-scale triangle averaging maximal operators whose definition is given in (\ref{TEdefinition}) and (\ref{multiscaleTEdef}) respectively. Also, let $\mathcal{R}(\mathcal{T}_E)$ be the $L^{p}$ improving boundedness region of $\mathcal{T}_{E}$, namely  
\begin{equation}\label{regionTE}
\begin{split}
    \mathcal{R}(\mathcal{T}_E):=\{(1/p,1/q,1/r)&\colon 1\leq p,q\leq \infty, r>0,\,  1/p+1/q\geq 1/r \\
    &\text{ and }\mathcal{T}_{E}:L^{p}\rightarrow L^q\rightarrow L^r\text{ is bounded}\}.
\end{split}  
\end{equation}

As in the proof of Proposition \ref{prop:generalconinuityAE}, by interpolating the continuity estimates for $\mathcal{T}(E)$ at $L^{2}\times L^{2}\rightarrow L^2$ given by Corollary \ref{continuityTE} with its $L^{p}$ improving bounds, we get the following proposition. 

\begin{prop}
    Let $E\subset[1,2]$ with upper Minkowski dimension $\beta$ and let $d>4+\beta$. Then the following continuity estimates hold for any point $(1/p,1/q,1/r)\in \text{int}(\mathcal{R}(\mathcal{T}_{E}))$.\\
    There exists $\gamma>0$ such that
\begin{equation}
    \|\mathcal{T}_{E}(f-\tau_h f,g)\|_{L^r}+\|\mathcal{T}_{E}(f,g-\tau_h g)\|_{L^r}\lesssim |h|^{\gamma}\|f\|_{L^p}\|g\|_{L^q},\quad\forall\,|h|<1.
\end{equation}
Moreover, under the same hypothesis there exist $\gamma_1,\gamma_2>0$ such that
\begin{equation}
    \|\mathcal{T}_{E}(f-\tau_{h_1} f, g-\tau_{h_2}g)\|_{L^r}\lesssim |h_1|^{\gamma_1}|h_2|^{\gamma_2}\|f\|_{L^p} \|g\|_{L^q},\quad \forall\,|h_1|< 1,\,\forall\, |h_2|<1.
\end{equation}
\end{prop}

Again by the arguments in \cite{sparsetriangle} and \cite{BFOPZ} we have the following corollary.

\begin{cor}[Sparse bounds for $\mathcal{T}_{E}^{*}$]\label{sparseboundTE}
Let $E\subset[1,2]$ with upper Minkowski dimension $\beta$, and let $d> 4+\beta$. Then, for any $(1/p,1/q,1/r)\in\text{int}(\mathcal{R}(\mathcal{T}_E))$ with $r>1$ and $p,q\leq r$, one has sparse domination for $\mathcal{T}^{*}_{E}$ with parameters $(p,q,r')$. That is, for any $f,g,h\in C^{\infty}_{0}(\R^d)$, there exists a sparse family $\mathcal{S}$ such that 
\begin{equation}
    \begin{split}
|\langle\mathcal{T}^{*}_{E}(f,g),h\rangle|\lesssim \sum_{Q\in \mathcal{S}} \langle f \rangle_{Q,p} \langle g\rangle_{Q,q}  \langle h\rangle_{Q,r'} |Q|.
    \end{split}
\end{equation}
\end{cor}

\begin{rem}\label{rmk: PS}
    Observe that in the corollary above we only claim sparse bounds for sufficiently large dimensions $d\geq 5$ in the case $\beta\in [0,1)$ and $d\geq 6$ in the case $\beta=1$, due to the constraints in our continuity estimates Corollary \ref{continuityTE}. Although a version of \cite{sparsetriangle} claimed continuity estimates and sparse bounds in all dimensions $d\ge 2$ for the case $E=\{1\}$ and $E=[1,2]$, there is unfortunately an error in the proof (more precisely in the proof of their continuity estimates). In personal communication with the authors of \cite{sparsetriangle}, we have learned that an upcoming amended version of their paper will prove continuity estimates, hence sparse bounds, for the triangle averaging operator in the case $E=\{1\}$ and $E=[1,2]$ in high enough dimensions, in contrast to all dimensions $d\geq 2$ as originally claimed. Our result improves on their new dimensional restrictions by lowering the threshold to $d\ge 5$ in the lacunary case $E=\{1\}$ and $d\geq 6$ in the full case $E=[1,2]$, and is new in other cases of $E$. In dimensions $d=2, 3, 4$, continuity estimates, and in turn sparse domination for the maximal triangle operator, remain open questions.
\end{rem}

Lastly, we can also define the $L^{p}$ improving region associated to the biparameter-like bilinear maximal operator $\mathcal{A}^{(2)}_{\hat{\mu},E_1,E_2}$,
\begin{equation}
\begin{split}
    \mathcal{R}^{(2)}({\mu},E_1,E_2)
    =\{(1/p,1/q,1/r)&\colon 1\leq p,q\leq \infty, r>0,\,1/p+1/q\geq 1/r\\
    \text{ and }&\mathcal{A}^{(2)}_{\hat{\mu},E_1,E_2}:L^{p}\times L^{q}\rightarrow L^r \text{ is bounded}\}
\end{split}
\end{equation}

Recall that in Corollary \ref{continuitybiparameter} we proved continuity estimates for the biparameter like single-scale operator $\mathcal{A}^{(2)}_{\hat{\mu},E_1,E_2}$ at $L^{2}\times L^{2}\rightarrow L^{2}$. By multilinear interpolation continuity estimates hold at $L^{p}\times L^{q}\rightarrow L^r$ for any point $(1/p,1/q,1/r)\in \text{int}(\mathcal{R}^{(2)}({\mu},E_1,E_2))$ and the following sparse bounds corollary follows.

\begin{cor}[Sparse bounds for $\mathcal{M}^{(2)}_{\hat{\mu},E_1,E_2}$]\label{cor:sparseforbiparameter}  Let $E_1,E_2\subset[1,2]$ with upper Minkowski dimensions $\beta_1$ and $\beta_2$ respectively. Suppose $m$ is an $a$-admissible with $2a>d+\beta_1+\beta_2$.
 Then, for any $(1/p,1/q,1/r)\in \text{int}(\mathcal{R}^{(2)}(\hat{\mu},E_1,E_2))$ with $r>1$ and $p,q\leq r$, one has sparse domination for $\mathcal{M}^{(2)}_{\hat{\mu},E_1,E_2}$ with parameters $(p,q,r')$. Namely, for any $f,g,h\in C^{\infty}_{0}(\R^d)$, there exists a sparse family $\mathcal{S}$ such that 
\begin{equation}
    \begin{split}
|\langle\mathcal{M}^{(2)}_{\hat{\mu},E_1,E_2}(f,g),h\rangle|\lesssim \sum_{Q\in \mathcal{S}} \langle f \rangle_{Q,p} \langle g\rangle_{Q,q}  \langle h\rangle_{Q,r'} |Q|.
    \end{split}
\end{equation}
\end{cor}

\section{Some Lebesgue  bounds for \texorpdfstring{$\mathcal{A}_{\hat{\mu},E}$}{Lg}, \texorpdfstring{$\mathcal{M}_{\hat{\mu},E}$}{Lg} and their biparameter variants}\label{sec:somelebesguebounds}

As in the previous section we will restrict ourselves to the case $m(\xi,\eta)=\hat{\mu}(\xi,\eta)$ where $\mu$ is a compactly supported finite measure in $\R^{2d}$. Throughout the section we assume $\mu$ is $a$-admissible and $E\subset [1,2]$ is a dilation set with upper Minkowski dimension $\beta$.

\subsection{H\"older bounds for \texorpdfstring{$\mathcal{M}_{\hat{\mu},E}$}{Lg} in the Banach case.}

In this subsection we prove some H\"older bounds $L^{p}\times L^{q}\rightarrow L^{r}$ with $1/p+1/q=1/r$ and $1<p,q,r\leq \infty$. We do so by making use of the known bounds for linear versions of such maximal operators. The proof is inspired from that of Proposition 3 in \cite{BGHHO}, where they deal with the case $E=[1,2]$.

\begin{prop}[H\"older Banach bounds for $\mathcal{M}_{\hat{\mu},E}$ and its biparameter variants] \label{banachholderbounds}If $2a\geq d+\beta$, then
\begin{equation}
\mathcal{M}_{\hat{\mu},E}:\,L^{p}\times L^{q}\rightarrow L^{r}
\end{equation}
for all $1<p,q\leq \infty$ and $1/r=1/p+1/q<1$.
Moreover, if $2a\ge d+\max\{\beta_1,\beta_2\}$ then the same H\"older Banach bounds hold true for the biparameter variants $\mathcal{M}^{(2)}_{\hat\mu, E_1,E_2}$ and $\mathcal{M}^{bip}_{\hat\mu, E_1, E_2}$. 
\end{prop}

\begin{proof}
    We first observe that for any $1<p\leq \infty$ one has  
$$\mathcal{M}_{\hat{\mu},E}:\,L^{p}\times L^{\infty}\rightarrow L^{p}.$$
To see this, first observe that the case $p=\infty$ is trivially true. Next, let us assume $1<p<\infty$. For $g\in L^{\infty}$,
\begin{equation*}
    \mathcal{M}_{\hat{\mu},E}(f,g)(x)\leq  \|g\|_{\infty} \sup_{l\in \Z}\sup_{t\in E} \int |f(x-2^{l}ty)|\,d\mu (y,z).
\end{equation*}
For any $t\in E$, and $l\in \Z$, by using that the function constant $1$ in $\R^d$, has Fourier transform $\hat{1}(\eta)=\delta_0(\eta)$, we get
\begin{equation*}
\begin{split}
        \int |f(x-t2^ly)|\,d\mu(y,z) =& \int |f(x-t2^ly)|1(x-t2^lz)\,d\mu(y,z) \\
        =&\int_{\R^{2d}} \widehat{|f|}(\xi)\delta_{0}(\eta)\hat{\mu}(t2^{l}\xi,t2^{l}\eta)e^{2\pi ix\cdot (\xi+\eta)}\,d\xi d\eta\\
        \leq&\int_{\R^{d}} \widehat{|f|}(\xi)\hat{\mu}(t2^{l}\xi,0)e^{2\pi i x\cdot \xi}\,d\xi. 
\end{split}
 \end{equation*}

Consider the sublinear operator
$$f\mapsto M^{0}_{E}f(x):=\sup_{l\in \Z}\sup_{t\in E} \left|\int_{\R^{d}}\hat{f}(\xi)\hat{\mu}(t2^{l}\xi,0)e^{2\pi i x\cdot \xi}\,d\xi\right|. $$
Observe that $M_{E}^{0}$ is bounded on $L^{p}$ for all $1<p<\infty$. That is indeed the case for $2a\geq d+\beta$ by an application of Theorem B in \cite{DuoanVargas} for the multiplier $m(\xi)=\hat{\mu}(\xi,0)$ which has decay $|\partial^{\alpha}m(\xi)|\lesssim_{\alpha} (1+|\xi|)^{-a}$ for any multi-index $\alpha$. 

That is enough to get the claimed bound since 
$$\|\mathcal{M}_{\hat{\mu},E}(f,g)\|_{L^{p}}\lesssim \|g\|_{L^{\infty}} \|M_E^{0}(|f|)\|_{L^p}\lesssim \|g\|_{L^{\infty}}\|f\|_{L^{p}}.$$

Symmetrically one can also show 
$$\mathcal{M}_{\hat{\mu},E}:L^{\infty}\times L^{q}\rightarrow L^{q},$$
for all $1<q\leq \infty$.
Hence by interpolation $\mathcal{M}_{\hat{\mu},E}$ satisfies H\"older bounds of the form 
$$\mathcal{M}_{\hat{\mu},E}:L^{p}\times L^{q}\rightarrow L^{r}$$
for all $1<p,q\leq \infty$ and $1/r=1/p+1/q<1$.

The proof above implies our claim for the biparameter case since one will have
$$\mathcal{M}_{\hat{\mu},E_1,E_2}^{bip}:\,L^{p}\times L^{\infty}\rightarrow L^{p}$$
for any $1<p\leq \infty$ as long as $2a\geq d+\beta_1$, and 
$$\mathcal{M}_{\hat{\mu},E_1,E_2}^{bip}:\,L^{\infty}\times L^{q}\rightarrow L^{q}$$
for any $1<q\leq \infty$ as long as $2a\geq d+\beta_2$.
\end{proof}

\subsection{Some \texorpdfstring{$L^{p}$}{Lg} improving bounds for the single-scale \texorpdfstring{$\mathcal{A}_{\hat{\mu},E}$}{Lg}}
 The goal of this subsection is to give a partial description of the boundedness region $\mathcal{R}(\mu,E)$ of the operator $\mathcal{A}_{\hat\mu,E}$. We already know from Proposition \ref{banachholderbounds} and Theorem \ref{sobolevE} that if $\hat{\mu}$ is $a$-admissible with $a$ sufficiently large, namely $2a>d+\beta$, then the region $\mathcal{R}(\mu,E)$ contains at least the interior of the convex closure of the points $(0,0,0), (0,1,1), (1,0,1)$ and $(1/2,1/2,1/2)$. Since we actually have Sobolev smoothing estimates at $L^{2}\times L^{2}\rightarrow L^{2}$, we can do better, by using Sobolev embedding theorems which will lead to $L^{p}$ improving bounds of the form $\mathcal{A}_{\hat{\mu},E}:L^{p}\times L^{q}\rightarrow L^{2}$ with $p,q<2$. In fact, the same bound holds for the general multiplier $m$, not only the case $m=\hat{\mu}$. We state and prove the result below in this generality.

\begin{prop}[Bounds for $\mathcal{A}_{m,E}$ and $\mathcal{A}_{m,E_1,E_2}^{(2)}$ with target space $L^2$]\label{prop:LpimprovingforAEtargetL2}
    Suppose $m$ is a-admissible up to order $1$ with $2a>d+\beta$. If $2a \leq 2d+\beta$, then
\begin{equation}\label{Lpimprovingbounds}
      \mathcal{A}_{m,E}:\, L^{p}\times L^{q}\rightarrow L^{2}  
    \end{equation}
    for $\frac{2d}{2a-\beta}<p,q\leq 2$ with $\frac{1}{p}+\frac{1}{q}<\frac{1}{2}+\frac{2a-\beta}{2d}$. If $2a>2d+\beta$,  (\ref{Lpimprovingbounds}) holds for $1<p,q\leq 2$ and $1/p+1/q<\frac{3}{2}$.

    Moreover, if $m$ is $a$-admissible up to order $2$ with $2a>d+\beta_1+\beta_2$ and $2a\leq 2d+{\beta_1}+{\beta_2}$, then 
$$\mathcal{A}_{m,E_1,E_2}^{(2)}:\,L^{p}\times L^{q}\rightarrow L^{r}$$
for $\frac{2d}{2a-\beta_1-\beta_2}<p,q\leq 2$ with $\frac{1}{p}+\frac{1}{q}<\frac{1}{2}+\frac{2a-\beta_1-\beta_2}{2d}$. While if $2a> 2d+{\beta_1}+{\beta_2}$, then 
$$\mathcal{A}_{m,E_1,E_2}^{(2)}:\,L^{p}\times L^{q}\rightarrow L^{r}$$
for $1<p,q\leq 2$ with $\frac{1}{p}+\frac{1}{q}<\frac{3}{2}$.
\end{prop}

\begin{figure}[h]
\begin{center}
         \scalebox{1}{
\begin{tikzpicture}
%nodes
\draw (-0.3,3.8) node {$\frac{1}{q}$};
\draw (3.8,0) node {$\frac{1}{p}$};
\draw (3,-0.4) node {$1$};
\draw (-0.4,3) node {$1$};
\draw (1.5,-0.5) node{$\frac{1}{2}$};
\draw (-0.5,1.5) node{$\frac{1}{2}$};
\draw (-0.5,2.5) node{$\frac{2a-\beta}{2d}$};
\draw (2.5,-0.5) node{$\frac{2a-\beta}{2d}$};
\draw[purple] (3.4,2.3) node {\small{$\frac{1}{p}+\frac{1}{q}=\frac{1}{2}+\frac{2a-\beta}{2d}$}};
%drawing axis
\draw[->,line width=1pt] (-0.2,0)--(3.5,0);
\draw[->,line width=1pt] (0,-0.2)--(0,3.5);
%drawing small segments for labels
\draw[-,line width=0.75pt] (3,-0.1)--(3,0.1);
\draw[-,line width=0.75pt] (-0.1,3)--(0.1,3);
\draw[-,line width=0.75pt] (1.5,-0.1)--(1.5,0.1);
\draw[-,line width=0.75pt] (-0.1,1.5)--(0.1,1.5);
\draw[-,line width=0.75pt] (-0.1,2.5)--(0.1,2.5);
\draw[-,line width=0.75pt] (2.5,-0.1)--(2.5,0.1);
%dashed lines
\draw[dashed] (3,0)--(3,3)--(0,3);
\draw[purple,dashed] (2.5,0)--(2.5,1.5);
\draw[purple,dashed] (0,2.5)--(1.5,2.5);
\draw[purple,dashed] (1.5,0)--(1.5,1.5);
\draw[purple,dashed] (0,1.5)--(1.5,1.5);

%drawing triangle region
\filldraw[purple!25!white] (1.5,2.5)--(1.5,1.5)--(2.5,1.5)--(1.5,2.5);
\draw[purple, line width=1pt] (1.5,2.5)--(1.5,1.5)--(2.5,1.5);
\draw[purple, dashed, line width=0.75pt] (2.5,1.5)--(1.5,2.5); 

%drawing points
\draw[purple,line width=0.75pt] (2.5,1.5) circle (2pt);
\draw[purple,line width=0.75pt] (1.5,2.5) circle (2pt);
\fill[purple] (1.5,1.5) circle (2pt);
\fill[white] (2.5,1.5) circle (2pt);
\fill[white] (1.5,2.5) circle (2pt);

%drawing blue segment
%\draw[blue,line width=0.75pt] (1.5,0)--(0,1.5);
%\filldraw[blue] (1.5,0) circle (2pt);
%\filldraw[blue] (0,1.5) circle (2pt);

\end{tikzpicture}}
\end{center}
\caption{Region for pairs $(1/p,1/q)$ such that $\|\mathcal{A}_{m,E}\|_
{L^{p}\times L^{q}\rightarrow L^{2}}<\infty$ given by Proposition \ref{prop:LpimprovingforAEtargetL2} in the case $a\leq d+\frac{\beta}{2}$. %and segment in blue obtained from proposition \ref{banachholderbounds}.
}
\end{figure}

\begin{rem}\label{remarkAEregion}
In the case $m=\hat\mu$, by Proposition \ref{banachholderbounds} and the trivial inequality $\mathcal{A}_{\hat{\mu},E}(f,g)(x)\leq \mathcal{M}_{\hat{\mu},E}(f,g)(x)$, we also know that $\mathcal{A}_{\hat{\mu},E}:L^{2}\times L^{\infty}\rightarrow L^{2}$ and $\mathcal{A}_{\hat{\mu},E}:L^{\infty}\times L^{2}\rightarrow L^{2}$. In the case $a\leq d+\frac{\beta}{2}$, interpolation implies $\mathcal{A}_{\hat{\mu},E}:L^{p}\times L^{q}\rightarrow L^{2}$ for any $(\frac{1}{p},\frac{1}{q})$ in the interior of the convex closure of the points $(\frac{1}{2},0),\,(0,\frac{1}{2}),\, (\frac{2a-\beta}{2d},\frac{1}{2})$ and $(\frac{1}{2},\frac{2a-\beta}{2d})$ and in the closed segment connecting $(\frac{1}{2},0)$ and $(0,\frac{1}{2})$.       
\end{rem}

\begin{figure}[h]
\begin{center}
         \scalebox{1}{
\begin{tikzpicture}
%nodes
\draw (-0.3,3.8) node {$\frac{1}{q}$};
\draw (3.8,0) node {$\frac{1}{p}$};
\draw (3,-0.4) node {$1$};
\draw (-0.4,3) node {$1$};
\draw (1.5,-0.5) node{$\frac{1}{2}$};
\draw (-0.5,1.5) node{$\frac{1}{2}$};
\draw (-0.5,2.5) node{$\frac{2a-\beta}{2d}$};
\draw (2.5,-0.5) node{$\frac{2a-\beta}{2d}$};
\draw[purple] (3.4,2.3) node {\small{$\frac{1}{p}+\frac{1}{q}=\frac{1}{2}+\frac{2a-\beta}{2d}$}};
%drawing axis
\draw[->,line width=1pt] (-0.2,0)--(3.5,0);
\draw[->,line width=1pt] (0,-0.2)--(0,3.5);
%drawing small segments for labels
\draw[-,line width=0.75pt] (3,-0.1)--(3,0.1);
\draw[-,line width=0.75pt] (-0.1,3)--(0.1,3);
\draw[-,line width=0.75pt] (1.5,-0.1)--(1.5,0.1);
\draw[-,line width=0.75pt] (-0.1,1.5)--(0.1,1.5);
\draw[-,line width=0.75pt] (-0.1,2.5)--(0.1,2.5);
\draw[-,line width=0.75pt] (2.5,-0.1)--(2.5,0.1);
%dashed lines
\draw[dashed] (3,0)--(3,3)--(0,3);
\draw[purple,dashed] (2.5,0)--(2.5,1.5);
\draw[purple,dashed] (0,2.5)--(1.5,2.5);
\draw[purple,dashed] (1.5,0)--(1.5,1.5);
\draw[purple,dashed] (0,1.5)--(1.5,1.5);

%drawing region
\filldraw[blue!20!white](1.5,0)--(2.5,1.5)--(1.5,2.5)--(0,1.5);
\filldraw[purple!25!white] (1.5,2.5)--(1.5,1.5)--(2.5,1.5)--(1.5,2.5);
%\draw[purple, line width=1pt] (1.5,2.5)--(1.5,1.5)--(2.5,1.5);
\draw[purple, dashed, line width=0.75pt] (2.5,1.5)--(1.5,2.5); 
\draw[blue, dashed, line width=0.75pt] (1.5,0)--(2.5,1.5); 
\draw[blue, dashed, line width=0.75pt] (0,1.5)--(1.5,2.5); 

%drawing points
\draw[purple,line width=0.75pt] (2.5,1.5) circle (2pt);
\draw[purple,line width=0.75pt] (1.5,2.5) circle (2pt);
%\fill[purple] (1.5,1.5) circle (2pt);
\fill[white] (2.5,1.5) circle (2pt);
\fill[white] (1.5,2.5) circle (2pt);

%drawing blue segment
\draw[blue,line width=0.75pt] (1.5,0)--(0,1.5);
\filldraw[blue] (1.5,0) circle (2pt);
\filldraw[blue] (0,1.5) circle (2pt);

\end{tikzpicture}}
\end{center}
\caption{Colored region represents the pairs $(1/p,1/q)$ such that $\|\mathcal{A}_{\hat{\mu},E}\|_
{L^{p}\times L^{q}\rightarrow L^{2}}<\infty$ given by Remark \ref{remarkAEregion} in the case $a\leq d+\frac{\beta}{2}$. 
}
\end{figure}
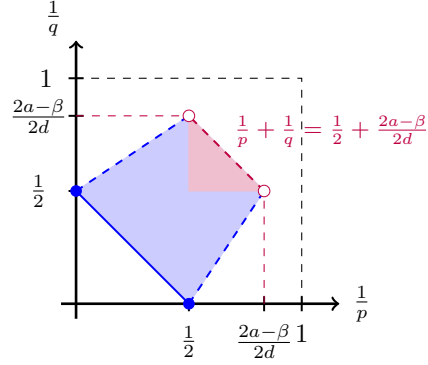

\begin{proof}[Proof of Proposition \ref{prop:LpimprovingforAEtargetL2}] The second claim for $\mathcal{A}_{m,E}$ follows from the first by setting $a=d+\beta/2$ in the first claim, so let us assume $a\leq d+\frac{\beta}{2}$.

 By the duality between $H^{s}$ and $H^{-s}$, we have 
 $$\|f\|_{H^{-s}}=\sup\{\langle f,g\rangle_{L^{2}}\colon\, g\in \mathcal{S}(\R^d),\|g\|_{H^s}\leq 1\}.$$
For $0<s<d/2$, Sobolev embedding implies that $H^{s}$ embedds in $L^{q_s}$ for $q_s=\frac{2d}{d-2s}$ (see Theorem 1.3.5 in \cite{ModernFAgrafakos}, for example). Hence, for $0<s<d/2$ and $1/q_s+1/q_s'=1$,
\begin{equation}\label{sobolevembedding}\|f\|_{H^{-s}}\lesssim \|f\|_{L^{q_s'}}=\|f\|_{L^{p_s}}, \quad p_s:=\frac{2d}{d+2s}.
\end{equation}

Since $2a\leq 2d+{\beta}$ one has $\frac{2a-d-\beta}{2}\leq \frac{d}{2}$. So from Theorem \ref{sobolevE} and (\ref{sobolevembedding}), we get for any $0\leq s<\frac{2a-d-\beta}{2}$,
\begin{equation}
    \begin{split}
\|\mathcal{A}_{m,E}(f,g)\|_{L^{2}}&\lesssim \|f\|_{H^{-s}}\|g\|_{L^{2}}\\
&\lesssim \|f\|_{L^{p_s}}\|g\|_{L^{2}}
    \end{split}
\end{equation}
for $p_s=\frac{2d}{d+2s}$. Therefore,
\begin{equation}\label{firstsobolevembeddingconsequence}
\mathcal{A}_{m,E}:L^{p}\times L^{2}\rightarrow L^{2}
\end{equation}
for any $\frac{2d}{2a-\beta}<p\leq 2$. Similarly, Theorem \ref{sobolevE} and inequality (\ref{sobolevembedding}) also imply 
\begin{equation}\label{secondsoboleembeddingconsequence}
\mathcal{A}_{m,E}:L^{2}\times L^{q}\rightarrow L^{2}\,\text{ for }\frac{2d}{2a-\beta}<q\leq 2.
\end{equation}
Interpolating (\ref{firstsobolevembeddingconsequence}) and (\ref{secondsoboleembeddingconsequence}) we get the claimed $L^{p}$ improving bounds.

In the biparameter case, we are again reduced to just checking the case $2a\leq 2d+\beta_1+\beta_2$. For any $0\leq s<\frac{2a-d-\beta_1-\beta_2}{2}$, it follows from Corollary \ref{biparametersobolev} combined with the Sobolev embedding in (\ref{sobolevembedding}) that 
\begin{equation}
\|\mathcal{A}_{m,E_1,E_2}^{(2)}(f,g)\|_{L^{2}}\lesssim \|f\|_{H^{-s}}\|g\|_{L^2}\lesssim \|f\|_{L^{p_s}}\|g\|_{L^2}
\end{equation}
for $p_s=\frac{2d}{d+2s}$.
That means that 
$$\|\mathcal{A}_{m,E_1,E_2}^{(2)}\|_{L^{p}\times L^{2}\rightarrow L^2}<\infty$$
for $\frac{2d}{2a-\beta_1-\beta_2}<p\leq 2$, and similarly,
$$\|\mathcal{A}_{m,E_1,E_2}^{(2)}\|_{L^{2}\times L^{q}\rightarrow L^2}<\infty$$
for $\frac{2d}{2a-\beta_1-\beta_2}<q\leq 2$. The claimed bounds then follow from interpolation. 
\end{proof}

\subsection{H\"older bounds for \texorpdfstring{$\mathcal{M}_{\hat{\mu},E}$}{Lg} from the sparse bounds}
In this subsection we finally prove Theorem \ref{Lebesgueboundsmultiscale}. It will a consequence of the Lebesgue bounds that have been obtained so far in the previous two sections and the Lebesgue H\"older bounds for $\mathcal{M}_{\hat{\mu},E}$ that can be obtained as a consequence of the sparse bounds for $\mathcal{M}_{\hat{\mu},E}$ given by Corollary \ref{sparseboundME}. Note that, as usual, the sparse bounds in fact imply stronger weighted norm inequalities for multilinear $A_p$ weights. Since the deduction is standard, we omit those corollaries in this article, while only discuss the unweighted case for comparison with earlier results in this section. 

\begin{proof}[Proof of Theorem \ref{Lebesgueboundsmultiscale}]
    Let $(1/p_0,1/q_0)$ be a point in the interior of the triangle determined by the points $(\frac{1}{2},\frac{1}{2}),\,(\frac{2a-\beta}{2d},\frac{1}{2})$ and $(\frac{1}{2},\frac{2a-\beta}{2d})$. 
    Observe that $(1/p_0,1/q_0,1/2)$ is a point in $\text{int}(\mathcal{R}(\mu,E))$ with $r_0:=2>1$ and $1/p_0,1/q_0>1/2=1/r_0$, according to Proposition \ref{prop:LpimprovingforAEtargetL2}. So by Corollary \ref{sparseboundME} we know that $\mathcal{M}_{\hat{\mu}, E}$ satisfies sparse bounds with parameters $(p_0,q_0,2)$. By \cite[Proposition 1.2]{CDPO} it follows that $\mathcal{M}_{\hat{\mu},E}:L^{p}\times L^{q}\rightarrow L^{r}$ is bounded for any  $0\leq 1/p< 1/p_0,\,0\leq 1/q< 1/q_0$ with $\max\{1/p,1/q\}>0$ and $1/r=1/p+1/q$. By varying over all $(1/p_0,1/q_0)$ we get H\"older boundedness for $\mathcal{M}_{\hat{\mu},E}$ for $(1/p,1/q)$ satifying $0\leq 1/p,1/q<\frac{2a-\beta}{2d}$ and $0<1/p+1/q<1/2+\frac{2a-\beta}{2d}$. We also recall that from Proposition \ref{banachholderbounds} we have H\"older bounds for any $0\leq 1/p+1/q<1$, which by interpolation leads to the claimed bounds.
\end{proof}

\begin{rem}
    The proof above adapts to get H\"older Lebesgue bounds for $\mathcal{M}_{\hat{\mu},E_1,E_2}^{(2)}$. Suppose that $2a>d+\beta_1+\beta_2$, and that we are in the case $2a\leq 2d+\beta_1+\beta_2$. By using Proposition \ref{prop:LpimprovingforAEtargetL2} for $\mathcal{A}_{\hat{\mu},E_1,E_2}^{(2)}$ combined with Corollary \ref{cor:sparseforbiparameter}, we have $(p_0,q_0,2)$ sparse bounds for $\mathcal{M}_{\hat{\mu},E_1,E_2}^{(2)}$ for any $(1/p_0,1/q_0)$ in the interior of the triangle determined by $(\frac{1}{2},\frac{1}{2}),\,(\frac{2a-\beta_1-\beta_2}{2d},\frac{1}{2})$ and $(\frac{1}{2},\frac{2a-\beta_1-\beta_2}{2d})$. That, combined with \cite[Proposition 1.2]{CDPO} and the Banach bounds from Proposition \ref{banachholderbounds} lead to H\"older Lebesgue bounds $\mathcal{M}_{\hat{\mu},E_1,E_2}^{(2)}:L^{p}\times L^{q}\rightarrow L^{\frac{pq}{p+q}}$ for any $(1/p,1/q)$ in interior of the convex closure of the points
    $$(1,0),\,(0,0),\,(0,1),\,\left(\frac{1}{2},\frac{2a-\beta_1-\beta_2}{2d}\right),\,\text{ and }\left(\frac{2a-\beta_1-\beta_2}{2d},\frac{1}{2}\right).$$
\end{rem}

\subsection{Lifting \texorpdfstring{$L^{p}$}{Lg} improving bounds for \texorpdfstring{$\mathcal{A}_{\hat{\mu},E}$}{Lg} to quasi-Banach H\"older bounds}

Before wrapping up the section, we prove some results that further enlarge the known region of $L^p$ improving bounds for the single-scale operator $\mathcal{A}_{\hat{\mu},E}$. These bounds are not needed in the proof of Theorem \ref{Lebesgueboundsmultiscale}, hence are presented separately here. The sharp boundedness region for $\mathcal{A}_{\hat\mu, E}$, even in the case $\mu=\sigma_{2d-1}$, remains an interesting open question. We will derive further information on the region by studying some examples in the next section.

Through a minor variant of an argument in \cite[Proposition 4.1]{IPS}, for all the single-scale bilinear operators of the form $\mathcal{A}_{\hat{\mu},E}$, where $\mu$ is a compactly supported finite surface measure, we can deduce bounds of the form $L^{p}\times L^{q}\rightarrow L^{r}$, $1/r=1/r_H(p,q):=1/p+1/q$, whenever we know an $L^{p}$ improving bound of the form $L^p\times L^q \rightarrow L^{r_0}$,  with $1/p+1/q\geq 1/r_0$ and $r_H(p,q)<1$. 

The idea is that, due to the local nature of this single-scale maximal operators (in the sense that $\mathcal{A}_{\hat{\mu},E}(f,g)(x)$ only depends on values of $f,g$ at points in a fixed-size neighborhood of $x$), we can morally reduce to being on a domain of finite measure and use the fact that $L^{r_0}$ embeds into $L^{r(p,q)}$ for such a domain.

\begin{prop}\label{liftingtoholderbound}
    Let $A$ be a maximal bilinear Radon transform of the form $\mathcal{A}_{\hat{\mu},E}$ or $\mathcal{A}_{\hat{\mu},E_1,E_2}^{(2)}$ associated to a finite measure $\mu$ compactly supported in a ball of radius $K>1$, and whose dilation sets are subsets of $[1,2]$. If $A:\,L^p\times L^q\rightarrow L^{r_0}$ continuously, with $1/r_0\leq 1/p+1/q$ and $1/p+1/q>1$ then $A:\,L^p\times L^q\rightarrow L^r$ continuously, for $1/r=1/p+1/q$ with the operator norm controlled by $C(K)\Vert A\Vert_{L^p\times L^q\rightarrow L^{r_0}}$.
\end{prop}
\begin{proof}
    We can tile Euclidean space by unit cubes $\{Q_l\}_{l\in\Z^d}$ and let $S$ denote the set of integer lattice points in the box centered at the origin with sidelength  $10K$ in Euclidean space, i.e. $S=\Z^{d}\cap [-5K,5K]^{d}$. By sub-linearity and by the compact support assumptions for $\mu$ we get
    \begin{align*}
        \Vert A(f,g)\Vert_{L^r}&\leq  \left(\int \left( \sum_{s\in S} \sum_{l\in \BZ^d} A(f\chi_{Q_l},g\chi_{Q_l+s})(x)\right)^{r}\,dx\right)^{1/r} \\
        &\lesssim_K\sum_{s\in S} \left(\int\left|\sum_{l\in\Z^d}A(f\chi_{Q_l},g\chi_{Q_{l}+s})(x)\right|^{r}\,dx\right)^{1/r},
    \end{align*}
    because $S$ is finite and $L^{r}$ is a quasi-Banach space. Now using that $r<1$,
\begin{align*}
\sum_{s\in S} &\left(\int\left|\sum_{l\in \Z^d}A(f\chi_{Q_l},g\chi_{Q_{l}+s})(x)\right|^{r}\,dx\right)^{1/r}\\
\leq&\sum_{s\in S} \left(\sum_{l\in \Z^d}\int|A(f\chi_{Q_l},g\chi_{Q_{l}+s})(x)|^r \,dx\right)^{1/r}.
    \end{align*}

    Notice that each term in the sum is an integral of a function whose support has measure at most some universal constant $C(K)$. Thus, we can replace by the $L^{r_0}$ norm at the cost of just a multiplicative constant. Thus,
    \begin{align*}
        \Vert A(f,g)\Vert_{L^r}&\le C(K) \sum_{s\in S} \left(\sum_{l\in \BZ^d}\Vert A(f\chi_{Q_l},g\chi_{Q_l+s})\Vert_{L^{r_0}}^{r}\right)^{1/r} \\
        &\le C(K)\Vert A\Vert_{L^p\times L^q\rightarrow L^{r_0}} \sum_{s\in S}\left(\sum_{l\in \BZ^d} \Vert f \chi_{Q_l}\Vert_{L^p}^{r} \Vert g\chi_{Q_l+s}\Vert_{L^q}^{r}\right)^{1/r} \\
        &\le C(K)\Vert A\Vert_{L^p\times L^q\rightarrow L^{r_0}} \|f\|_{L^p}\|g\|_{L^q},
    \end{align*}
    where the last step follows from the fact that $\#S\lesssim K^{d}$ and H\"older's inequality, since $\|f\|_p^{p}=\sum_{l\in \BZ^d} \Vert f \chi_{Q_l}\Vert_{L^{p}}^{p}$.
\end{proof}

In Proposition \ref{prop:LpimprovingforAEtargetL2} we obtained some $L^{p}$ improving bounds for $\mathcal{A}_{\hat{\mu},E}:L^{p}\times L^{q}\rightarrow L^{2}$, which satisfy $1/p+1/q>1$. Combining that with the lifting to H\"older bounds given by Proposition \ref{liftingtoholderbound} and interpolation, we conclude the following. 

\begin{prop}[$L^p$ improving bounds for $\mathcal{A}_{\hat{\mu},E}$]\label{lpimprovingsinglescale}
Suppose $\hat{\mu}$ is a-admissible with $2a>d+\beta$. If $a \leq d+\beta/2$, then
\begin{equation}
      \mathcal{A}_{\hat{\mu},E}:\, L^{p}\times L^{q}\rightarrow L^{r}  
    \end{equation}
    for $\frac{2d}{2a-\beta}<p,q\leq 2$ with $\frac{1}{p}+\frac{1}{q}<\frac{1}{2}+\frac{2a-\beta}{2d}$ and any $2\geq r\geq r_H(p,q)$ where $r_H=r_H(p,q)$ is given by the H\"older relation $1/r_H=1/p+1/q$.
    
Moreover, if $a>d+\beta/2$, the hypotheses on $p,q$ become $1<p,q\leq 2$ and $1/p+1/q<3/2$.
\end{prop}

An analogue of this result obviously holds for the biparameter operator $\mathcal{A}^{(2)}_{\hat\mu,E_1,E_2}$ as well, as a consequence of Proposition \ref{prop:LpimprovingforAEtargetL2} and Proposition \ref{liftingtoholderbound}. We leave the exact range of boundedness exponents to the interested reader.

\begin{rem}
In the particular case of $E=\{1\}$  and $\mu=\mu_S$ where $S$ is a compact smooth surface in $\R^{2d}$ with $k>d$ nonvanishing principal curvatures, by lifting $L^{p}$  improving bounds with target space $L^{1}$, \cite{choleeshuin} proved in Proposition 1.2 H\"older bounds $\mathcal{A}_{\hat{\mu}_S,1}:L^{p}\times L^{q}\rightarrow L^{r_H(p,q)}$ for $1\leq p,q\leq 2$, $3/2\leq 1/p+1/q=1/r_H<1+k/2d$. For comparison, the proposition above gives H\"older bounds for $1/p,1/q\in [1/2,1)$  with $1/p+1/q=1/r_H<1/2+k/2d$. These two bounds are not directly comparable to each other. However, in this particular case ($E=\{1\}$, $\hat\mu=\hat\mu_S$), one can improve the result in \cite{choleeshuin} by interpolating it with our earlier Proposition \ref{banachholderbounds} to derive the better bounds for $\mathcal{A}_{\hat\mu_S,1}$ for all H\"older exponents $(p,q,r_H(p,q))$ with $p,q>1$ and $1/p+1/q<1+k/2d$. In particular, in general we do not expect the region obtained by Proposition \ref{lpimprovingsinglescale} above to be sharp, but it does give us nontrivial interesting $L^p$ improving bounds for a large class of bilinear multipliers. 
\end{rem}

\section{Some necessary conditions for the boundedness of \texorpdfstring{$\mathcal{A}_E$}{Lg}}\label{necessaryEsinglescale}

In this section, we will focus on the case $\mu=\sigma_{2d-1}$, i.e., the normalized spherical measure in the unit sphere $S^{2d-1}\subset \R^{2d}$. We extend to the bilinear setting some examples from the works \cite{AHRS} and \cite{BFOPZ} to derive necessary conditions for boundedness of the single-scale bilinear maximal operator $\mathcal{A}_{E}:=\mathcal{A}_{\hat{\sigma}_{2d-1},E}$ in terms of the (upper) Minkowski and Assouad dimensions of the fractal dilation set $E\subset [1,2]$. This seems to suggest that, similar to the linear case, a complete description of the $L^p$ improving region of $\mathcal{A}_E$ may depend on both the upper Minkowski and Assouad dimensions of the set $E$.

In general, all of the examples fall into the same regime: we choose $f,g$ to be indicator functions of suitable regions (balls, boxes, or annuli) with dimensions depending on a parameter $\delta$. Then, we can argue that $\mathcal{A}_E(f,g)(x)$ has a lower bound (in terms of $\delta$) on some region whose measure we understand. Testing the bound
\begin{equation}\label{bound test nec}
    \Vert \mathcal{A}_E(f,g)\Vert_r\lesssim \Vert f\Vert_p \Vert g\Vert_q
\end{equation}
and sending $\delta$ to 0 gives necessary relations among the exponents $p,q,r$. For convenience, we will denote $\A(r,s)=B(0,s)\setminus {B(0,r)}=\{x\in \R^d\colon r\leq |x|<s\}$ for annuli centered at the origin.

\begin{prop}\label{firstnecessary}
Suppose $E$ has upper Minkowski dimension $\beta$. If $\mathcal{A}_E:L^p\times L^q\rightarrow L^r$ continuously then we have
\begin{equation}
    \frac{1}{p}+\frac{1}{q}\le\frac{2d-1}{d}+\frac{1-\beta}{dr}.
\end{equation}
\end{prop}
\begin{proof}
Let $f=g=\chi_{B(0,C\delta)}$ be indicator functions of balls where $C$ is a sufficiently large universal constant, such as $C=100$. By the slicing formula, we have that
\[
\mathcal{A}_E(f,g)(x)=\sup_{t\in E}\int_{B^{d}(0,1)} g(x-ty)\left(\int_{S^{d-1}} f(x-t\sqrt{1-|y|^2}z)\,d\sigma_{d-1}(z)\right)(1-|y|^2)^{(d-2)/2}\, dy,
\]
and we will denote this quantity by $k(x)$ for simplicity. First, we illustrate the case $E=\{1\}$ and set $t=1$. For any $x$ such that $||x|-\frac{1}{\sqrt2}|\le \delta$, we claim that we have a lower bound $k(x)\gtrsim \delta^{2d-1}$. Fix such an $x$. We get a lower bound to $k(x)$ by restricting the integration in $y$ to a small delta ball around $x$, say $B_{x}^{\delta}:=\{y\in \R^d\colon|y-x|<\delta \}$. Observe that implies that $|y|$ is $2\delta$ close to $1/\sqrt{2}$ so we have that $\big| |y|-\sqrt{1-|y|^{2}}\big|\le 10\delta$.
Also, for any $y\in B_x^{\delta}$ one has $(1-|y|^2)^{\frac{d-2}{2}}$ is comparable to $1$ and $g(x-y)=1$, so
$$k(x)\gtrsim \int_{B_x^{\delta}} \left(\int_{S^{d-1}} f(x-\sqrt{1-|y|^2}z)d\sigma_{d-1}(z)\right) dy $$
The inner integrand contributes about $\delta^{d-1}$ for such values of $y$ because we can guarantee that $f(x-t\sqrt{1-|y|^{2}}z)=1$ for all $z\in S^{d-1}$ that is $c\delta$-close to $\frac{x}{|x|}\in S^{d-1}$ and a geodesic ball of radius $\delta$ inside $S^{d-1}$ has measure comparable to $\delta^{d-1}$. Thus, $k(x)\gtrsim |B_x^{\delta}|\delta^{d-1}\sim\delta^{2d-1}$ for $x$ on an annulus of measure about $\delta$ (see Figure \ref{picannball1} for illustration).

Now, we consider varying $t\in E$ for a general dilation set $E\subset [1,2]$. Instead of a single annulus $\{x\in \R^d\colon |\sqrt{2}|x|-1|<\sqrt{2}\delta\}$, we get about $N(E,\delta)$ many annuli on which the lower bound $k(x)\gtrsim \delta^{2d-1}$ holds, namely if $\{I_i\}_{i=1}^{N(E,\delta)}$ is a minimal cover of $E$ with $\delta$ closed intervals then we get the lower bound in the set 
$$\{x\in \R^{d}\colon \sqrt{2}|x|\in\sqcup_{i=1}^{N(E,\delta)}I_i \}$$
by using that if $\sqrt{2}|x|\in I_i$, one can take $t_i\in I_i\cap E$ to compute
$$k(x)\geq \mathcal{A}_{t_i}(f,g)(x)\gtrsim \int_{y\colon |y-\frac{x}{t_i}|<\delta} \left(\int_{S^{d-1}} f(x-t_i\sqrt{1-|y|^2}z)d\sigma(z)\right)dy\gtrsim \delta^{d-1}|B_{x/t_i}^{\delta}|$$
so we have the lower bound on a set of measure about $N(E,\delta)\delta$. We put this into the bound \eqref{bound test nec} to deduce
\begin{equation}\label{first ex with N delta}
    \delta^{2d-1}\delta^{1/r}N(E,\delta)^{1/r}\le\delta^{d/p}\delta^{d/q}.
\end{equation}
Recalling the definition of upper Minkowski dimension, for any $\epsilon>0$ there exists a sequence of $\delta>0$ converging to zero such that
\[
N(E,\delta)\ge \delta^{-\beta+\epsilon} 
\]
Now, comparing the exponents in \eqref{first ex with N delta} gives the claim.
\end{proof}

\begin{rem}
    In the case of the multi-scale maximal operator $\mathcal{M}_E$, we recall that the H\"older condition $1/p+1/q=1/r$ is necessary by scaling. Since $\mathcal{A}_E(f,g)\leq \mathcal{M}_E(f,g)$, by replacing $1/p+1/q$ with $1/r$ in the necessary condition condition for $\mathcal{A}_E$ given by Proposition \ref{firstnecessary} we get that the condition 
    \begin{equation}\label{nec no equality}
    \frac{1}{r}\leq \frac{2d-1}{d-1+\beta}.
    \end{equation}
    is necessary for $\mathcal{M}_E$ to be $L^p\times L^q\to L^r$ bounded.
\end{rem}

The geometry behind the previous proof in the simplest case of $E=\{1\}$ is given in Figure \ref{picannball1}.

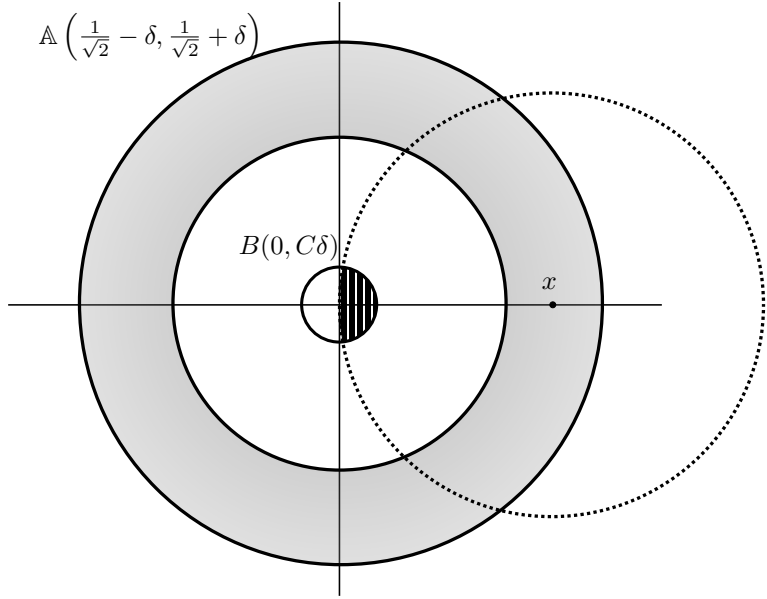
\begin{figure}[htbp]
\centering
\def\svgwidth{.7\textwidth}
%% Creator: Inkscape 1.3 (0e150ed, 2023-07-21), www.inkscape.org
%% Accompanies image file 'ball annulus example text 2.pdf' (pdf, eps, ps)
\begingroup%
  \makeatletter%
  \providecommand\color[2][]{%
    \errmessage{(Inkscape) Color is used for the text in Inkscape, but the package 'color.sty' is not loaded}%
    \renewcommand\color[2][]{}%
  }%
  \providecommand\transparent[1]{%
    \errmessage{(Inkscape) Transparency is used (non-zero) for the text in Inkscape, but the package 'transparent.sty' is not loaded}%
    \renewcommand\transparent[1]{}%
  }%
  \providecommand\rotatebox[2]{#2}%
  \newcommand*\fsize{\dimexpr\f@size pt\relax}%
  \newcommand*\lineheight[1]{\fontsize{\fsize}{#1\fsize}\selectfont}%
  \ifx\svgwidth\undefined%
    \setlength{\unitlength}{635.18671471bp}%
    \ifx\svgscale\undefined%
      \relax%
    \else%
      \setlength{\unitlength}{\unitlength * \real{\svgscale}}%
    \fi%
  \else%
    \setlength{\unitlength}{\svgwidth}%
  \fi%
  \global\let\svgwidth\undefined%
  \global\let\svgscale\undefined%
  \makeatother%
  \begin{picture}(1,0.78581058)%
    \lineheight{1}%
    \setlength\tabcolsep{0pt}%
    \put(0,0){\includegraphics[width=\unitlength,page=1]{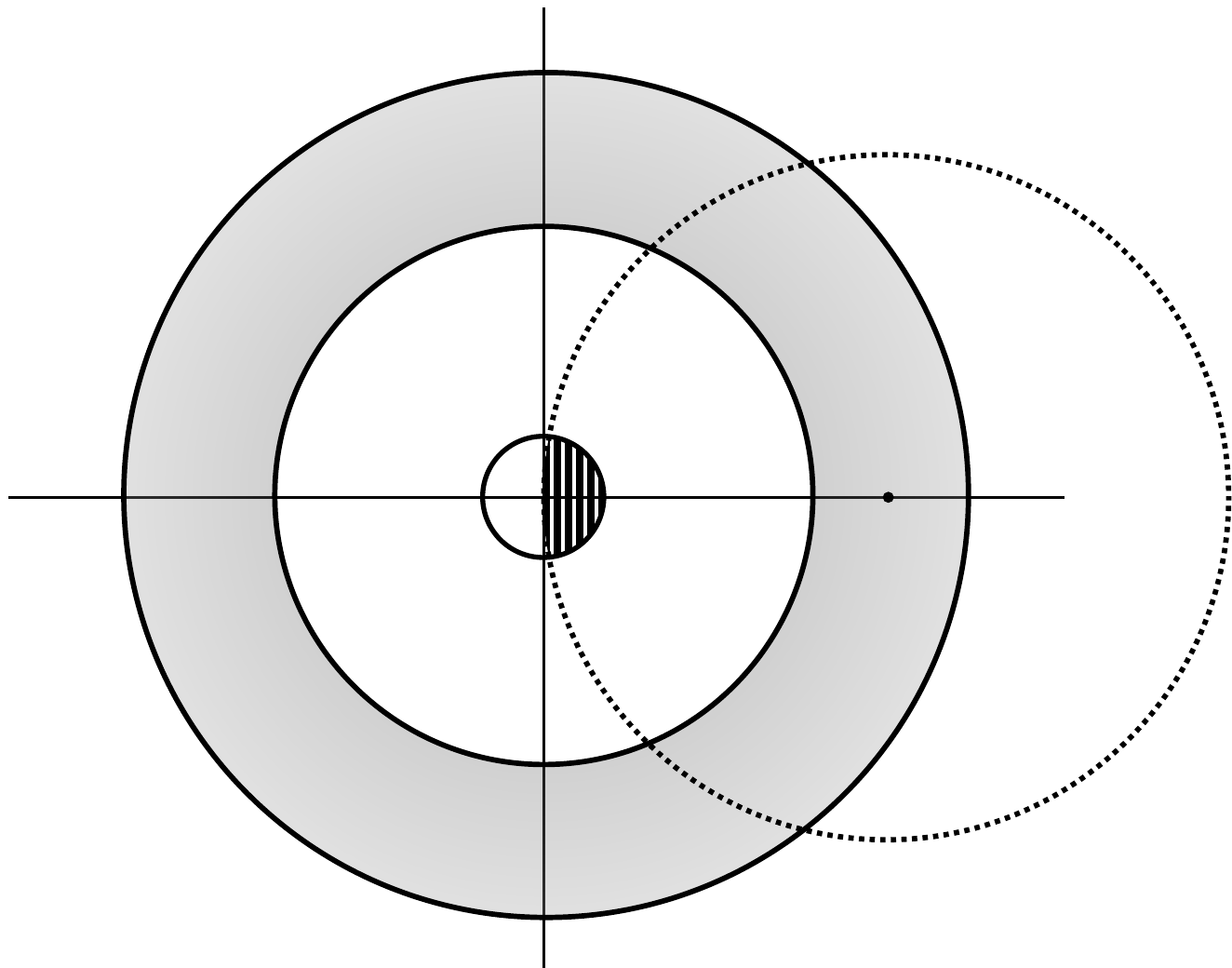}}%
    \put(0.70684817,0.4029436){\color[rgb]{0,0,0}\makebox(0,0)[lt]{\lineheight{1.25}\smash{\begin{tabular}[t]{l}$x$\end{tabular}}}}%
    \put(0.30970556,0.44840464){\color[rgb]{0,0,0}\makebox(0,0)[lt]{\lineheight{1.25}\smash{\begin{tabular}[t]{l}$B(0,C\delta)$\end{tabular}}}}%
    \put(0.04716007,0.72506499){\color[rgb]{0,0,0}\makebox(0,0)[lt]{\lineheight{1.25}\smash{\begin{tabular}[t]{l}$\mathbb{A}\left(\frac{1}{\sqrt2}-\delta,\frac{1}{\sqrt2}+\delta\right)$\end{tabular}}}}%
  \end{picture}%
\endgroup%

\caption{In the case $E=\{1\}$, $y$ lives in a $\delta$ ball around $x$ so that $x-y$ lies in $B(0,C\delta)$. For such a fixed $y$, the dashed circle illustrates the sphere along which $x-\sqrt{1-|y|^2 }z$ varies when $z$ varies over $S^{d-1}$. This intersects $B(0,C\delta)$ for $z$ in a $\delta$ cap around $\frac{x}{|x|}\in S^{d-1}$, which has spherical measure about $\delta^{d-1}$.}
\label{picannball1}
\end{figure}

The idea is that when the slicing formula is applied to indicator functions, it computes the measure of intersections of spheres and balls centered at $x$ with the sets in the indicator functions. The figure demonstrates the $t=1$ case, but as described in the proof, the general case would consist of more concentric brown annuli consisting of $x$ values where $k(x)$ has a lower bound $\delta^{2d-1}$.

\begin{prop}\label{secondnecessary}
If $\mathcal{A}_E:L^p\times L^q\rightarrow L^r$ continuously then we have
\begin{equation}
    \frac{1}{p}+\frac{1}{q}\le 1+\frac{d}{r}.
\end{equation}
\end{prop}
\begin{proof} 
By the equality $\mathcal{A}_t(f,g)(x)=\mathcal{A}_1(f(t\,\cdot),g(t\,\cdot))(\frac{x}{t})$, it is clear that $\|\mathcal{A}_t\|_{L^p\times L^q\rightarrow L^{r}}<\infty$ if and only if $\|\mathcal{A}_1\|_{L^p\times L^q\rightarrow L^{r}}<\infty$. Take any point $t_E\in E$. Since $\mathcal{A}_{t_E}\leq \mathcal{A}_{E}$ any necessary condition for $\mathcal{A}_{t_E}$ it is also necessary for $\mathcal{A}_E$. The fact that $1/p+1/q\leq 1+{d}/{r}$ is necessary for the boundedness of $\mathcal{A}_{1}$ is written out in detail in \cite[Proposition 3.3]{JL}. Although that proposition only considers the case $E=[1,2]$, their argument derives a lower bound for $\tilde{\mathcal{M}}=\mathcal{A}_{[1,2]}$ by replacing the supremum over $t\in [1,2]$ by $t=1$ and thus applies equally well to the single scale average $\mathcal{A}_1$.
\end{proof}

\begin{prop}\label{thirdnecessary}
Suppose $E$ has upper Minkowski dimension $\beta$. If $\mathcal{A}_E:L^p\times L^q\rightarrow L^r$ continuously then we have
\begin{equation}\label{inequalitythirdnec}
    \frac{1}{p}+\frac{1}{q}\le\frac{2d}{d+1}+\frac{1}{r}\left(1-\frac{2\beta}{(d+1)}\right)
\end{equation}
\end{prop}
\begin{proof} The case $E=[1,2]$ was obtained in \cite{BFOPZ}. One can get this more general condition by adapting their Knapp-type example. In their example they define 
 \[
R_1=[-C_1\sqrt{\delta}, C_1\sqrt{\delta}]^{d-1}\times [-C_1\delta, C_1\delta],\quad  R_2=[-C_2\sqrt{\delta}, C_2\sqrt{\delta}]^{d-1}\times [-C_2\delta, C_2\delta],
\]
\[R_3=[-\sqrt{\delta},\sqrt{\delta}]^{d-1}\times [\frac{1}{\sqrt{2}},\sqrt{2}],
\]where say $C_1, C_2=100$, and $\delta$ is sufficiently small. For $f_{\delta}=\chi_{R_1}$ and $g_{\delta}=\chi_{R_2}$ they showed that 
$$\tilde{\mathcal{M}}(f_{\delta},g_{\delta})(x)\gtrsim \delta^{d}$$
for all $x\in R_3=\{x=(x',x_d)\in \R^{d-1}\times \R\colon x'\in [-\sqrt{\delta},\sqrt{\delta}]^{d-1}\text{ and }\sqrt{2}x_d\in [1,2]\}$.

For a more general $E$ we replace $R_3$ with 
\[R_3^{E}:=\{x=(x',x_d)\in \R^{d-1}\times \R\colon x'\in [-\sqrt{\delta},\sqrt{\delta}]^{d-1}\text{ and }\sqrt{2}x_d\in \sqcup_{i=1}^{N(E,\delta)} I_i\},\]
where $\{I_i\colon 1\leq i\leq N(E,\delta)\}$ is a minimal collection of closed intervals of length $|I_i|=\delta$ covering $E$. For each $i$ take $t_i\in E\cap I_i$. Then for $x\in R_{3}^{E} $ with $\sqrt{2}x_d\in I_i$, one can check that  
$$\mathcal{A}_E(f_{\delta},g_{\delta})(x)\geq \mathcal{A}_{t_i}(f_{\delta},g_{\delta})(x)\gtrsim \delta^{\frac{d-1}{2}}\delta \delta^{\frac{d-1}{2}}=\delta^{d}.$$
That combined with $|R_3^{E}|\sim \delta^{\frac{d-1}{2}}N(E,\delta)\delta$ is enough to finish the computation.
One can also get this proposition as a corollary of Proposition \ref{fourthnecessary} in the case $\beta=\gamma$ since the upper Assouad dimension of $E$ will be at least $\beta$. 
\end{proof}

\begin{figure}[htbp]
\centering
\def\svgwidth{.4\textwidth}
%% Creator: Inkscape 1.3 (0e150ed, 2023-07-21), www.inkscape.org
%% Accompanies image file 'Knapp example.pdf' (pdf, eps, ps)
\begingroup%
  \makeatletter%
  \providecommand\color[2][]{%
    \errmessage{(Inkscape) Color is used for the text in Inkscape, but the package 'color.sty' is not loaded}%
    \renewcommand\color[2][]{}%
  }%
  \providecommand\transparent[1]{%
    \errmessage{(Inkscape) Transparency is used (non-zero) for the text in Inkscape, but the package 'transparent.sty' is not loaded}%
    \renewcommand\transparent[1]{}%
  }%
  \providecommand\rotatebox[2]{#2}%
  \newcommand*\fsize{\dimexpr\f@size pt\relax}%
  \newcommand*\lineheight[1]{\fontsize{\fsize}{#1\fsize}\selectfont}%
  \ifx\svgwidth\undefined%
    \setlength{\unitlength}{559.67897911bp}%
    \ifx\svgscale\undefined%
      \relax%
    \else%
      \setlength{\unitlength}{\unitlength * \real{\svgscale}}%
    \fi%
  \else%
    \setlength{\unitlength}{\svgwidth}%
  \fi%
  \global\let\svgwidth\undefined%
  \global\let\svgscale\undefined%
  \makeatother%
  \begin{picture}(1,0.84873735)%
    \lineheight{1}%
    \setlength\tabcolsep{0pt}%
    \put(0,0){\includegraphics[width=\unitlength,page=1]{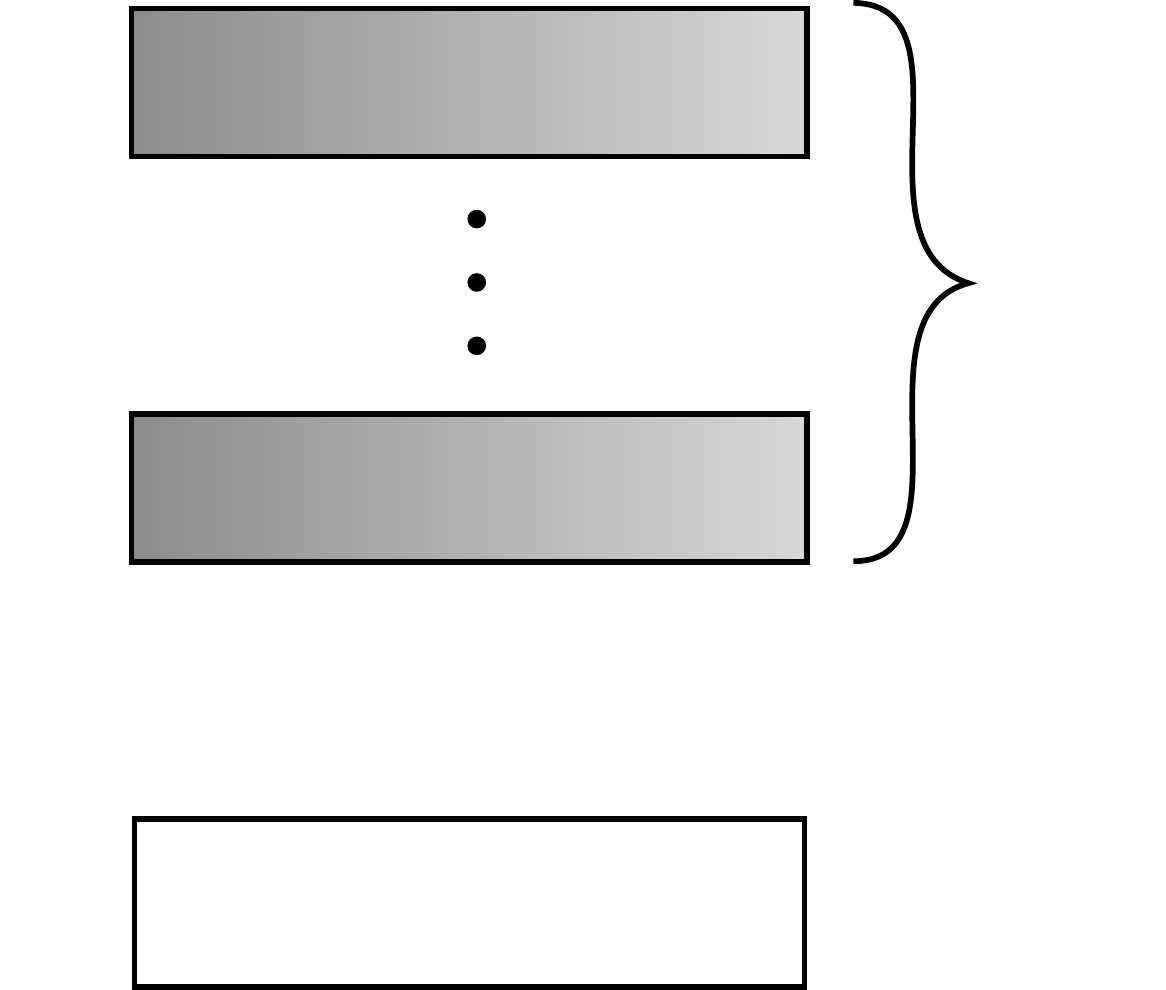}}%
    \put(0.88826599,0.62106386){\color[rgb]{0,0,0}\makebox(0,0)[lt]{\lineheight{1.25}\smash{\begin{tabular}[t]{l}$N(E,\delta)$\\copies\end{tabular}}}}%
    \put(0,0){\includegraphics[width=\unitlength,page=2]{Knapp-example.pdf}}%
    \put(0.43479409,0.43926641){\color[rgb]{0,0,0}\makebox(0,0)[lt]{\lineheight{1.25}\smash{\begin{tabular}[t]{l}$x$\end{tabular}}}}%
  \end{picture}%
\endgroup%

\caption{Picture illustrating the proof of Proposition \ref{thirdnecessary}. The gray region represents the region $R_{3}^{E}$ for $x$ where the lower bound for $\mathcal{A}_E(f_{\delta},g_{\delta})(x)$ holds. The empty rectangle represents the support of $f_{\delta}$ or $g_{\delta}$, and the dashed circle represents the spherical averages that will show up in the slicing. The key is that for $x$ in a shaded rectangle, we can find a radius in our dilation set so that dashed circle intersects a substantial part of the empty rectangle.}
\label{picknapp}
\end{figure}

If all we know is that $\text{dim}(E)=\beta$, by combining the Proposition \ref{firstnecessary},\,\ref{secondnecessary} and \ref{thirdnecessary} we get that for
\begin{equation}
    \begin{split}
    \|\mathcal{A}_E\|_{L^{p}\times L^{q}\rightarrow L^{r}}<\infty
    \end{split}
\end{equation}
it is necessary that
\begin{equation}\label{necessaryforbetaequalgamma}
    \begin{split}
\frac{1}{p}+\frac{1}{q}\leq 1+\min\left\{\frac{d-1}{d}+\frac{1-\beta}{dr},\,\frac{d-1}{d+1}\left(1+\frac{1}{r}\frac{d+1-2\beta}{d-1}\right),\,\frac{d}{r}\right\}.
\end{split}
\end{equation}

In the case $\beta=1$ this coincides with the necessary conditions obtained in \cite{BFOPZ} for the case $E=[1,2]$ which were proved to be sufficient in \cite{bhojak2023sharp} (up to some pieces of the boundary). To get more refined necessary conditions we will need to take into account the Assouad dimension of $E$, which already played an important role in \cite{AHRS}. Let $\gamma=\text{dim}_{A}(E)$ be the Assouad dimension of $E$. One can check that when $\gamma>\beta$ the necessary condition given in Proposition \ref{fourthnecessary} is more restrictive than the one above and they coincide when $\gamma=\beta$.

\begin{prop}\label{fourthnecessary}
Suppose $E\subset [1,2]$ is an Assouad regular set with (upper) Minkowski dimension $\beta$ and (upper) Assouad dimension $\gamma$. Let $\alpha=\beta/\gamma$ and let $C_{d,\alpha}^{-1}=1+\alpha(d-1)/2$. If $\mathcal{A}_E:L^p\times L^q\rightarrow L^r$ continuously then we have
\begin{equation}
    \frac{1}{p}+\frac{1}{q}\le C_{d,\alpha}\left\{ \alpha(d-1)+1+\frac{1}{r}\left(d-\beta-\frac{\alpha(d-1)}{2}\right) \right\}.
\end{equation}
That is, 
$$\frac{1}{p}+\frac{1}{q}\leq 1+\frac{\beta(d-1)}{\beta(d-1)+2\gamma}+\left(\frac{(d-\beta)2\gamma-(d-1)\beta}{\beta(d-1)+2\gamma}\right)\frac{1}{r}$$
In particular, when $\beta=\gamma$, one recovers inequality (\ref{inequalitythirdnec}).
\end{prop}

\begin{proof}
From the definition of Assouad spectrum and the fact that $E$ is Assouad regular, for any $\epsilon>0$ there exists $\delta>0$ (which can be shrunk smaller if desired) such that we can find an interval $I\subset [1,2]$ of length $\delta^{1-\beta/\gamma}$ with the property $N(E\cap I,\delta)\ge (|I|/\delta)^{\gamma-\epsilon}$. For convenience, we will denote $\alpha=\beta/\gamma\leq 1$ and $\sigma=\delta^{\alpha/2}\ge \delta^{1/2}$. Let $h_{\delta,I}$ be the indicator function of the set
\[
\{(y',y_d)\in \BR^d: \left||y|-2^{-1/2}r\right|\le \delta, |y'|\le \sigma\}
\]
where $r$ is the left endpoint of $I$. Then we have the estimate
\[
\Vert h_{\delta,I}\Vert_p \approx (\sigma^{d-1}\delta)^{1/p}
\]
Cover $E\cap I$ by a minimal collection of pairwise disjoint $\delta$-intervals $\mathcal{J}$, which has cardinality comparable to $N(E\cap I,\delta)$. Consider the region
\[
R=\left\{(x',x_d)\colon |x'|\leq c\delta/\sigma, x_d=\frac{-1}{\sqrt{2}}(t-r),\text{ for some }t\in \bigcup\limits_{J\in \mathcal{J}} J\right\}
\]
We claim that for any $x\in R$, we have the lower bound
\begin{equation}\label{Assouad test}
\mathcal{A}_E(h_{\delta,I},h_{\delta,I})(x)\gtrsim \sigma^{2d-2}\delta.
\end{equation}
The calculation is analogous to the earlier one in the proof of Proposition \ref{thirdnecessary}; by using the slicing formula, one sees that balls centered in $R$ with radius chosen suitably intersect the support of $h_{\delta,I}$ in a set of measure about $\sigma^{d-1}\delta$ and spheres intersect in a set of area about $\sigma^{d-1}$. Here, choosing the radius suitably means taking the radius to be the closest value in $E$ to the $x_d$ coordinate and then rescaling by $\frac{1}{\sqrt 2}$. \\
We notice that $R$ has volume about $(\delta/\sigma)^{d-1}\delta N(E\cap I,\delta)$. Putting the lower bound \eqref{Assouad test} into \eqref{bound test nec} leads to the inequality
\begin{equation}
    \sigma^{2d-2}\delta\left((\delta/\sigma)^{d-1}\delta N(E\cap I,\delta)\right)^{1/r}\le (\sigma^{d-1}\delta)^{p^{-1}+q^{-1}}
\end{equation}
Recalling that $\sigma=\delta^{\alpha/2}$, the claim follows by comparing the exponents and doing routine algebraic manipulation.
\end{proof}

Putting Propositions \ref{firstnecessary}, \ref{secondnecessary}, and \ref{fourthnecessary} all together we get the following necessary conditions for the boundedness of $\mathcal{A}_E$.

\begin{prop}[Necessary conditions for boundedness of $\mathcal{A}_E$]\label{allnecessaryAE}
Let $E\subset[1,2]$ be a Assouad regular set with $\text{dim}_{M}(E)=\beta$ and $\text{dim}_{A}(E)=\gamma$. If $\mathcal{A}_E:L^p\times L^q\rightarrow L^r$ is bounded, then 
$$\frac{1}{r}\leq \frac{1}{p}+\frac{1}{q}\leq 1+m_{linear}(d,r,\beta,\gamma).$$
where $$m_{linear}(d,r,\beta,\gamma):=\text{min}\left\{\frac{d-1}{d}+\frac{1-\beta}{dr},\frac{\beta(d-1)}{\beta(d-1)+2\gamma}+\frac{(d-\beta)2\gamma-(d-1)\beta}{\beta(d-1)+2\gamma}\frac{1}{r},\frac{d}{r}\right\}.$$
\end{prop}

The example of Proposition \ref{firstnecessary} can also be adapted to give necessary conditions in the biparameter setting.
\begin{prop}
Suppose $E_1,E_2$ have upper Minkowski dimensions $\beta_1,\beta_2$. Define the quantity 
\[
\beta^*=\dim_M(E_1E_2(E_1^2+E_2^2)^{-1/2}).
\]
If $\mathcal{A}^{(2)}_{S^{2d-1},E_1,E_2}:L^p\times L^q\rightarrow L^r$ continuously then we have
\begin{equation}
    \frac{1}{p}+\frac{1}{q}\le\frac{2d-1}{d}+\frac{1-\beta^*}{dr}
\end{equation}
\end{prop}
\begin{proof}
    Let $f=g=\chi_{B(0,C\delta)}$ and denote $h=\mathcal{A}^{(2)}_{\hat{\sigma},E_1,E_2}(f,g)$. Mimicking the single parameter case, we want to find the set of $x$ where we have the lower bound $|h(x)|\gtrsim \delta^{2d-1}$. Since we now have parameters $t_1,t_2$, the region where this lower bound holds contains all $x$ such that there exist $t_1\in E_1,t_2\in E_2$ and  some $|y|\in (0,1)$ satisfying
    \[
    |x|=t_1|y|=t_2\sqrt{1-|y|^2}.
    \]
    Really, the first equation just needs to hold up to $\delta$ error so that we can compute the $\delta$-covering number and take a limit. Substituting the second equation necessitates that $|y|$ solves the quadratic equation
    \[
    \left(\frac{t_1^2+t_2^2}{t_2^2}\right)|y|^2=1
    \]
    and so we can write the solution as
    \[
    |x|=t_1|y|=\frac{t_1t_2}{\sqrt{t_1^2+t_2^2}}
    \]
    Thus, the values of $|x|$ that are admissible is a $C\delta$-neighborhood of points of the form
    \[
    E^*:=\left\{ \frac{t_1t_2}{\sqrt{t_1^2+t_2^2}}:t_1\in E_1,t_2\in E_2 \right\}
    \]
    Call $\beta^*=\dim_{M}(E^*)$, which depends on $E_1,E_2$. Testing this against the boundedness condition gives
    \begin{equation}\label{multiparameter ex with N delta}
    \delta^{2d-1}\delta^{1/r}N(E^*,\delta)^{1/r}\le\delta^{d/p}\delta^{d/q}
    \end{equation}
    which leads to the inequality
    \[
    \frac{d}{p}+\frac{d}{q}\le 2d-1+\frac{1-\beta^*+\epsilon}{r}
    \]
    We can send $\epsilon$ to 0 in the limit, so dividing through by $d$ gives the claim.
\end{proof}

Similarly, we can adapt Proposition \ref{thirdnecessary} to the biparameter setting.
\begin{prop}
    Suppose $E_1,E_2$ have upper Minkowski dimension $\beta_1,\beta_2$ and let $\beta^*$ be as in the previous proposition. If $\mathcal{A}_E:L^p\times L^q\rightarrow L^r$ continuously then we have
\begin{equation}
    \frac{1}{p}+\frac{1}{q}\le\frac{2d}{d+1}+\frac{1}{r}\left(1-\frac{2\beta^*}{(d+1)}\right)
\end{equation}
\end{prop}
\begin{proof}
Once again, we can define 
 \[
R_1=[-C_1\sqrt{\delta}, C_1\sqrt{\delta}]^{d-1}\times [-C_1\delta, C_1\delta],\quad  R_2=[-C_2\sqrt{\delta}, C_2\sqrt{\delta}]^{d-1}\times [-C_2\delta, C_2\delta],
\]
\[R_3=[-\sqrt{\delta},\sqrt{\delta}]^{d-1}\times \left(\bigcup_{i=1}^{N(E^*,\delta)}I_i\right)
\]
where we may take $C_1, C_2=100$, $\delta$ is sufficiently small, and $\{I_i\}$ is a minimal collection of $\delta$-intervals covering the set $E^*$. Then for $x\in R_{3}$, one can check that  
$$\mathcal{A}_E(f_{\delta},g_{\delta})(x)\gtrsim \delta^{d}$$
That combined with $|R_3|\sim \delta^{\frac{d-1}{2}}N(E^*,\delta)\delta$ is enough to finish the computation.
\end{proof}

\begin{rem}
    In fact, Proposition \ref{fourthnecessary} also generalizes to the biparameter setting, with $\beta,\gamma$ being replaced by $\beta^*$ and $\gamma^*$ respectively, the Minkowski and Assouad dimensions of the set $E^*$. The details are very similar to the other biparameter arguments we have given, so we leave them to the interested reader.
\end{rem}
In fact, we can say more about $\beta^*$ by making use of Lemma \ref{lem:alg of mink dim}.
The dimension of $E^*$ is equal to the dimension of $\frac{E_1^2E_2^2}{E_1^2+E_2^2}$ which can be rewritten as $(E_1^{-2}+E_2^{-2})^{-1}$. Notice that if we freeze one of the sets $E_1$ or $E_2$ and vary the $t_i$s in the other set, we immediately see from Lemma \ref{lem:alg of mink dim} that $\beta^*\ge\max\{\beta_1,\beta_2\}$ since dimension is preserved under taking reciprocals, squaring, and translation by a constant. This lower bound can be achieved, such as in the case that $\beta_1=1$. On the other hand, given two sets $S_1,S_2$, their sumset satisfies the dimension bound
\[
\dim_{M}(S_1+S_2)\le \dim_{M}(S_1)+\dim_{M}(S_2).
\]
Thus, the description
\[
\max\{\beta_1,\beta_2\}\le \beta^*\le \beta_1+\beta_2
\]
is the best we can say for an arbitrary pair of sets $E_1,E_2\subset [1,2]$.

\bibliographystyle{alpha}
\bibliography{sources}

\end{document}